\documentclass{amsart}

\usepackage{amsmath}
\usepackage{amsthm,amsfonts,amssymb,color,url}
\usepackage[dvipsnames]{xcolor}
\usepackage[all]{xy}
\usepackage[colorlinks=true]{hyperref}
\usepackage{tikz}
\usepackage{verbatim}
\usetikzlibrary{arrows}
\def\wt{\widetilde}
\def\wh{\widehat}
\def\:{{\colon}}
\def\lam{{\lambda}}
\def\alp{{\alpha}}
\def\veps{{\varepsilon}}

\newcommand{\ord}{\operatorname{ord}}

\newcommand{\Ram}{\operatorname{Ram}}

\newcommand{\Coker}{\operatorname{Coker}}
\newcommand{\trdeg}{{\mathrm{tr.deg.}}}
\newcommand{\hatotimes}{{\wh{\otimes}}}
\newcommand{\circcirc}{{\circ\circ}}
\newcommand{\di}{{\diamond}}

\newcommand{\rmlog}{{\mathrm log}}

\newcommand{\Br}{{\rm Br}}
\newcommand{\cha}{{\rm char}}
\newcommand{\mon}{{\rm mon}}
\newcommand{\hyp}{{\rm hyp}}
\newcommand{\Int}{{\rm Int}}
\newcommand{\slope}{{\rm slope}}
\newcommand{\Ann}{{\rm Ann}}
\newcommand{\toisom}{\widetilde{\to}}
\newcommand{\into}{{\hookrightarrow}}

\newcommand{\bfF}{{\mathbf{F}}}
\newcommand{\bfA}{{\mathbf{A}}}
\newcommand{\bfP}{{\mathbf{P}}}
\newcommand{\bfR}{{\mathbf{R}}}
\newcommand{\bfQ}{{\mathbf{Q}}}
\newcommand{\bfZ}{{\mathbf{Z}}}
\newcommand{\bfN}{{\mathbf{N}}}
\newcommand{\bfC}{{\mathbf{C}}}

\newcommand{\gtr}{{\mathfrak{r}}}

\newcommand{\calO}{{\mathcal{O}}}
\newcommand{\calC}{{\mathcal{C}}}
\newcommand{\calM}{{\mathcal{M}}}
\newcommand{\calH}{{\mathcal H}}
\newcommand{\calA}{{\mathcal A}}
\newcommand{\calE}{{\mathcal E}}
\newcommand{\calF}{{\mathcal F}}
\newcommand{\calG}{{\mathcal G}}

\newcommand{\gtX}{{\mathfrak{X}}}
\newcommand{\gtY}{{\mathfrak{Y}}}
\newcommand{\gty}{{\mathfrak{y}}}

\newcommand{\tilt}{{\wt t}}
\newcommand{\tilx}{{\wt x}}
\newcommand{\tily}{{\wt y}}
\newcommand{\tilf}{{\wt f}}
\newcommand{\tilk}{{\wt k}}
\newcommand{\tilK}{{\wt K}}
\newcommand{\tilE}{{\wt E}}

\newcommand{\oy}{{\overline{y}}}
\newcommand{\oV}{{\overline{V}}}
\newcommand{\oU}{{\overline{U}}}

\newcommand{\tilcalE}{{\mathcal{\widetilde E}}}
\newcommand{\tilcalF}{{\mathcal{\widetilde F}}}
\newcommand{\tilcalG}{{\mathcal{\widetilde G}}}
\newcommand{\wHx}{{\wt{\calH(x)}}}
\newcommand{\wHy}{{\wt{\calH(y)}}}
\newcommand{\calOcirc}{\calO^\circ}
\newcommand{\calMcirc}{\calM^\circ}

\newcommand{\hatOmega}{{\wh\Omega}}
\newcommand{\hatl}{{\wh{l}}}

\newcommand{\kcirc}{{k^\circ}}
\newcommand{\Kcirc}{{K^\circ}}
\newcommand{\lcirc}{{l^\circ}}
\newcommand{\Lcirc}{{L^\circ}}
\newcommand{\Fcirc}{{F^\circ}}

\newcommand{\kcirccirc}{{k^{\circ\circ}}}
\newcommand{\lcirccirc}{{l^{\circ\circ}}}
\newcommand{\Kcirccirc}{{K^{\circ\circ}}}
\newcommand{\Lcirccirc}{{L^{\circ\circ}}}
\newcommand{\Fcirccirc}{{F^{\circ\circ}}}

\newcommand{\val}{\operatorname{val}}

\CompileMatrices

{
   \newtheorem{theorem}[subsubsection]{Theorem}
      \newtheorem*{theorem*}{Theorem}

   \newtheorem{lemma}[subsubsection]{Lemma}

   \newtheorem{corollary}[subsubsection]{Corollary}
   
   \newtheorem*{conjecture*}{Conjecture}
   
}
{\theoremstyle{definition}
          \newtheorem*{exercise*}{Exercise}
   
   \newtheorem{example}[subsubsection]{Example}
   \newtheorem*{example*}{Example}
   \newtheorem{definition}[subsubsection]{Definition}
   
   \newtheorem*{definition*}{Definition}
   \newtheorem{rem}[subsubsection]{Remark}
   \newtheorem{remark}[subsubsection]{Remark}

}

\begin{document}

\author{Adina Cohen}
\address{Einstein Institute of Mathematics, The Hebrew University of Jerusalem, Giv'at Ram, Jerusalem, 91904, Israel}
\email{adina.cohen@mail.huji.ac.il}

\author{Michael Temkin}
\address{Einstein Institute of Mathematics, The Hebrew University of Jerusalem, Giv'at Ram, Jerusalem, 91904, Israel}
\email{temkin@math.huji.ac.il}

\author{Dmitri Trushin}
\address{Einstein Institute of Mathematics, The Hebrew University of Jerusalem, Giv'at Ram, Jerusalem, 91904, Israel}
\email{trushindima@yandex.ru}

\thanks{This work was supported by the European Union Seventh Framework Programme (FP7/2007-2013) under grant agreement 268182 and BSF grant 2010255.}
\title{Morphisms of Berkovich curves and the different function}

\begin{abstract}
Given a generically \'etale morphism $f\:Y\to X$ of quasi-smooth Berkovich curves, we define a different function $\delta_f\:Y\to[0,1]$ that measures the wildness of the topological ramification locus of $f$. This provides a new invariant for studying $f$, which cannot be obtained by the usual reduction techniques. We prove that $\delta_f$ is a piecewise monomial function satisfying a balancing condition at type 2 points analogous to the classical Riemann-Hurwitz formula, and show that $\delta_f$ can be used to explicitly construct the simultaneous skeletons of $X$ and $Y$. As another application, we use our results to completely describe the topological ramification locus of $f$ when its degree equals to the residue characteristic $p$.
\end{abstract}
\maketitle

\section{Introduction}

\subsection{Motivation}
Throughout this paper, $k$ denotes an algebraically closed complete non-archimedean real-valued field whose valuation will be denoted $|\ |\:k\to\bfR_{\ge 0}$. By a nice compact Berkovich curve we mean a compact separated quasi-smooth strictly $k$-analytic curve. Such objects play a central role in a variety of recent papers (e.g., \cite{metrics_on_curves}, \cite{ABBR}, \cite{XFaber1}, \cite{Baldassarri}, \cite{Poineau-Pulita}), and their structure is adequately described by the semistable reduction theorem. Nevertheless, morphisms between nice curves are not understood so well, and the main aim of this paper is to start filling in this gap. Since the case when $f\:Y\to X$ is tame is classical, we study the phenomena occurring in the wild case. For this, we introduce a {\em different function} $\delta_f\:Y\to[0,1]$ that measures the ``wildness'' of $f$, and this paper is devoted to a detailed study of $\delta_f$ and the properties of $f$ reflected by $\delta_f$.

In particular, we will see that $\delta_f$ is tightly related to the minimal simultaneous semistable reduction of $Y$ and $X$, and if the degree of $f$ equals to $\cha(\tilk)$ then its topological ramification locus and metric structure are completely encoded in $\delta_f$. In a sequel work \cite{radialization}, we will show that in the general case the latter are completely controlled by a more complicated invariant $\phi_f$, which can be viewed as a family of Herbrand functions and associates to points of $Y$ piecewise monomial automorphisms of $[0,1]$, and the different $\delta_f(y)$ is just the coefficient of the linear part of $\phi_f(y)$. Note, however, that our study of $\delta_f$ in this work is much more detailed than the study of $\phi_f$ in \cite{radialization}, and many results, including the genus formulas are not extended to $\phi_f$. So, this work and \cite{radialization} are rather complementary.

\subsection{Known results}
Before outlining our methods and results, let us discuss the state of the art in the field.

\subsubsection{The tame case}\label{intrtame}
Berkovich introduced tame \'etale coverings in \cite[Section 6.3]{berihes} and showed that any connected tame \'etale covering of a disc is trivial and any connected tame \'etale covering of an annulus is Kummer (\cite{berihes} uses the word ``standard''). As a corollary, one can easily obtain the following description of an arbitrary tame morphism $f$: there is a compatible pair of skeletons $\Gamma_X\subset X$ and $\Gamma_Y=f^{-1}(\Gamma_X)$ such that $f$ totally splits on their complements. In particular, the topological ramification locus is a finite graph. Moreover, it suffices to choose $\Gamma_X$ that contains the image $f(\Ram(f))$ of the ramification locus of $f$, since its preimage is automatically a skeleton. The latter observation can be used to give a simple proof of the semistable reduction theorem for curves $Y$ that admit a morphism $Y\to\bfP^1_k$ without wild topological ramification (e.g., when $\cha(\tilk)=0$).

\subsubsection{The wild case}
The situation with wild morphisms is much more complicated. By the simultaneous semistable reduction theorem, see \ref{simulsec}, one can find skeletons $\Gamma_X$ and $\Gamma_Y=f^{-1}(\Gamma_X)$ such that the restriction of $f$ onto their complements is a disjoint union of \'etale coverings of open discs by open discs. However, these coverings do not have to split and may be pretty complicated, so the description of $f$ provided by this theorem is not really satisfactory. In addition, it is not clear how $(\Gamma_Y,\Gamma_X)$ is related to $f$ even when $X=\bfP^1_k$. In the tame case, we can simply take $\Gamma_X$ to be the convex hull of $f(\Ram(f))$, but in the wild case the latter is not so informative (e.g., it can be a single point when $\cha(k)>0$).

Furthermore, if $k$ is of mixed characteristic $p$, for example $k=\bfC_p$, then already for the wild Kummer covering $\bfP^1_k\to\bfP^1_k$ given by $t\mapsto t^p$, the topological ramification locus $T$ is a metric neighborhood of the interval $[0,\infty]\subset\bfP^1_k$. Although $T$ is a huge set, it possesses a reasonable ``finite combinatorial description'', so it is natural to wonder if the topological ramification locus can be described ``combinatorially'' in general. To the best of our knowledge, this question was only studied in the works~\cite{XFaber1} and~\cite{XFaber2} of X. Faber. In particular, Faber managed to bound from above the topological ramification locus of morphisms $\bfP^1_k\to\bfP^1_k$ when $\Ram(f)$ contains no wildly ramified points (e.g., $\cha(k)=0$): it is contained in a certain metric neighborhood of the convex hull of $f(\Ram(f))$. In addition, Faber showed that no such metric neighborhood exists if there are wildly ramified points. It was not even conjectured in the literature what a precise structure of the topological ramification locus in general might be (see Section~\ref{sequel} below).

\subsubsection{The different}
The different is a classical invariant that measures wildness of a valued field extension, so it is quit natural to consider it when studying wild covers $f\:Y\to X$ of Berkovich curves. Nevertheless, it seems that the different was not used in the literature devoted to Berkovich spaces, although it did show up in the adjacent areas of rigid and, especially, formal geometries. In rigid geometry, L\"utkebohmert, following Gabber's ideas, used the different to prove a rather deep non-archimedean version of Riemann's existence theorem, see \cite{Lutkebohmert}. In fact, L\"utkebohmert implicitly introduced the different function on certain intervals in $Y$, showed that it is piecewise monomial on them, and obtained certain estimates on $\delta_f$, specific for the mixed characteristic case. Later, Ramero gave in \cite{Ramero} another proof of Riemann's existence theorem, which also makes use of the different.

In formal geometry, the different was used in a whole cluster of works related to lifting problems, automorphisms of open discs and Oort's conjecture. For example, see \cite{Raynaud}, \cite{Green-Matignon}, \cite{Obus-Wewers} and the literature cited there. The motivation and context in these papers differs from ours. Typically one assumes that the covering is Galois and the characteristic is mixed and studies the covers in a much more detailed way. Often, one also restricts the Galois group, for example, assuming that it is cyclic or even of degree $p$. Finally, the ground field is assumed to be discretely valued.

\subsection{Main results and outline of the paper}

\subsubsection{Preliminaries}
In Section \ref{onedimsec} we study the different of extensions of one-dimensional analytic $k$-fields, i.e. fields that can appear as $\calH(x)$ where $x$ is a point of a curve. To large extent this is based on \cite[Section 6]{temst} and \cite[Chapter 6]{Gabber-Ramero}. For fields, our main formula for computing the differents is established in Corollary~\ref{differentcor}.

Section~\ref{ancurvesec} is devoted to systematization of various material about Berkovich curves we use. Most of it is well known, although some statements are hard to find in the literature.

\subsubsection{Local behaviour of $\delta_f$}
The main player of this paper is introduced in Section~\ref{difsec}: we associate to $f\:Y\to X$ a {\em different function} $\delta_f\:Y\setminus Y(k)\to[0,1)$ whose value at $y$ equals to the different of the extension $\calH(y)/\calH(x)$, where $x=f(y)$. Thus, $\delta_f$ reflects how the different varies in one-dimensional analytic families of extensions of valued fields. In Theorem~\ref{compdeltalem} we compute $\delta_f$ by use of tame parameters, and obtain, as a corollary, that $\delta_f$ is piecewise $|k^\times|$-monomial on any interval $I\subset Y\setminus Y(k)$. This extends L\"utkebohmert's (implicit) results to the case when $I$ has a type 4 endpoint. In addition, we describe in Theorem~\ref{restrictth} all restrictions satisfied by the multiplicity of $f$ at a point $y$, the value of the different at $y$ and the slope of the different in some direction from $y$. As a very particular case, this recovers the classical fact that for a fixed multiplicity, the different is bounded in the mixed characteristic case, unlike the equicharacteristic one.

The major part of Section~\ref{difsec} is occupied with the study of local behaviour of $\delta_f$ at a type 2 point $y$. In Theorem~\ref{prop:local_RH} we show that if $y$ is inner then the slopes of the different along all branches $v$ at $y$ satisfy a balancing condition analogous to the Riemann-Hurwitz formula, where the role of the classical differential ramification indices $R_v$ (e.g., see \cite[IV.2.4]{Hartshorne}) is played by the numbers $S_v=-\slope_v\delta_f+n_v-1$. In particular, we show that almost all $S_v$ vanish, and hence $\delta_f$ increases in almost all directions from $y$ whenever $\delta_f(y)<1$. This indicates that $\delta_f$ is very different from functions of the form $|h|$ for $h\in\Gamma(\calO_Y)$, and is somewhat analogous to $1/r(y)$, where $r(y)$ is a radius function on a disc.

\begin{remark}
(i) We often call this balancing condition the {\em local Riemann-Hurwitz formula} at a type 2 point. The formula is new, though it should be noted that in the situation studied by A. Obus in \cite{Obus} (the ground field is discretely valued, the characteristic is mixed and the covering is Galois with cyclic $p$-Sylow subgroups), one can easily deduce it from \cite[5.6,5.8,5.10]{Obus}, though translation of notation requires some effort.

(ii) The balancing condition should not be confused with local Riemann-Hurwitz formulas at a formal fiber of a closed point of a formal model. The latter type of formulas compute the genus of such a formal fiber, see, for example, \cite[Theorem~3.4]{Saidi}. We establish a formula of this type in Theorem~\ref{genusth2}.

(iii) Although the branches at $y$ correspond to a reduction curve $C_y$ with function field $\wHy$ (see \ref{germredsec}), the numbers $S_v$ cannot be described in terms of any reduction data and our balancing condition does not reduce to a Riemann-Hurwitz formula for curves over $\tilk$. For example, it can freely happen that the extension $\wHy/\wHx$ is purely inseparable, and hence does not provide any new information about $C_y$, while not all $S_v$ vanish, and hence $\delta_f$ distinguishes a few branches at $y$.

(iv) Our main formula for $\delta_f$ (Theorem~\ref{compdeltalem}) involves radii functions $r_t$, so it is not so surprising that $\delta_f$ and $r_t^{-1}$ behave similarly. Also, the different is related to the norms on the sheaves $\Omega_X$ and $\Omega_Y$ (see Section~\ref{kahlersec}), which quite differ from the norm on $\calO_X$. In particular, if $\omega\in\Gamma(\Omega_X)$ then $|\omega|_\Omega$ decreases in almost all directions.
\end{remark}

Finally, in Section \ref{limsec} we extend $\delta_f$ to a function $\delta_f^\rmlog\:Y\to[0,1]$ by continuity and show that it is piecewise $|k^\times|$-monomial with zeros at wild ramification points. Also, we show that the order of the zero at $y\in Y(k)$ is the value of the logarithmic different of the extension of DVRs $\calO_y/\calO_{x}$, see Theorem~\ref{limth}.

\subsubsection{Applications to the structure of $f$}
There is a standard graph-theoretic language describing skeletons of nice compact curves and tame morphisms between them. In Section~\ref{combsec} we extend it by adding in a datum related to the different function. Then we prove a combinatorial Riemann-Hurwitz formula for maps of such graphs.

In Section \ref{mainsec} we study what $\delta_f$ can tell about $f$. In particular, Theorem \ref{genusth} expresses the genus of $Y$ in terms of the genus of $X$, the ramification divisor, and indices $R_b$ at the boundary points of $Y$. By Theorem~\ref{prop:local_RH}, the local Riemann-Hurwitz formula can fail only at a boundary type 2 point $b$, and $R_b$ measures its failure. Again, the essentially new feature here are the indices $R_b$, that cannot be defined without the different function (e.g., in terms of geometry over $\tilk$). In addition, we describe the global structure of $\delta_f$ and its relation to the skeletons. As we saw, $S_v=0$ for almost any branch $v\in\Br(y)$, i.e. the different $\delta_f$ increases with slope $n_v-1$ in the direction of $v$. So, we say that $\delta_f$ is trivialized by a skeleton $\Gamma_Y$ if for any branch $v$ not pointing towards $\Gamma_Y$, we have that $S_v=0$. If one only has that $S_v=0$ for all points $y\in\Gamma_Y$ of type 2 and all branches $v$ at $y$ pointing outside of $\Gamma_Y$, then we say that $\Gamma_Y$ locally trivializes $\delta_f$. By Theorem~\ref{deltatrivth}, any simultaneous skeleton $(\Gamma_Y,\Gamma_X)$ trivializes $\delta_f$, and, conversely, Theorem~\ref{localcharth} states that if $\Gamma$ is the preimage of a skeleton $\Gamma_X$ and $\delta_f$ is locally trivialized by $\Gamma$ then $\Gamma$ is a skeleton (in particular, $\delta_f$ is trivialized by it). As a corollary, one obtains a constructive description of the skeletons of $Y$ in terms of $f$ and the skeletons of $X$, see Remark~\ref{skeletonrem}.

\subsubsection{Degree-$p$ coverings}
We describe the structure of morphisms of degree $p$ in Section~\ref{degpsec}. Trivialization of $\delta_f$ by a skeleton $\Gamma_Y$ allows to express $\delta_f$ in terms of its restriction onto $\Gamma_Y$ and the multiplicity function $n_y\:Y\to\bfN$. This does not give a complete description as $n_y$ can be complicated, but the situation improves when $\deg f=p$. In this case, if $f$ is wild at $y$ then $n_y=p$ and we obtain a full control on $\delta_f$ and the topological ramification locus. Namely, if $\Gamma_Y$ trivializes $\delta_f$ then $\delta_f$ increases in all directions pointing outside of $\Gamma_Y$ with constant slope $p-1$. In particular, we obtain in Theorem~\ref{coneth} the following finite combinatorial description of the topological ramification locus $T$ of $f$: it is a radial set around the subgraph $T\cap\Gamma_Y$ of $\Gamma_Y$ whose radius at $y\in T\cap\Gamma_Y$ is $\delta_f(y)^{1/(p-1)}$.

We conclude the paper by illustrating our results in the case of a degree two covering $f\:E\to\bfP_k^1$, where $E$ is an elliptic curve. In particular, we classify all ten possible configurations of the minimal skeleton of $f$, explain their relation to $\delta_f$ and relate their metric to the absolute value of the $j$-invariant, see Theorem~\ref{skelelliptic}. In the wild case (i.e. $\cha(\tilk)=2$) this seems to be new, especially, what concerns the ``tropicalization" of the supersingular configurations, although a brief analysis of the mixed characteristic case can be found in \cite[5.1]{Green-Matignon}.

\subsection{A sequel}\label{sequel}
To complete a ``combinatorial'' description of finite generically \'etale morphisms $f\:Y\to X$ between nice compact curves, one should provide a ``finite'' description of the multiplicity function $n_y$. This will be done in a separate paper \cite{radialization} by proving the following result:

Let $T_d$ be the set of points with $n_y=d$. Then each $T_{\ge d}=\coprod_{i\ge d}T_i$ is a closed set and there exists a skeleton $\Gamma_Y$ and piecewise monomial functions $r_i\:\Gamma_Y\to(0,1]$ with $1\le i\le [\log_p(\deg f)]$, such that each $T_d$ with $d\notin p^\bfN$ is contained in $\Gamma_Y$ and each $T_{\ge p^i}$ is the radial set with center $\Gamma_Y\cap T_{\ge p^i}$ of radius $r_i(y)$. In addition, it will be shown in \cite{radialization} that the radii $r_i(y)$ can be easily obtained from the breaks of the ramification filtration of $\calH(y)/\calH(x)$ and, in fact, they determine the ramification filtration. %Also, it may be the case that in the mixed characteristic case one can relate $r_i$'s to convergence radii of some ODE's (see \cite{Baldassarri} and \cite{Poineau-Pulita}), but this question is not studied in \cite{radialization}.

\setcounter{tocdepth}{1}
\tableofcontents

\section{One-dimensional analytic fields}\label{onedimsec}
In Section \ref{onedimsec} we recall some facts about extensions of {\em analytic} fields, i.e. complete real-valued fields. Recall that the ground field $k$ is assumed to be analytic, non-trivially valued and algebraically closed.

\subsection{Local uniformization and parameters}

\subsubsection{One-dimensional $k$-fields and their types}
An analytic $k$-field $K$ is called {\em one-dimensional} if it is finite over a subfield $\wh{k(t)}$ with $t\notin k$. Recall that the sum of $F_{K/k}=\trdeg_\tilk(\tilK)$ and $E_{K/k}=\dim_\bfQ(|K^\times|/|k^\times|\otimes_\bfZ\bfQ)$ is bounded by 1 by Abhyankar's inequality. We say that $K$ is of type 2 if $F_{K/k}=1$, of type 3 if $E_{K/k}=1$, and of type 4 if $E_{K/k}=F_{K/k}=0$.

\subsubsection{Parameters}
If $K$ is a one-dimensional analytic $k$-field and $t\in K\setminus k$ then the extension $K/\wh{k(t)}$ is finite by \cite[Corollary~6.3.4]{temst}. If $K/\wh{k(t)}$ is separable (resp. tame, resp. unramified) then we say that $t$ is a parameter (resp. tame parameter, resp. unramified parameter). The following theorem is proved in \cite[Theorem 6.3.1(i)]{temst} by a direct (though involved) valuation-theoretic argument. One can view it both as a local uniformization of one-dimensional fields and a far reaching generalization of the separable transcendence basis theorem in dimension one.

\begin{theorem}\label{unramparam}
Any one-dimensional $k$-field possesses an unramified parameter.
\end{theorem}

\begin{remark}
Theorem \ref{unramparam} is an easy consequence of the semistable reduction theorem recalled below. Conversely, using Theorem \ref{unramparam} one can prove the semistable reduction theorem relatively easily, and in the algebraic setting this was done in \cite{temst}.
\end{remark}

\subsubsection{Monomial parameters}
We say that a parameter $t\in K$ is {\em monomial} if the induced valuation on $l=k(t)$ is a generalized Gauss valuation, i.e. $|\sum_i a_it^i|=\max_i|a_i|r^i$, where $r=|t|$. %In such case, we also use the notation $l=k(t)_r$ to indicate the valuation of $l$.

\begin{lemma}\label{monomparamlem}
Assume that $K$ is a one-dimensional analytic $k$-field and $t\in K$ is a parameter.

(i) The infimum $\inf_{c\in k}|t-c|$ is achieved if and only if $K$ is of type 2 or 3.

(ii) The infimum is achieved for $c=c_0$ if and only if $t-c_0$ is a monomial parameter.
\end{lemma}
\begin{proof}
The field $L=\wh{k(t)}$ is of the same type as $K$, so it suffices to work with $L$. In this case, $L=\calH(x)$ for a point $x\in\bfP_k^1$ and the assertion of the lemma reduces to the following obvious claims: $x$ can be moved to a point of the interval $(0,\infty)$ by an appropriate translation of $\bfP_k^1$ if and only if it is of type 2 or 3, and $x\in(0,\infty)$ if and only if the valuation of $\calH(x)$ is a generalized Gauss valuation.
\end{proof}

\subsubsection{Radius}
Given a parameter $t\in K$, the number $r_t=\inf_{c\in k}|t-c|$ will be called the {\em radius} of $t$. Note that $t$ induces a map $\calM(K)\to\bfA^1_k$ and $r_t$ is the radius of its image $x\in\bfA^1_k$, i.e. the infimum of radii of discs containing $x$.

\subsection{Completed differentials}

\subsubsection{K\"ahler seminorm}
As in \cite[4.1.1]{Temkintopforms}, given an extension of real-valued fields $l/k$ we provide the module of differentials $\Omega_{l/k}$ with the maximal $l$-seminorm $|\ |_\Omega=|\ |_{\Omega,l/k}$ such that $d\:l\to\Omega_{l/k}$ is contracting. We call $|\ |_\Omega$ the {\em K\"ahler seminorm}.

\subsubsection{Completed differentials}
The completion with respect to $|\ |_\Omega$ will be denoted $\hatOmega_{l/k}$. For shortness, we call its norm the {\em K\"ahler norm} and denote it $|\ |_\Omega$.

\subsubsection{Unit balls}
The unit balls of $\Omega_{l/k}$ and $\hatOmega_{l/k}$ will be denoted $\Omega^\di_{l/k}$ and $\hatOmega^\di_{l/k}$, respectively. Perhaps, $\Omega^\circ$ would be a better notation than $\Omega^\di$, but in non-archimedean geometry $\circ$ is reserved for the spectral seminorm of algebras. Note that $\hatOmega^\di_{l/k}$ is the $\pi$-adic completion of $\Omega^\di_{l/k}$, where $0\neq\pi\in\lcirccirc$.

\subsubsection{A relation to differentials of rings of integers}
The unit balls considered above are tightly related to the $\lcirc$-module $\Omega_{\lcirc/\kcirc}$ and its $\pi$-adic completion $\hatOmega_{\lcirc/\kcirc}$.

\begin{lemma}\label{compldiflem}
Let $l/k$ be as above and $L=\hatl$, then

(i) The natural map $\hatOmega_{l/k}\to\hatOmega_{L/k}$ is an isometric isomorphism.

(ii) The natural map $\phi\:\Omega_{\lcirc/\kcirc}\to\Omega_{l/k}$ is an embedding and $|\ |_\Omega$ is the maximal seminorm on $\Omega_{l/k}$ such that $\Omega_{\lcirc/\kcirc}\subseteq\Omega^\di_{l/k}$.

(iii) The natural map $\hatOmega_{\Lcirc/\kcirc}\to\hatOmega_{L/k}$ is an embedding and the norm on $\hatOmega_{L/k}$ is the maximal norm such that $\hatOmega_{\Lcirc/\kcirc}\subseteq\hatOmega^\di_{L/k}$.
\end{lemma}
\begin{proof}
Part (i) is proved in \cite[Corollary 5.6.7]{Temkintopforms} . Since $\Omega_{l/k}=\Omega_{\lcirc/\kcirc}\otimes_{\lcirc}l$, and $\Omega_{\lcirc/\kcirc}$ is torsion free by \cite[Theorem 6.5.20(i)]{Gabber-Ramero}, $\phi$ is an embedding. The group $|k^\times|$ is dense in $\bfR_{>0}$, hence the second claim of (ii) follows from \cite[Corollary~5.3.3 and Theorem~5.1.8]{Temkintopforms}. Finally, (iii) is obtained from (i) and (ii) by passing to the completions.
\end{proof}

\subsection{Differentials of one-dimensional fields}

\subsubsection{Main computation}
In the one-dimensional case, K\"ahler norm and related $\Kcirc$-modules can be explicitly computed as follows.

\begin{theorem}\label{compldifth}
Assume that $K$ is a one-dimensional analytic $k$-field and $t\in K$ is a tame parameter, then

(i) $\hatOmega_{K/k}$ is a one-dimensional vector space with basis $dt$, and $|dt|_\Omega=r_t$.

(ii) $\hatOmega_{\Kcirc/\kcirc}$ is the $\Kcirc$-submodule of $\hatOmega_{K/k}$ generated by the elements $\frac{dt}c$, where $c\in k$ is such that there exists $a\in k$ with $|t-a|\le|c|$.
\end{theorem}
\begin{proof}
Set $l=k(t)$ and $F=\wh{l}$. By Lemma~\ref{compldiflem}(i), $\hatOmega_{F/k}$ is one-dimensional with basis $dt$. The extension $K/F$ is finite and separable, hence the map $\Omega_{F/k}\otimes_FK\to\Omega_{K/k}$ is an isomorphism, and we obtain that the map $$\psi_{K/F/k}\:\hatOmega_{F/k}\otimes_FK\to\hatOmega_{K/k}$$ has a dense image. It follows that either $\hatOmega_{K/k}$ vanishes, or it is one-dimensional with basis $dt$, and in the latter case $\psi_{K/F/k}$ is an isomorphism. Since (ii) implies that $|dt|_{\Omega,K/k}=r_t>0$, we see that (ii) implies (i). Furthermore, if $|t-a|\le|c|$ then $\frac{dt}{c}=d\frac{t-a}c\in\hatOmega_{\Kcirc/\kcirc}$. So we should only prove that any $df$ with $f\in\Kcirc$ lies in the module generated by the elements $\frac{dt}c$ as in (ii).

Step 1. {\it The theorem holds true when $K=F$.} First, assume that $K/k$ is of type 3. To prove the theorem we can replace $t$ with any element $t-a$ for $a\in k$. So, we can assume that $K=k\{r^{-1}t,rt^{-1}\}$. Obviously, $\hatOmega_{\Kcirc/\kcirc}$ is generated by the elements $df$ where $f=\sum a_it^i\in\Kcirc$ and $a_0=0$. Since $a_0=0$ we actually have that $|f|<1$, hence $df=f_tdt=\sum ia_it^{i-1}dt$ and $|f_t|<|t^{-1}|$. In particular, taking $c\in k$ with $|f_t|<|c|^{-1}<|t^{-1}|$ we achieve that $df\in\lcirc\frac{dt}c$ and $|t|<|c|$.

Assume, now, that the type is 2 or 4. Let $C$ denote the set of all elements $c\in l$ such that $|t-a|\le |c|$ for some $a\in k$. Recall that $\Omega_{\lcirc/\kcirc}$ embeds into $\Omega_{l/k}=ldt$ by Lemma~\ref{compldiflem}(ii). We claim that $\Omega_{\lcirc/\kcirc}$ is $\lcirc$-generated by the elements $\frac{dt}c$ with $c\in C$. Take any element $f\in\lcirc$. Then $f=a\prod_i(t-a_i)^{n_i}$ and hence $df=\sum_i n_i(t-a_i)^{-1}fdt$. Note that for any $i$ there exists $c_i\in C$ such that $|c_i|\le|t-a_i|$. Indeed, if no such $c_i$ exists then $|t-a_i|=r_t\notin|k^\times|$, and so $K$ is of type 3. Thus, we can choose $c_i$ as above and then $|n_i(t-a_i)^{-1}f|\le|c_i|^{-1}$. In particular, $df\in\sum_i\lcirc\frac{dt}{c_i}$ and we have proved an analogue of the theorem for the extension $l/k$. Passing to the completions we obtain the assertion of Step 1.

Step 2. {\em The general case}. We start with the following result.

\begin{lemma}\label{tamelem}
If $K/F$ is a tamely ramified algebraic extension of real-valued fields and the valuation on $F$ is not discrete then $\Omega_{\Kcirc/\Fcirc}=0$.
\end{lemma}
\begin{proof}
We have that $\Omega_{\Kcirc/\Fcirc}=\Kcirccirc/\Fcirccirc\Kcirc$ by \cite[Lemma~5.2.7]{Temkintopforms}, and it remains to note that $\Kcirccirc/\Fcirccirc\Kcirc=0$ since the valuation on $F$ is not discrete.
\end{proof}

Returning to the proof of Theorem~\ref{compldifth}, consider the map $$\phi\:\Omega_{\Fcirc/\kcirc}\otimes_{\Fcirc}\Kcirc\to\Omega_{\Kcirc/\kcirc}.$$ It has zero kernel by \cite[Theorem 6.3.23]{Gabber-Ramero} and $\Coker(\phi)=\Omega_{\Kcirc/\Fcirc}=0$ by Lemma~\ref{tamelem}. So, $\phi$ is an isomorphism and hence its completion $$\psi_{\Kcirc/\Fcirc/\kcirc}\:\hatOmega_{\Fcirc/\kcirc}\hatotimes_{\Fcirc}\Kcirc\to\hatOmega_{\Kcirc/\kcirc}$$ is an isomorphism. Since the theorem holds for $F$, we obtain that the assertion (ii) holds for $K$. The assertion (i) follows.
\end{proof}

The assertion of Theorem \ref{compldifth} becomes especially convenient for applications when $t$ is a tame monomial parameter, and so $|t|=r_t$. Let us make the assertion of the theorem more explicit in this case. We will use notation $K^\circ_s=\{x\in K|\ |x|\le s\}$ and $K^\circcirc_s=\{x\in K|\ |x|<s\}$.

\begin{corollary}\label{compldifcor}
Assume that $K$ is a one-dimensional analytic $k$-field, $t\in K$ is a tame parameter and $s=r_t^{-1}$. Then,

(i) $\hatOmega_{K/k}^\di=K^\circ_sdt$. In particular, if $t$ is a tame monomial parameter then $\frac{dt}t$ is a basis of $\hatOmega_{K/k}^\di$.

(ii) $\hatOmega_{\Kcirc/\kcirc}=K^\circ_sdt$ if $K$ is of type 2, and $\hatOmega_{\Kcirc/\kcirc}=K^\circcirc_sdt$ if $K$ is of type 3 or 4. In particular, $\hatOmega_{\Kcirc/\kcirc}$ is a free module if and only if $K$ is of type 2, and in this case for any tame monomial parameter $x$ with $|x|=1$ we have that $\hatOmega_{\Kcirc/\kcirc}=\Kcirc dx$.
\end{corollary}

\subsubsection{Quasi-invertible modules}
If $K$ is a real-valued field then we call a $\Kcirc$-module $M$ {\em quasi-invertible} if it is torsion free, $M_K=M\otimes_{\Kcirc}K$ is of dimension one, and $0\subsetneq M\subsetneq M_K$. In fact, any quasi-invertible module is of the form $K^\circ_s$ or $K^\circcirc_s$. Note that the notion ``almost invertible" introduced by Gabber and Ramero is more general since an almost invertible module may contain torsion.

Given a real-valued field $l$ with an $l$-vector space $V$ and two $\lcirc$-submodules $M,N\subset V$, we define the ratio $(M:_lN)$ to be the set of all elements $x\in l$ such that $xN\subseteq M$. It is a fractional ideal of $\lcirc$. The absolute value of the ratio is $$|M:_lN|=\sup_{x\in(M:_lN)}|x|.$$ In the rank-one case, this absolute value is multiplicative, namely we have the following obvious result.

\begin{lemma}\label{almsostlem}
Assume that $l$ is a real-valued field, $V$ is a one-dimensional $l$-vector space, and $M,P,Q\subset V$ are quasi-invertible submodules. Then $|M:_lP|\cdot|P:_lQ|=|M:_lQ|$.
\end{lemma}

\subsubsection{Relative differentials}
Using absolute differentials, we can also describe differentials of extensions of one-dimensional fields.

\begin{lemma}\label{reldiflem}
Assume that $L/K$ is a finite separable extension of one-dimensional fields. If $L$ is of type 2 then $\Omega_{\Lcirc/\Kcirc}$ is of the form $L^\circ_r/L^\circ_s$, so it is a finitely presented cyclic module, and if $L$ is of type 3 or 4 then $\Omega_{\Lcirc/\Kcirc}$ is of the form $L^\circcirc_r/L^\circcirc_s$.
\end{lemma}
\begin{proof}
Since $\Omega_{\Lcirc/\Kcirc}\otimes_{\Lcirc}L=\Omega_{L/K}=0$, the module $\Omega_{\Lcirc/\Kcirc}$ is torsion. Furthermore, it is annihilated by a single element $0\neq\pi\in\Kcirc$ by \cite[6.3.8 and 6.3.23]{Gabber-Ramero}. (An alternative and more elementary argument is to use the first fundamental sequence to reduce to the case when $L/K$ is an elementary extension as in \cite[6.3.13]{Gabber-Ramero}, and in the latter case one can bound the torsion of $\Omega_{\Lcirc/\Kcirc}$ by the same computations as used in the proof of loc.cit.) It follows that $\Omega_{\Lcirc/\Kcirc}=\hatOmega_{\Lcirc/\Kcirc}$, so completing the first fundamental sequence of $\kcirc\into\Kcirc\into\Lcirc$ we obtain that $\Omega_{\Lcirc/\Kcirc}$ is the cokernel of the map $$\psi_{\Lcirc/\Kcirc/\kcirc}\:\hatOmega_{\Kcirc/\kcirc}\wh{\otimes}_{\Kcirc}\Lcirc\to\hatOmega_{\Lcirc/\kcirc}.$$ Since both the target and the image of $\psi_{\Lcirc/\Kcirc/\kcirc}$ are quasi-invertible modules described by Corollary~\ref{compldifcor}(ii), the lemma follows.
\end{proof}

\subsection{Different for one-dimensional fields}

\subsubsection{The definition}
Given a separable extension of one-dimensional $k$-fields $L/K$ we define its {\em different} as the absolute value of the annihilator of $\Omega_{\Lcirc/\Kcirc}$, i.e. $$\delta_{L/K}=|\Ann(\Omega_{\Lcirc/\Kcirc})|,$$ where we set $|I|=\sup_{x\in I}|x|$ for an ideal $I\subseteq\Lcirc$. Note that $\delta_{L/K}=s/r$, where $r$ and $s$ are as in Lemma~\ref{reldiflem}.

\begin{remark}
(i) If the valuations are not discrete, the different of a tamely ramified extension equals to 1 by Lemma~\ref{tamelem}. In a sense, the different measures ``wildness'' of extensions, though it may be equal to 1 for wildly ramified extensions (these are so-called almost unramified extensions).

(ii) Usually one defines the different as a fractional ideal of $\Kcirc$, but we are only interested in its absolute value.

(iii) Our definition with the annihilator is an analogue of the classical definition that concerns discrete valuation fields with perfect residue fields. However, it is meaningful only because $\Omega_{\Lcirc/\kcirc}$ is quasi-invertible and hence its quotient $\Omega_{\Lcirc/\Kcirc}$ is a ``rank one'' torsion module. In particular, we will show below that the different we define is multiplicative, and Lemma~\ref{almsostlem} will be used in the argument.

(iv) For general extensions of valued fields one can define the different by use of an analogue of the zeroth Fitting ideal of $\Omega_{\Lcirc/\Kcirc}$, see \cite{Gabber-Ramero}.
\end{remark}

\subsubsection{Comparison of K\"ahler norms}
The different of the extension appears as the scaling factor when comparing the K\"ahler norms of a field and its extension.

\begin{theorem}\label{comparth}
Assume that $L/K$ is a separable extension of one-dimensional fields and consider the isomorphism $\psi\:\hatOmega_{K/k}\otimes_KL\to\hatOmega_{L/k}$. Then $|x'|_{\Omega,L}=\delta_{L/K}|x|_{\Omega,K}$ for any $x\in\hatOmega_{K/k}$ with $x'=\psi(x\otimes 1)$.
\end{theorem}
\begin{proof}
By Lemma \ref{compldiflem}(iii), $$|x'|_{\Omega,L}=|x'\Lcirc:_L\hatOmega_{\Lcirc/\kcirc}|$$ and similarly
$$|x|_{\Omega,K}=|x\Kcirc:_K\hatOmega_{\Kcirc/\kcirc}|=|x'\Lcirc:_L\psi(\hatOmega_{\Kcirc/\kcirc}\otimes_{\Kcirc}\Lcirc)|.$$
It remains to use Lemma~\ref{almsostlem} and the fact that $$\delta_{L/K}=|\psi(\hatOmega_{\Kcirc/\kcirc}\otimes_{\Kcirc}\Lcirc):_L\hatOmega_{\Lcirc/\kcirc}|.$$
\end{proof}

\begin{corollary}
If $L/K/F$ is a tower of finite separable extensions of one-dimensional analytic $k$-fields then $\delta_{L/F}=\delta_{L/K}\delta_{K/F}$.
\end{corollary}

As another corollary, we obtain a convenient way to compute the differents.

\begin{corollary}\label{differentcor}
Assume that $L/K$ is a finite separable extension of one-dimensional $k$-fields and $t\in L$, $x\in K$ are parameters. Then,

(i) $\delta_{L/K}=\left|\frac{dx}{dt}\right|\frac{|dt|_{\Omega,L}}{|dx|_{\Omega,K}}$.

(ii) If the parameters are tame then $\delta_{L/K}=\left|\frac{dx}{dt}\right|\frac{r_t}{r_x}$. %In particular, if the parameters are tame and monomial then $\delta^{-1}_{L/K}=\left|\frac{dx}{dt}\right|\frac{|t|}{|x|}$.
\end{corollary}
\begin{proof}
By definition, $\frac{dx}{dt}$ is the element $h\in L$ such that $hdt=dx$. By Theorem~\ref{comparth}, $$|dt|_{\Omega,L}=|h^{-1}dx|_{\Omega,L}=|h^{-1}|\cdot|dx|_{\Omega,L}=\delta_{L/K}|h^{-1}|\cdot|dx|_{\Omega,K}$$
and we obtain (i). The second claim follows from (i) and Theorem~\ref{compldifth}(i).
\end{proof}

\section{Analytic curves}\label{ancurvesec}
In Section~\ref{ancurvesec} we provide a brief review of some basic facts about $k$-analytic spaces, making a special accent on curves. Concerning facts about curves, proofs can be found in \cite[Section 4.1]{berbook}, \cite{curvesbook} and the literature cited there.

\subsection{$G$-topology}\label{Gsec}
First, we recall the terminology of \cite{Temkintopforms} related to the $G$-topology.

\subsubsection{Choice of the $G$-topology}
Although a separated analytic curve $X$ is good and hence the sheaf $\calO_X$ is reasonable, we prefer to work with coherent sheaves in the $G$-topology. Since we consider only strictly analytic curves, $X_G$ denotes the $G$-topology of strictly $k$-analytic domains throughout the paper.

\subsubsection{The space $X_G$}
For any strictly $k$-analytic space $X$, the topos of sheaves of sets $X_G^\sim$ on the $G$-topological space $X_G$ is coherent and hence has enough points by the famous theorem of Deligne (see \cite[VI.9.0]{sga42}). It follows easily, see \cite[Theorem~9.1.6]{Temkintopforms}, that $X_G^\sim$ is equivalent to the topos of a natural topologization $|X_G|$ of the set of points of $X_G^\sim$. In particular, it is safe to replace $X_G$ with the larger space $|X_G|$ when working with sheaves. From now on, we will use the notation $X_G$ to denote the honest topological space $|X_G|$, adopting the convention of \cite{Temkintopforms}.

\begin{remark}
(i) Any $G$-sheaf on $X$ extends uniquely to $X_G$, so we can (and will) view any $G$-sheaf on $X$ as a sheaf on $X_G$, justifying the notation. The main profit of working with $X_G$ is that we can consider {\em non-analytic points}, i.e. points of $X_G\setminus X$, and stalks at these points as an integral part of the picture.

(ii) On the set-theoretical level, we identify $X$ with a subset of $X_G$. In fact, the $G$-topology of $X$ is induced from the topology of $X_G$ in the following sense: $G$-open sets are restrictions of open sets of $X_G$ and a covering $U=\cup_i U_i$ is a $G$-covering if there exist an open covering $V=\cup_i V_i$ in $X_G$ such that $V\cap X=U$ and $V_i\cap X=U_i$.
\end{remark}

\subsubsection{Residue fields}
Given a point $x\in X_G$, by $\kappa_G(x)$ we denote the residue field of $\calO_{X_G,x}$. The spectral seminorms on affinoid subspaces containing $x$ induce a norm $|\ |_x$ on $\calO_{X_G,x}$ via the rule $|f|_x=\inf\rho_{\calA}(f)$, where $\rho_\calA$ denotes the spectral seminorm of $\calA$ and the infimum is over all affinoid domains $V=\calM(\calA)$ such that $x\in V$ and $f$ is defined on $V$. It is easy to see that $|\ |_x$ is a real semivaluation with kernel equal to the maximal ideal $m_{G,x}$. In particular, giving $|\ |_x$ is equivalent to giving a real valuation on $\kappa_G(x)$ that will also be denoted $|\ |_x$.

\subsubsection{Completed residue field}
Let $\calH(x)$ be the completion of $(\kappa_G(x),|\ |_x)$. This notation is compatible with the classical notation when $x\in X$. Indeed, $\calH(x)$ is preserved when passing to an analytic subdomain containing $x$, and if $X$ is good then $\kappa(x)\subseteq\kappa_G(x)$ is dense. Thus, $\calH(x)$ is the completion of both fields.

\subsubsection{Adic interpretation}
In the framework of rigid geometry, van der Put studied the points of $X_G$ in \cite{Put@compositio}. In fact, he used the language of prime filters of analytic domains, but this is equivalent to the topos-theoretic definition. In particular, van der Put showed that $\calOcirc_{X_G,x}$ is the preimage of a valuation subring of $\kappa_{G}(x)$ under $\calO_{X_G,x}\to\kappa_G(x)$. Thus, giving the subring $\calOcirc_{X_G,x}\subseteq\calO_{X_G,x}$ is equivalent to giving an equivalence class of semivaluations on $\calO_{X_G,x}$ with kernel equal to $m_{G,x}$. By a slight abuse of language we fix one such semivaluation and the induced valuation on $\kappa_G(x)$ and denote them $\|\ \|_x$. This valuation extends to $\calH(x)$ by continuity and we will use the same notation for the extension. Note that throughout this paper $|\ |$ refers to real-valued (semi) valuations, while $\|\ \|$ refers to (semi) valuations that may have values in arbitrary valued groups.

\begin{remark}
The valuative interpretation of the points of $X_G$ is very important in adic geometry. In fact, $X_G$ is the underlying topological space of the Huber's adic space corresponding to $X$, see \cite{Huberbook}.
\end{remark}

\subsubsection{The retraction $\gtr_X$}
Further properties of the space $X_G$ were studied in \cite{Put-Schneider}. In particular, one shows that there is a retraction $\gtr_X\:X_G\to X$ and $\gtr_X$ identifies $X$ with the maximal Hausdorff quotient of $X_G$. In particular, any point $x\in X_G$ has a single generization $y=\gtr_X(x)$ in $X\subseteq X_G$ and $\gtr_X^{-1}(y)=\oy$ is the closure of $y$ in $X_G$.

\subsubsection{Germ reductions}
If $x\in X_G$ and $y=\gtr_X(x)$ then the generization map $\phi\:\calO_{X_G,x}\to\calO_{X_G,y}$ is local and induces an embedding of real valued fields $\kappa_G(x)\into\kappa_G(y)$. In particular, $|\ |_x$ is induced from $|\ |_y$ via $\phi$, so in the sequel we will freely use $|\ |_y$ instead of $|\ |_x$. In addition, $\calH(x)\toisom\calH(y)$, see \cite[Lemma~9.2.5]{Temkintopforms}. The valuation $\|\ \|_x$ is composed from the real valuation $|\ |_x$ of $\calH(x)$ and a valuation on the residue field $\wHx$. So, all points of $\oy$ are interpreted as valuations on $\wHy$.

\begin{remark}
In fact, the closure of $y$ in $X_G$ can be identified with the reduction of the germ $(X,y)$, which is a certain Riemann-Zariski space $\wt{(X,y)}$ of valuations on $\wHy$, see \cite[Remark~2.6]{temlocal}.
\end{remark}

\subsection{Topological ramification}

\subsubsection{Topological ramification index}
Assume that a morphism of $k$-analytic spaces $f\:Y\to X$ is \'etale at a point $y\in Y$ and let $x=f(y)$. The number $n_y=[\calH(y):\calH(x)]$ will be called the {\em topological ramification index} or the {\em multiplicity} of $f$ at $y$. Naturally, we say that $f$ is {\em topologically ramified} at $y$ if $n_y>1$, and the set of all points with $n_y>1$ is called the {\em topological ramification locus}.

\begin{remark}\label{topramrem}
(i) In \cite[Section 6.3]{berihes}, Berkovich calls $n_y$ the ``geometric ramification index'', but we prefer to change the terminology since the word ``geometric'' usually refers to a property satisfied after an arbitrary base field extension.

(ii) The term ``topological ramification'' is justified by the observation that $n_y=1$ if and only if $f$ is injective on a neighborhood of $y$. In fact, $n_y=1$ if and only if $f$ is a local isomorphism at $y$ by \cite[Theorem 3.4.1]{berihes}.

(iii) In fact, $n_y$ is the local degree of $f$ at $y$, the notion that makes sense for arbitrary flat quasi-finite morphisms. In particular, if $f$ is finite and the nice compact curves are connected then the value of $n_x=\sum_{y\in f^{-1}(x)}n_y$ is independent of $x$.
\end{remark}

\subsubsection{Tame and wild ramification}
We say that $f$ is {\em tamely topologically ramified} at $y$ if $n_y$ is invertible in $\tilk$. Otherwise, the topological ramification at $y$ is called {\em wild}. Usually, we will simply say that $f$ is wild \'etale or tame \'etale at $y$.

\begin{remark}
The definition of tameness is due to Berkovich, see \cite[Section 6.3]{berihes}. One might be surprised that a (type 2) point $y$ with an unramified extension $\calH(y)/\calH(x)$ is declared wildly ramified when $p|n_y$. However, this definition has the following two advantages: any tame \'etale covering of a disc is trivial (see \cite[Theorem 6.3.2]{berihes}) and the tame ramification locus is a finite graph (we do not need the second claim, and it will be proved elsewhere). Both claims would fail if one extends the definition of tameness by only requiring that $\calH(y)/\calH(x)$ is tame.
\end{remark}

\subsubsection{Ramification points}
It will be convenient to extend the above classification to actual ramification points, at least in the case of curves. So, assume that $f\:Y\to X$ is a generically \'etale morphism of quasi-smooth curves, and $y\in Y$ is a ramification point of $Y$. Let $n_y$ be the usual ramification index, i.e. $m_{f(y)}\calO_y=m_y^{n_y}$. We say that the ramification at $y$ is {\em tame} if $n_y$ is invertible in $k$, and the ramification is {\em wild} otherwise. This fits the usual algebraic definition.

In addition, we classify $y$ as a point of topological ramification and call $n_y$ the topological ramification index at $y$. We say that the topological ramification at $y$ is {\em tame} if $n_y$ is invertible in $\tilk$. Otherwise, the topological ramification is wild. This fits the intuition that topological ramification is related to the ramification of the completed residue fields $\calH(s)$, while usual ramification is related to ramification of the local rings $\calO_s$.

\subsubsection{Wild and tame loci}
The set of all points $y\in Y$ where $f$ is topologically tame (resp. topologically wild) will be called the {\em tame locus} (resp. the {\em wild locus}) of $f$.

\subsection{Local structure of analytic curves}

\subsubsection{Nice compact curves}
A purely one-dimensional $k$-analytic space is called a $k$-analytic curve, and we will omit $k$ as a rule. Almost all our work is concerned with compact separated strictly $k$-analytic curves $X$ that are smooth at Zariski closed points. For the sake of brevity, we will refer to such curves {\em nice compact curves} throughout the paper.

Note that the smoothness assumption on a nice compact curve $X$ means that $X$ is rig-smooth, i.e. the associated rigid space is smooth. Since $X$ is strictly analytic, this is also equivalent to requiring that $X$ is quasi-smooth, see \cite[Section~2.1.8]{Ducpl}. Furthermore, $X$ is smooth if and only if it has no boundary, so this happens if and only if it is proper.

\subsubsection{Types of points}\label{typessec}
Let $X$ be a $k$-analytic curve. Points of type 1 are the $k$-points. For any other point $x\in X$, the $k$-field $\calH(x)$ is one-dimensional and the type of $x$ is the  type of $\calH(x)$. To any type 2 point we associate the {\em genus} $g(x)$ equal to the genus of the functional $\tilk$-field $\wHx$. It will also be convenient to set $g(x)=0$ for any point of another type. A point is called {\em monomial} (resp. {\em hyperbolic}) if it is of type 2 or 3 (resp. 2, 3 or 4). The set of such points will be denoted $X^\mon$ (resp. $X^\hyp$).

\subsubsection{Intervals}
By an interval $I\subset X$ we mean a topological subspace homeomorphic to a closed or an open interval and provided with an orientation. Often we will denote $I$ as $[x,y]$ or $(x,y)$; such notation specifies the boundary of the interval and the orientation.

\subsubsection{Branches}\label{brsec}
A {\em branch} $v$ of $X$ at a point $x\in X$ is an isomorphism class of germs of intervals $[x,y]\subset X$, see \cite[Secion~1.7]{curvesbook}. The set of branches at $x$ will be denoted $\Br(x)$, it can also be identified with the set of connected components of $V\setminus\{x\}$, where $V$ is a sufficiently small connected neighborhood of $x$. We claim that any point of type 1 or 4 is unibranch, and any point of type 3 has two branches. This is clear for points in $\bfP^1_k$ and the general case is reduced to this by Section~\ref{elemsec} below. Note that we use here that $X$ is nice: non-smooth points of type 1 may have more than one branch and boundary points of type 3 may have 1 or 0 branches (and even more than two branches in the non-separated case). For shortness, we will denote the branch at a type 1 or 4 point $x$ by the same letter.

\begin{remark}
It is convenient to introduce branches at arbitrary points, but only branches at type 2 points are really informative.
\end{remark}

\subsubsection{Elementary neighborhoods}\label{elemsec}
A point $x\in X$ has a fundamental family of elementary neighborhoods $U_i$, see \cite[Section~3.6]{berihes}. Recall that $U_i$ are isomorphic to discs if the type of $x$ is 1 or 4, $U_i$ are annuli if the type is 3, and $U_i\setminus\{x\}$ is a union of open discs and finitely many open annuli if the type is 2. In the latter case, these discs and annuli are parameterized by the branches at $x$.

\subsubsection{Parameters}
Recall that \'etale morphisms in rigid geometry correspond to quasi-\'etale morphisms of Berkovich spaces. The discrepancy is due to the fact that \'etale morphisms in Berkovich geometry are defined to be without boundary. A morphism is quasi-\'etale if it is \'etale $G$-locally on the source and the target, see \cite[Section 3]{bervanish1} for the precise definition.

By a {\em parameter at} $x$ we mean any element $t_x\in\calO_{X_G,x}$ such that the induced map $f\:U\to\bfA^1_k$ from an affinoid subdomain containing $x$ is quasi-\'etale at $x$. Thus, if $x$ is of type 1 then $t_x-c$ is a uniformizer of $\calO_{X,x}$ for $c=t_x(x)\in k$, and if $x$ is hyperbolic then $t_x$ is an element of
$\kappa_G(x)\setminus k$ if $\cha(k)=0$ and an element of $\kappa_G(x)\setminus(\kappa_G(x))^p$ if $\cha(k)=p$.

If $x$ is hyperbolic then the parameter $t_x$ is {\em tame} (resp. {\em monomial}) if it is a tame (resp. monomial) parameter of $\calH(x)$. If $x$ is of type 1 then any uniformizer $t_x$ is tame by definition, and $t_x$ is {\em monomial} if $t_x(x)=0$. We warn the reader that even when $t_x$ is tame, it may happen that $p$ divides $[\calH(x):\wh{k(t_x)}]$ and thus $f$ has wild topological ramification at $x$.

There always exists a tame parameter, and if $x$ is monomial then it can also be chosen monomial. Indeed, this is obvious when $x$ is of type 1, and for hyperbolic points one can choose such a parameter $t'_x\in\calH(x)$ and take $t_x\in\kappa_G(x)=\calO_{X_G,x}$ such that $|t'_x-t_x|<r_{t_x}$. It is easy to see that $t_x$ is also a tame (resp. tame and monomial) parameter.

\subsection{Type 5 points}

\subsubsection{Germ reduction curves}\label{germredsec}
Any (non-analytic) point $x\in X_G\setminus X$ corresponds to a non-trivial valuation on $\wHy$, where $y=\gtr(x)\in X$. For curves this can happen only when $y$ is of type 2 and then $x$ corresponds to a discrete valuation on $\wHy$ trivial on $\tilk$. The closure of $y$ in $X_G$ can be identified with a normal $\tilk$-curve $C_y$ such that $\tilk(C_y)=\wHy$. Indeed, $C_y=\widetilde{(X,y)}$ is a one-dimensional Riemann-Zariski space over $\tilk$, hence it is a normal curve. We will call $C_y$ the {\em germ reduction} of $X$ at $y$. For concrete computations one can often use the following recipe: there always exist a formal model $\gtX$ such that the reduction map $\pi_\gtX\:X\to\gtX_s$ takes $y$ to the generic point of an irreducible component $Z\subset\gtX_s$ and $\tilk(Z)=\wHy$, and then $C_y$ is the normalization of $Z$. Note also that $C_y$ is proper if and only if $y\in\Int(X)$ is an inner point, see \cite[Theorem~4.1]{temlocal}.

\subsubsection{Type 5 points}
If $X$ is a curve then any point $x\in X_G\setminus X$ will be called a {\em type 5} point. It corresponds to a closed point of the germ reduction $C_y$ at a type 2 point $y$. The valuation $\|\ \|_x$ corresponding to $\calOcirc_{X_G,x}$ is of rank two, and it is composed from the real-valuation $|\ |_x$ (i.e. induced from the generization map $\calO_{X_G,x}\to\calO_{X_G,y}$) and a discrete valuation on $\wt{\kappa_G(x)}=\wHx=\wHy$. We denote the latter as $\veps_x\:\wHx^\times\to\lam_x^\bfZ$, where $\lam_x$ is the generator of the group of values satisfying $\lam_x<1$. Since $|\calH(x)^\times|=|k^\times|$, the group of values $\|\calH(x)^\times\|_x$ splits canonically into the lexicographic product $|\calH(x)|_y\times\lam_x^\bfZ$ of ordered groups, in particular, $\veps_x$ extends to a homomorphism (but not a valuation!) $\calH(x)^\times\to\lam_x^\bfZ$ so that $\|f\|_x=(|f|_y,\veps_x(f))$. In the sequel, we will prefer to work with the additive homomorphism $\nu_x\:\calH(x)^\times\to\bfZ$ corresponding to $\veps_x$, so that $\|f\|_x=(|f|_y,\lam_x^{\nu_x(f)})$.

\begin{remark}
There is a natural bijection between type 5 specializations of a type 2 point $y$ and the branches at $y$, see \cite[Section~3.2]{curvesbook}. So, we will freely identify them.
\end{remark}

\subsubsection{Multiplicities}\label{multsec}
Assume that $f\:Y\to X$ is a generically \'etale morphism between nice compact curves, $y\in Y_G$ is of type 5 and $x=f(y)\in X_G$. Consider the type 2 points $\eta$ and $\epsilon$ that generize $y$ and $x$, respectively, and let $f_\eta\:C_\eta\to C_\epsilon$ be the induced morphism between the germ reductions. The {\em multiplicity} $n_y$ of $f$ at $y$ is defined to be the usual (algebraic) multiplicity of $y$ in the fiber $f_\eta^{-1}(x)$.

\begin{remark}\label{mult5rem}
(i) Note that $n_\eta=[\calH(\eta):\calH(\epsilon)]=[\wt{\calH(\eta)}:\wt{\calH(\epsilon)}]$ since $\calH(\epsilon)$ is stable for any point of type 2. This fact and a local computation of the degree of $f_\eta$ allow to extend Remark~\ref{topramrem}(iii) to points of type 5: if $f$ is finite and $X$ is connected then $\sum_{y\in f^{-1}(x)}n_y=\deg(f)$ for any $x$ of type 5. In particular, this indicates that our definition of the multiplicity is ``correct''.

(ii) Note that for type 5 points, $n_y=\#\|\calH(y)^\times\|_y/\|\calH(x)^\times\|_x$ is the ramification index of the extension of valued fields of height two, but it can be strictly smaller than $[\calH(y):\calH(x)]$. Note also that $\lam_x=\lam_y^{n_y}$.
\end{remark}

\subsubsection{Parameters}
We say that points of type 5 are both monomial and hyperbolic, so $X_G^\mon$ consists of $X^\mon$ and points of type 5, and similarly for $X_G^\hyp$. If $x$ is of type 5 then by a parameter at $x$ we mean an element $t_x\in\kappa_G(x)$ such that $t_x\notin k$ and $t_x$ is not a $p$-th power when $\cha(k)=p>0$. Furthermore, $t_x$ is {\em monomial} if $|t_x-c|_x\ge|t_x|$ for any $c\in k$. This happens if and only if $\nu_x(t_x)\neq 0$. Also, $t_x$ is {\em tame} if for the induced map $f\:U\to\bfA^1_k$ the map between the germ reductions $C_x\to C_{f(x)}$ is not wildly ramified at $x$, i.e. $n_x$ is invertible in $\tilk$. In particular, $t_x$ is tame monomial if and only if $\nu_x(t_x)$ is invertible in $\tilk$, hence there exist plenty of such parameters.

\subsection{Global structure}
The main result about global structure of nice compact curves is the semistable reduction theorem, which can be formulated either in terms of formal models or in terms of skeletons, see \cite[Section~4.3]{berbook}. We will only discuss the second approach.

\subsubsection{Skeletons of curves}
By a finite topological graph we mean a topological space $\Gamma$ with a finite set of vertices $\Gamma^0\subseteq\Gamma$ which is isomorphic to the topological realization of a finite graph. In particular, $\Gamma\setminus\Gamma^0$ is a finite disjoint union (perhaps empty) of open intervals called edges of $\Gamma$. If no confusion with combinatorial graphs is possible, we will simply say that $\Gamma$ is a finite graph.

By a {\em skeleton} of a nice compact curve $X$ we mean a finite graph $\Gamma\subset X$ such that the following conditions hold:

(i) $\Gamma^0$ consists of type 1 and 2 points and contains all boundary points and points of positive genus,

(ii) $X\setminus \Gamma$ is a disjoint union of open discs.

We explain below that any nice curve possesses a skeleton, but let us list basic properties of skeletons first. To make notation uniform, by a semi-annulus we mean either an open annulus or an open disc punched at a type 1 point.

\begin{remark}
(i) Since points of type 1 and 4 are unibranch, $\Gamma$ contains no type 4 points, and any $x\in\Gamma$ of type 1 is a vertex (in fact, a leaf).

(ii) Any edge $e\subset\Gamma$ is contained in a semi-annulus $A\subset X$ so that $e$ is the skeleton of $A$, see, for example, \cite[Theorem~4.1.14]{curvesbook}. In fact, this result means that $X\setminus\Gamma^0$ is a disjoint union of open discs and semi-annuli, with semi-annuli parameterized by the edges of $\Gamma$. In terms of \cite{curvesbook} this means that $\Gamma^0$ is a triangulation of $X$.

(iii) Any skeleton $\Gamma$ is a deformation retract of $X$; in particular, $\pi_0(\Gamma)=\pi_0(X)$ and $h^1(\Gamma) = h^1(X)$.
\end{remark}

\subsubsection{Enlarging a skeleton}
One of a very special features of the theory of curves is that any enlargement of a skeleton is again a skeleton.

\begin{lemma}\label{enlargelem}
Let $X$ be a nice compact curve with a skeleton $\Gamma$, and assume that $\Gamma'\subset X$ is a finite subgraph such that $\Gamma\subseteq\Gamma'$, $\pi_0(\Gamma)=\pi_0(\Gamma')$ and all vertices of $\Gamma'$ are of type 1 and 2. Then $\Gamma'$ is a skeleton of $X$.
\end{lemma}
\begin{proof}
This reduces to proving that if $D\subset X$ is an open disc with limit point $q\in X$, and $T$ is a finite connected subgraph of $X$ such that $q\in T$ and $T\setminus\{q\}\subset D$, then $D\setminus T$ is a disjoint union of discs. By induction on the size of $T^0$ this reduces to the case when $T=[x,q]$, where $x$ is of type 1 or 2 inside of $D$. The latter is trivial.
\end{proof}

\subsubsection{Semistable reduction}
The semistable reduction theorem asserts that any nice compact curve $X$ possesses a skeleton $\Gamma$. Moreover, it follows from Lemma~\ref{enlargelem} that for any finite set $V$ of type 1 and 2 points one can achieve that $V\subset\Gamma^0$. As we have seen above, this provides very detailed information about $X$.

\subsubsection{Stable reduction}
Assume that $X$ is connected. The stable reduction theorem sharpens the semistable reduction by asserting that, excluding a few degenerate cases, there exists a unique minimal skeleton $\Delta(X,V)$ containing $V$. It turns out that the only degenerate cases are as follows: $X=\bfP^1_k$ and $V$ consists of at most 2 points of type 1, and $X$ is a Tate curve while $V$ is empty. For example, see \cite[Sections~5.4, 5.5]{curvesbook}.

\subsubsection{Morphisms of annuli}
Let $A=\calM(k\{R^{-1}t,rt^{-1}\})$ be a closed annulus. Its minimal skeleton $l$ can be naturally identified with the interval $[r,R]$. For example, if $A$ is identified with the subdomain of $\bfA^1_k$ given by $r\le|t|\le R$ then $l$ consists of the generalized Gauss valuations with $r\le |t|\le R$. We will need the following classical result whose proof is omitted (for example, see \cite[Section~6.2]{berihes}).

\begin{lemma}\label{annulilem}
Let $A_1=\calM(k\{R^{-1}t,rt^{-1}\})$ and $A_2=\calM(k\{S^{-1}x,sx^{-1}\})$ be annuli with minimal skeletons $l_1$ and $l_2$, respectively, and assume that $f\:A_1\to A_2$ is a finite morphism. If $|\ |_i$ denotes the spectral norm of $A_i$ then

(i) $f$ is given by a series $x=h(t)=\sum_{i=-\infty}^\infty h_it^i$ and there exists $m\in\bfZ$ such that $|h-h_mt^m|_1<|h|_1$. The degree of $f$ equals to $|m|$.

(ii) $f^{-1}(l_2)=l_1$ and the induced map $l_1\to l_2$ is bijective and given by $|x|_2=|h_m|\cdot|t|_1^m$. In particular, $n_y=|m|$ for any $y\in l_1$.
\end{lemma}

\subsubsection{Skeleton of a morphism}
Let $f\colon Y\to X$ be a finite generically \'etale morphism of nice compact curves. By a {\em compatible} pair of skeletons we mean skeletons $\Gamma_X\subset X$ and $\Gamma_Y\subset Y$ such that $\Gamma_Y$ is the preimage of $\Gamma_X$, in the sense that $f^{-1}(\Gamma_X)=\Gamma_Y$ and $f^{-1}(\Gamma_X^0)=\Gamma_Y^0$. By a {\em skeleton} of $f$ we mean a compatible pair of skeletons $\Gamma_f=(\Gamma_Y,\Gamma_X)$ such that $\Gamma_Y$ contains the set $\Ram(f)$ of the ramification points of $f$. Note that on the complement to $\Gamma_f$, the morphism $f$ breaks down into a disjoint union of finite \'etale morphisms between open discs.

\begin{lemma}\label{edgemultlem}
If $(\Gamma_Y,\Gamma_X)$ is a skeleton of a morphism of nice compact curves $f\:Y\to X$ then the multiplicity function $n_y$ is constant along any edge $e\subset\Gamma_Y$.
\end{lemma}
\begin{proof}
Any open semiannulus is a union of closed annuli, hence the lemma follows from Lemma~\ref{annulilem}(ii).
\end{proof}

\subsubsection{Simultaneous semistable reduction}\label{simulsec}
The simultaneous semistable reduction theorem asserts that any finite generically \'etale morphism of nice compact curves possesses a skeleton. This is not essentially stronger than the semistable reduction theorem and can be deduced from it as follows. Start with any skeleton $\Gamma'_X$ of $X$, choose a skeleton $\Gamma'_Y$ of $Y$ that contains $\Ram(f)$ and $f^{-1}(\Gamma'_X)$, and set $\Gamma_X=f(\Gamma_Y)$ and $\Gamma_Y=f^{-1}(\Gamma_X)$. Clearly, $\Gamma_X$ is connected and contains $\Gamma'_X$, hence it is a skeleton by Lemma~\ref{enlargelem}. We claim that $\Gamma_Y$ is a skeleton too, and hence $(\Gamma_Y,\Gamma_X)$ is a skeleton of $f$.

We should prove that a connected component $D$ of $Y\setminus\Gamma_Y$ is an open disc. Note that $D$ is a connected component of $f^{-1}(E)$, where $E$ is a connected component of $X\setminus\Gamma_X$ and hence $E$ is a disc. In addition, $\Gamma'_Y\subseteq\Gamma_Y$ hence $D$ is contained in a connected component $D'$ of $Y\setminus\Gamma'_Y$, which is an open disc. Finally, $f(D')$ is contained in an open component $E'$ of $X\setminus \Gamma'_X$, which is an open disc too. It remains to use the simple fact that for any morphism $D'\to E'$ between open discs, the preimage of an open disc $E\subseteq E'$ is a disjoint union of open discs.

\begin{remark}
In the language of formal models, the theorem is due to Liu. The skeletal version appeared in \cite{ABBR}.
\end{remark}

\subsubsection{Simultaneous stable reduction}
One can also show that, excluding a couple of degenerate cases, there exists a unique minimal skeleton of $f$. In particular, if one starts with a skeleton $\Gamma$ of $X$ then there exists a unique minimal skeleton $(\Gamma_Y,\Gamma_X)$ of $f$ such that $\Gamma\subseteq\Gamma_X$ (\cite[Corolary~4.18]{ABBR}). Indeed, take the minimal skeleton $\Gamma'\subset Y$ containing $f^{-1}(\Gamma)$ and take $\Gamma_X$ to be the minimal skeleton containing $f(\Ram(f)\cup\Gamma')$.

\subsection{Piecewise monomial structure}

\subsubsection{A metric}
An interval in an analytic curve possesses a natural metric. For brevity, we only recall the approach of \cite[Section~5.58]{metrics_on_curves}, which makes use of semistable reduction. Probably, this is the shortest, though not the most conceptual, way. If $I\subset X^\mon$ is the skeleton of an annulus $A\subseteq X$ isomorphic to the subdomain of $\bfA^1_k$ given by $s<|t|<r$ then $l(I)=\log r-\log s$. In general, it follows from semistable reduction that there exists a finite subset $S$ such that the connected components $I_j$ of $I\setminus S$ are skeletons of open annuli and we set $l(I)=\sum_j l(I_j)$. The length $l(I)$ turns out to be independent of choices, so we obtain a metric on any interval inside of $X^\mon$. Moreover, this metric extends to $X^\hyp$ by continuity. All type 1 points are singular for the metric: if $[a,b]\subset X$ and $a$ is of type 1 then the length of $(a,b)$ is infinite.

\subsubsection{Radius parametrization}
Note that if $[x,y]$ is an interval in $\bfP^1_k$ and $y$ dominates $x$ then $l([x,y])=\log r(y)-\log r(x)$. More generally, by a {\em radius parametrization} of an interval $I\subset X$ we mean a function $r\:I\to[0,\infty]$ such that

(i) $l([a,b])=\log r(b)-\log r(a)$ for any subinterval $[a,b]\subset I$,

(ii) $r(x)\in|k^\times|$ for some point $x\in I$ of type 2.

In particular, if $I=[x,y]$ then $r(x)=0$ if and only if $x$ is of type 1 and $r(y)=\infty$ if and only if $y$ is of type 1. Also, $r(x)\in|k^\times|$ for any type 2 point and $r(x)\notin|k^\times|$ for any type 3 point.

\subsubsection{Piecewise monomial functions}
Let $S$ be a subset of $X$ (our cases of interest are $S=X$ and $S=X^\hyp$). A function $f\:S\to[0,\infty]$ is called {\em piecewise monomial} if for any interval $I\subset S$ there exists a finite subdivision $I=\cup_{j=1}^mI_j$ such that for each $j$ there exist $n\in\bfZ$ and $a\in(0,\infty)$ with $f|_{I_j}=ar^{n}$, where  $r$ is a radius parametrization of $I_j$. If, moreover, $a\in|k^\times|$ then we say that $f$ is {\em piecewise $|k^\times|$-monomial}; this property is independent of the choice of the radius parametrization. Note also that $n$ is independent of the radius parametrization once the orientation of $I_j$ is fixed, and $n$ changes sign if we switch the orientation.

\begin{example}\label{pmexam}
(i) If $f\in\calO_X(U)$ is an analytic function then $|f|$ is a piecewise $|k^\times|$-monomial function on $U$.

(ii) The radius function on $\bfA^1_k$ is piecewise $|k^\times|$-monomial. Note that it is semicontinuous but not continuous (in the usual topology).

(iii) A product of piecewise $|k^\times|$-monomial functions is piecewise $|k^\times|$-monomial.

(iv) If $f\:Y\to X$ is a morphism of curves and $I\subset Y$ is an interval then it follows from Lemma~\ref{pmmaplem} below that for any piecewise $|k^\times|$-monomial function $\phi\:X\to[0,\infty]$ the pullback $\phi^* f=\phi\circ f$ is a piecewise $|k^\times|$-monomial function on $Y$.

(v) As an important particular case of the above consider the following situation: $t\in\Gamma(\calO_Y)$ is a global function on $Y$ and $r_t$ is the radius function of $t$, i.e. $r_t(y)=\inf_{a\in k}|t-a|_y$ for any $y\in Y$. Then $t$ induces a morphism $Y\to\bfA^1_k$ and $r_t$ is the pullback of the radius function on the target. In particular, $r_t$ is piecewise $|k^\times|$-monomial.
\end{example}

\subsubsection{Slopes}
If $\phi\:X\to\bfR_+$ is piecewise $|k^\times|$-monomial, $x\in X$ is a point and $b\in\Br(x)$ is a branch then there exists an interval $I=[x,y]$ in the direction of $b$, and taking $I$ small enough we can achieve that $\phi=ar^n$ is monomial on $I$. We call $n$ the {\em slope} of $\phi$ at $b$ and denote it $\slope_b(\phi)$. As we have mentioned, $n$ depends only on the orientation of the interval, which is fixed by choosing $x$ to be the starting point.

\begin{remark}
(i) If $x$ is of type 3 and $u,v$ are its two branches then the slopes at $u$ and $v$ are opposite, that is, $\phi$ is monomial locally at $x$. Indeed, otherwise $\phi=ar^n$ at $u$ and $\phi=br^m$ at $v$ for $m\neq n$, and one gets that $ar^n=br^m$ at $x$, yielding a contradiction via $|r(x)|=(|a|/|b|)^{1/(m-n)}\in|k^\times|$. Up to the sign, these slopes are equal to the image of $\phi(y)$ in $|\calH(x)^\times|/|k^\times|=\bfZ$.

(ii) If $x$ is unibranch and $f\in\calO_{X,x}$ has zero of order $n$ at $x$ then $|f|$ has slope $n$ at $x$. In particular, $|f|$ is of slope zero at any type 4 point.

(iii) If $x$ is of type 2, $v\in C_x$ and $f\in\calO_{X_G,x}$ then $\|f\|_v=(|f|_x,\lambda_x^{\slope_v|f|})$.
\end{remark}

\subsubsection{Piecewise monomiality of morphisms}
The assumptions on the morphism $f$ in the following lemma are redundant, but we use them to give a short proof based on semistable reduction.

\begin{lemma}\label{pmmaplem}
Assume that $f\:Y\to X$ is a finite morphism between nice compact curves. If $\Gamma\subset Y$ is a finite graph then $f(\Gamma)$ is a finite graph and the induced map $\Gamma\to f(\Gamma)$ is piecewise $|k^\times|$-monomial with respect to the radius parameterizations on the edges of $\Gamma$ and $f(\Gamma)$.
\end{lemma}
\begin{proof}
If $\Gamma\subset Y^\mon$ then it is contained in a sufficiently large skeleton $\Gamma'\subset Y$. By the simultaneous semistable reduction we can find a skeleton $(\Gamma_Y,\Gamma_X)$ of $f$ such that $\Gamma'\subset\Gamma_Y$ (it suffices to require that $f(\Gamma'^0)\subseteq\Gamma_X^0$). Then it is clear that $f(\Gamma)$ is a finite graph and we claim that the maps on the edges are monomial. Indeed, this reduces to study of a map $\phi\:A_1\to A_2$ between closed annuli, and then Lemma~\ref{annulilem}(ii) does the job.

It remains to consider the case when $\Gamma$ contains a point $y$ of type 1 or 4, say $I=[y,q]$, and we should prove that the map is piecewise monomial at $y$. We know that the map is piecewise monomial on $(y,q)$, so we should only prove that it has finitely many breaks near $y$, i.e. the slope of $f$ changes finitely many times in a neighborhood of $y$. Shrinking $Y$ around $y$ we can assume that $Y=\calM(k\{t\})$ is a unit disc (see \ref{elemsec}) and $\Gamma$ is the interval $I=[y,q]$ connecting $y$ with the maximal point of $Y$. Similarly, we can assume that $X$ is a unit disc, and so $f$ is given by a series $h(t)=\sum_{i=0}^\infty a_it^i$. It suffices to prove that the slope of $f$ on $(y,q]$ is a non-negative increasing function. Furthermore, it suffices to check the latter claim for a closed subinterval $J\subset(y,q]$. By change of coordinates we can move $J$ to a subinterval of $[0,q]$, and then the claim becomes obvious: the slope equals to $m$ on any subinterval of $[0,q]$ on which $a_mt^m$ is the dominant term of $h(t)$.
\end{proof}

\subsubsection{The multiplicity function}
Let $n_f\:Y\to\bfN$ denote the {\em multiplicity function} $y\mapsto n_y$.

\begin{lemma}\label{multfuncor}
If $f\:Y\to X$ is as in Lemma~\ref{pmmaplem} then the multiplicity function $n_f$ is upper semicontinuous. In addition, if $I\subset Y$ is a closed interval then the restriction of $n_f$ onto $I$ has finitely many discontinuity points, all of which are of type 2.
\end{lemma}
\begin{proof}
Let us show that $n_f$ is upper semicontinuous at a point $y\in Y$. For an analytic neighborhood $X'$ of $x$ let $Y'$ be the connected component of $f^{-1}(X')$ that contains $y$. Taking $X'$ sufficiently small we can achieve that $y$ is the only preimage of $x$ in $Y'$. Then the finite map $Y'\to X'$ is of degree $n_f(y)$ and hence $n_f(y')\le n_f(y)$ for any $y'\in Y'$.

Now, let us study $n_f|_I$. The argument is similar to the one used in Lemma~\ref{pmmaplem}. Assume first that $I\subset Y^\mon$. By the simultaneous semistable reduction, we can find a skeleton $(\Gamma_Y,\Gamma_X)$ such that $I\subseteq\Gamma_Y$. If $e$ is an edge in $\Gamma_Y$ then the multiplicity equals to the absolute value of the slope of $f$ on $e$ and is constant along $e$ by Lemma~\ref{annulilem}(ii).

If $I=[y,q]$ with $y$ of type 1 or 4 then we reduce to the case when $X$ and $Y$ are discs, and the same argument as in the proof of Lemma~\ref{pmmaplem} shows that the multiplicity decreases when we approach $y$. In particular, it stabilizes from some stage. Shrinking $X$ and $Y$ we can assume that $n_f$ is constant along $(y,q]$. Then any point of $f((y,q])$ has a single preimage in $Y$ and, by continuity, $y$ is the single preimage of $f(y)$. Hence, $n_y$ equals to the degree of $f$ and so $n_f$ is constant on all of $[y,q]$.
\end{proof}

\subsubsection{Multiplicity of $f$ at a branch}
Lemma \ref{multfuncor} implies that for any branch $v\in\Br(x)$ there exists an interval $I=(x,y]$ along $v$ such that the multiplicity of $f$ is constant on $(x,y]$. We set $n_v=n_y$ and call it the {\em multiplicity} of $f$ at $v$.

\begin{remark}
The notation $n_v$ will be convenient in the sequel, but it does not contain a new information: if $x$ is of type 1, 3 or 4 then $n_v=n_x$, and if $x$ is of type 2 then $v$ can be viewed as a type 5 point and $n_v$ agrees with the definition of Section~\ref{multsec}.
\end{remark}

\subsubsection{Application to tame parameters}
We conclude Section \ref{ancurvesec} with the following result.

\begin{lemma}\label{tameparamlem}
Assume that $X$ is a nice compact curve, $x\in X_G$ is a point, and $t$ is a tame parameter at $x$. Then there exists an analytic subdomain $Y\subseteq X$ such that $x\in Y_G$ and $t$ is a tame parameter at any point of $Y$.
\end{lemma}
\begin{proof}
Shrinking $X$ around $x$ we can assume that $t$ induces a morphism $f\:X\to\bfA^1_k$. Type 4 fields have no non-trivial tame extensions, hence if $x$ is of type 4 then $\calH(x)=\wh{k(t)}$. The latter implies that $f$ is a local isomorphism at $x$ (e.g., by \cite[Theorem 3.4.1]{berihes}), and we are done.

The case of $x$ of type 1 is clear because $f$ is a local isomorphism at $x$. If $x$ is of type 3 then we can replace $t$ by $t+c$ with $c\in k$ making it monomial. Then Lemma~\ref{annulilem} implies that for a small enough annulus $A$ around $x$ with a coordinate $\tau$, the map $f$ is given by $t=h(\tau)=\sum h_i\tau^i$ such that $|h-h_m\tau^m|_A<|h|_A$ for some $m\neq 0$. Since $n_x=|m|$ and the parameter is tame, $m$ is invertible in $\tilk$. Then it is easy to see that $f$ has multiplicity $m$ on the skeleton of $A$ and multiplicity one outside of it, hence $f$ is tame everywhere on $A$.

If $x$ is of type 2 then it follows from simultaneous semistable reduction that replacing $X$ by an affinoid domain we can achieve that $X$ is finite over $Z=f(X)$ and $X\setminus\{x\}=\coprod X_i$, $Z\setminus\{z\}=\coprod Z_j$, where $z=f(x)$, and $X_i$ and $Z_j$ are open discs. Since $\calH(x)/\calH(z)$ is unramified, the map $C_x\to C_{z}$ is generically \'etale, and removing some $X_i$'s and $Z_j$'s we can achieve that $C_x\to C_z$ is \'etale. Thus, the multiplicity of $f$ at any branch $v\in C_x$ is one. On the other hand any restriction $f_i\:X_i\to Z_j$ is a finite \'etale morphism between open discs and a direct computation shows that its degree equals to the multiplicity of $f$ at the branch $v\in\Br(x)$ pointing towards $X_i$. Thus, each $f_i$ is an isomorphism, in particular, $t$ is a tame parameter everywhere on $X$.

It remains to consider the case when $x$ is of type 5, say $x\in C_y$ where $y$ is of type 2. By \ref{elemsec}, shrinking $X$ we can achieve that $X\setminus\{y\}$ is  a disjoint union of open discs and annuli parameterized by $C_y$. Let $A$ be the component corresponding to $x$; without restriction of generality, it is an annulus. It follows from the simultaneous semistable reduction that taking $A$ small enough we can achieve that $f$ induces a finite \'etale morphism of $A$ onto an open annulus in $\bfA^1_k$. Then the same argument as used for type 3 points, shows that $f$ is tame on $A$ since it is a tame parameter at $x$. It remains to achieve that $f$ is a tame parameter at the other connected components of $X\setminus\{y\}$. But we are allowed to remove finitely many of them, and it remains to use what we have already proved for type 2 points.
\end{proof}

\section{The different function}\label{difsec}

\subsection{Definition and first properties}

\subsubsection{A morphism $f$}\label{fsec}
In the sequel, we consider a generically \'etale morphism $f\:Y\to X$ between nice compact curves.

\begin{definition}
The {\em different function} of $f$ is the map $\delta_f\:Y^\hyp\to(0,1]$ that associates the different $\delta_{\calH(y)/\calH(f(y))}$ to a point $y\in Y^\hyp$.
\end{definition}

Note that $\delta_f=1$ on the tame locus of $f$, as follows from Lemma~\ref{tamelem}. We will later extend $\delta_f$ to all of $Y$. An extension of $\delta_f$ to type 5 points will not be used, but we prefer to discuss it for the sake of completeness.

\begin{remark}
(i) The only extension of $\delta_f$ to a map $Y_G^\hyp\to(0,1]$ is by composing it with the retraction $Y_G^\hyp\to Y^\hyp$, hence it is not informative. More naturally, one can simply set $\delta_f(y)=\delta_{\calH(y)/\calH(f(y))}$ for any type 5 point (the different of an arbitrary finite separable extension of valued fields is defined in \cite[Section~6]{Gabber-Ramero}). Then $\delta_f(y)$ is an element of $|\calH(y)^\times|$, which is not a subgroup of $\bfR_+^\times$ for type 5 points, and hence $\delta_f$ should be viewed as a section of $\calO_{Y_G}^\times/(\calO_{Y_G}^\circ)^\times$.

(ii) Using the same ideas as in the proof of Theorem~\ref{stalkth}(ii) below, one can show that if $v$ is a type 5 point and $y=\gtr_Y(v)$ then $\delta_f(v)=(\delta_f(y),\lambda_v^{-\slope_v\delta_f})$. In this paper, Theorem~\ref{stalkth}(ii) will be used to control the slopes of $\delta_f$, making it unnecessary to extend $\delta_f$ to type 5 points.
\end{remark}

\subsubsection{The maps $\phi_x$}\label{fibdifsec}
Let $x\in X_G$. For an affinoid domain $V=\calM(\calA_V)$ with $x\in V_G$ consider the map $\hatOmega_{\calA_V/k}\to\hatOmega_{\calH(x)/k}$. These maps are compatible with the inclusions $V'\into V$, so an $\calO_{X_G,x}$-linear colimit map $\phi_x\:\Omega_{X_G,x}\to\hatOmega_{\calH(x)/k}$ and a differential $d\:\calO_{X_G,x}\to\Omega_{X_G,x}$ arise. Moreover, this differential is compatible with the differential of $\calH(x)$, i.e. we obtain the following cartesian square

$$
\xymatrix{
\calO_{X_G,x} \ar[r]\ar[d]^d & \calH(x)\ar[d]^d \\
\Omega_{X_G,x} \ar[r]^{\phi_x} & \hatOmega_{\calH(x)/k}.
}
$$

\subsubsection{Computation of $\delta_f$}
The following lemma is our main tool for working with $\delta_f$. Here the functions $r_{t_y}$ and $r_{t_x}$ are as defined in Example~\ref{pmexam}(v).

\begin{theorem}\label{compdeltalem}
Let $f$ be as in Section \ref{fsec}. Assume that $t_y$ and $t_x$ are tame parameters at points $y\in Y_G^\hyp$ and $x=f(y)$. Then there exists an analytic domain $U$ such that $y\in U_G$, $h=\frac{dt_x}{dt_y}$ is defined in $U$, and for any $z\in U^\hyp$ $$\delta_f(z)=|h(z)|r_{t_y}(z)r_{t_x}(z)^{-1}.$$
\end{theorem}
\begin{proof}
By Lemma~\ref{tameparamlem}, we can replace $Y$ with an analytic domain containing $y$ so that $t_y$ is a tame parameter at any point of $Y$. Similarly, we can achieve that $t_x$ is a tame parameter everywhere.

Consider the $\calO_{Y_G,y}$-linear map $\phi_y\:\Omega_{Y_G,y}\to\hatOmega_{\calH(y)/k}$ as defined in Section \ref{fibdifsec}; it is compatible with the differentials of $\calO_{Y_G,y}$ and $\calH(y)$. Since $dt_y$ is a generator of $\hatOmega_{\calH(y)/k}$ by Theorem~\ref{compldifth}(i), it is a generator of the invertible $\calO_{Y_G,y}$-module $\Omega_{Y_G,y}$. Hence $dt_y$ is a generator of $\Omega_{Y_G}$ in a small enough neighborhood $U_G\subseteq Y_G$ of $y$, and then $h$ is defined in $U$. Let $z\in U$. Since $\phi_z$ is $\calO_{Y_G,z}$-linear, one also has that $dt_x=hdt_y$ in $\hatOmega_{\calH(z)/k}$. So, the claim follows from Corollary~\ref{differentcor}(ii).
\end{proof}

\subsubsection{Piecewise monomiality}
As a first corollary of Theorem~\ref{compdeltalem} we obtain that the different function is piecewise monomial.

\begin{corollary}\label{pmlem}
Assume that $f$ is as in Section \ref{fsec}. Then the different function $\delta_f\:Y_\hyp\to(0,1]$ is piecewise $|k^\times|$-monomial.
\end{corollary}
\begin{proof}
By Theorem~\ref{compdeltalem}, $\delta_f$ can be presented $G$-locally as a product of piecewise monomial functions $|h|$, $r_{t_y}$ and $r_{t_x}^{-1}$.
\end{proof}

\subsection{Restrictions on $\delta_f$}

\subsubsection{Tameness and wildness}
The relation between the different function and the wild topological ramification locus is as follows.

\begin{lemma}\label{wildlocus}
Assume that $f\:Y\to X$ is a finite generically \'etale morphism of nice compact curves. Then,
if $y\in Y^\mon$ is a monomial point then $\delta_f(y)<1$ if and only if the extension $\calH(y)/\calH(x)$ is wildly ramified.
\end{lemma}
\begin{proof}
Set $x=f(y)$. It follows from the definition of the different that $\delta_f(y)<1$ if and only if $\Omega_{\calH(y)^\circ/\calH(x)^\circ}$ contains an element not killed by $\kcirccirc$. Since $\calH(x)$ is stable, the extension $\calH(y)/\calH(x)$ is defectless, and \cite[Lemma~5.5.9]{Temkintopforms} implies that $\delta_f(y)=1$ if and only if this extension is tame.
\end{proof}

\begin{remark}\label{kumerrem}
(i) The lemma implies that if $f$ is wild at a monomial point $y$ with $\delta_f(y)=1$ then $y$ is of type 2 and $\calH(y)/\calH(x)$ is an unramified extension of degree divisible by $p$. For a type 4 point, it may freely happen that $\delta_f(y)=1$ but $f$ is not split at $y$ and so $\calH(y)/\calH(x)$ is wild.

(ii) A typical example is provided by the Kummer covering $\bfP^1_k\to\bfP^1_k$ of degree $p$ over $k=\bfC_p$ ($t$ goes to $t^p$). A simple direct computation shows that $f$ is split at all points whose distance from $I$ exceeds $\frac{\log |p|}{p-1}$, the equality $\delta_f=|p|$ holds on the interval $I=[0,\infty]$, and $\delta_f$ increases with slope $p-1$ in all directions from $I$. (This also follows from a general description of degree $p$ coverings we will prove in Theorem~\ref{coneth}.) In particular, the locus of wild points $y$ with $\delta_f(y)=1$ consists of all points whose distance from $I$ is $\frac{\log |p|}{p-1}$, and it contains both type 2 and type 4 points.
\end{remark}

\subsubsection{$\delta_f$ on an annulus}\label{deltaansec}
Consider the annulus $A=\calM(k\{rt^{-1},t\})$ with skeleton $I=[r,1]$. Let $y\in I$ be the end-point given by $|t|_y=1$ and let $v\in\Br(y)$ be the direction along $I$. Assume that $f\:A\to\bfA^1_k$ is a generically \'etale morphism given by $h(t)=\sum_i h_it^i$. Choosing an appropriate coordinate $x$ on the target we can achieve that $h_0=0$ and $|h|_y=\max_i |h_i|=1$. Let $m$ denote the minimal integer such that $|h_m|=1$; note that $n_y=|m|_{\bfR}$ (we prefer to keep the notation $|m|$ for the absolute value of $m$ in $k$). Since $t$ and $x$ are monomial along $I$, Theorem~\ref{compdeltalem} implies that for a point $z\in I$ close enough to $y$, the different can be computed as $\delta_f(z)=|h'|_z|t|_z|x^{-1}|_z=|h'|_z|t^{1-m}|_z$, where $h'=\frac{dx}{dt}=\sum_{i\in\bfZ}ih_it^{i-1}$.

Using the above formula we can compute $\delta=\delta_f(y)$ and $s=\slope_v\delta_f$ as follows: $\delta=|nh_n|$ and $s=1-n+m-1=m-n$, where $n$ denotes the minimal integer such that $|nh_n|=|h'|_y=\max_i |ih_i|$. The numbers $m$, $s$ and $\delta$ are subject to certain restrictions that we are going to describe. First, we claim that
\begin{equation}\label{deltaineq}
|m|\le\delta\le|n|.
\end{equation}
Indeed, the right inequality holds because $|h_n|\le 1$, and the left one holds because $|nh_n|\ge |mh_m|=|m|$. Now let us split into two cases.

Case 1. {\em Assume that $s=0$.}  In this case, $m=n$ and so $\delta=|m|$. (In particular, in the equicharacteristic case we automatically obtain that $\delta=1$.) Conversely, if $m\in\bfZ_{>0}$ and $\delta=|m|$ (in particular, $|m|\neq 0$) then $h=t^m$ gives rise to a generically \'etale morphism $f$ such that $n_v=m$, $\slope_v\delta_f=0$ and $\delta_f(y)=|m|$.

Case 2. {\em Assume that $s\neq 0$.} If $\delta=|n|$ then $h_n=1$, hence $n\ge m$ by the definition of $m$, and we obtain that $s<0$. If $\delta=|m|$ then $|nh_n|=|m|=|mh_m|$, hence $m\ge n$ by the definition of $n$, and we obtain that $s>0$. This shows that at least one inequality in (\ref{deltaineq}) is strict, and so $|n|>|m|$ and $|s|=|n|$. To summarize, $|m|\le\delta\le|s|$ with at least one inequality being strict and $s>0$ (resp. $s<0$) if the first (resp. the second) inequality is an equality.

Conversely, assume that $m\in\bfZ_{>0}$, $s\in\bfZ$ and $\delta\in(0,1]$ satisfy the above condition. A direct computation shows that if $a\in k$ satisfies $|a|=\delta|m-s|^{-1}$, then $h=t^m+at^{m-s}$ induces a morphism $f$ with $\slope_v\delta_f=s$, $\delta_f(y)=\delta$ and $n_v=m$ (recall that $n_v$ denotes the multiplicity of $f$ at the branch $v$, see \ref{multsec}). Furthermore, a similar argument shows that even if $\delta\notin |k|$, one can choose $a\in k$ and a type 3 point $y'\in I$ with a branch $v'\in\Br(y')$ such that $h=t^m+at^{m-s}$ induces a morphism $f$ with $n_{v'}=m$, $\slope_{v'}\delta_f=s$ and $\delta_f(y')=\delta$.

\subsubsection{Slopes and values of $\delta_f$}
It turns out that the above restrictions on $m$, $n$ and $\delta$ are general. In the following theorem all absolute values refer to the valuation of $k$, and given a morphism $f\colon Y\to X$, a point $y\in Y$ and a branch $v\in\Br(y)$, the multiplicities of $f$ at $y$ and $v$ are denoted $n_y$ and $n_v$, respectively.

\begin{theorem}\label{restrictth}
Let $f\colon Y\to X$ be a finite generically \'etale morphism of nice $k$-analytic curves.

(i) If $y\in Y^\hyp$, $m=n_y$ and $\delta=\delta_f(y)$ then $\delta\ge|m|$. Moreover, this is the only restriction on $n_y$ and $\delta_f(y)$, i.e. any pair $m\in\bfZ_{>0}$ and $\delta\in(0,1]$ with $\delta\ge|m|$ is realized for some morphism $f$ and $y$.

(ii) If $y\in Y^\hyp$, $v\in\Br(y)$, $m=n_v$, $s=\slope_v\delta_f$ and $\delta=\delta_f(y)$ then the inequality
$|m-s|\ge\delta\ge|m|$ holds and, in addition, $s\le 0$ whenever the first inequality is an equality, and $s\ge 0$ whenever the second inequality is an equality. Moreover, this is the only restriction on $n_v$, $\slope_v\delta_f$ and $\delta_f(y)$, i.e. any triple $(m,s,\delta)\in\bfZ_{>0}\times\bfZ\times(0,1]$ satisfying this condition is realized for some $f$, $y$ and $v$.
\end{theorem}
\begin{proof}
We start with (ii). In the case when $Y$ is an annulus and $X$ is a domain in $\bfA^1_k$, this condition on the triple was established in \ref{deltaansec} (for example, the asserted inequality is nothing else but (\ref{deltaineq})). Moreover, we saw that any such triple can be obtained already when $Y$ is an annulus and $X=\bfA^1_k$. Although in this case $f$ is not finite, we can shrink $Y$ and $\bfA^1_k$ around $v$ and $f(v)$ so that $f$ becomes finite. It remains to deduce that the triple $(m,s,\delta)$ satisfies the assertion of (ii) when $f$ is arbitrary. We will do this using the continuity of the triple along intervals.

Let $I=[y,z]$ be an interval in $Y$ in the direction of $v$. It follows from the simultaneous semistable reduction theorem that shrinking $I$ we can achieve that for any $t\in(y,z)$, the interval $[t,z]$ is the minimal skeleton of an annulus $A$ and $f$ restricts to a finite morphism $A\to A'$ with $A'$ an annulus in $X$. Let $v(t)\in\Br(t)$ be the branch towards $z$. Shrinking $I$ we can achieve that $n_{v(t)}=n_v$ and $\slope_{v(t)}\delta_f=s$ for any $t\in(y,z]$. By the case of annuli,
the triple $(m,s,\delta_{v(t)})$ satisfies the condition of (ii). It remains to use that the condition is closed and $\delta_f$ is continuous on $I$.

Now, let us prove (i). We claim that there exists a branch $v\in\Br(y)$ such that $n_v|n_y$. Indeed, only the case when $y$ is of type 2 needs an argument, but then the multiplicity of a general branch equals to the degree of inseparability of $\wt{\calH(y)}/\wt{\calH(x)}$, where $x=f(y)$. For such branch, $|n_v|\ge|n_y|$, and we use that $\delta_f(y)\ge|n_v|$ by (ii). It remains to prove that any pair $(m,\delta)$ with $\delta\ge|m|$ is achieved for some $f$ and $y$. This is done similarly to the construction in \ref{deltaansec}: one takes $Y$ to be an annulus and uses a binomial when the inequality is strict, and a monomial when it is an equality.
\end{proof}

\begin{remark}\label{restrictrem}
(i) The tame case (i.e. $|m|=1$) of Theorem~\ref{restrictth}(i) is trivial. In the wild case, we see that the different can be any number from $(0,1]$ in the equicharacteristic case, and it can be any number from $[|m|,1]$ in the mixed characteristic case. This is the control on the different in the mixed characteristic case that misses in the equicharacteristic one. Particular cases of this (e.g., for stable fields) showed up in \cite{Lutkebohmert} and \cite{XFaber1}.

(ii) Part (ii) of Theorem~\ref{restrictth} provides a strip for the values of $\delta$; clearly $s$ is non-negative at the low border and non-positive at the top border. In addition, $s=0$ happens only on the border of the strip, and if $s\neq 0$ then $|s|>|m|$ and the inequality rewrites as $|s|\ge\delta\ge|m|$.

(iii) We will later need the particular  case when $p=2$ and $f$ is wild at $v$, i.e. $n_v$ is even. In the equicharacteristic case, this automatically implies that $s$ is odd. In the mixed characteristic case, there are more options, but if we assume, in addition, that $n_v\in 4\bfZ+2$ then either $s$ is odd or $s=0$ and $\delta=|2|$.
\end{remark}

\subsection{K\"ahler norm on $\Omega_{X_G}$ and the different}\label{kahlersec}
Our next aim is to study the behaviour of $\delta_f$ in a neighborhood of a type 2 point. This question is not local for the $G$-topology, in particular, we cannot work with a single parameter and a sheaf-theoretic argument is required. In the current section we will interpret $\delta_f$ as annihilator of a certain torsion $\calOcirc_{X_G}$-sheaf.

\subsubsection{The norm on $\Omega_{X_G}$}\label{kahlernormsec}
Recall that a seminorm on a sheaf of modules $\calF$ on a site $\calC$ is introduced in \cite[3.1.2]{Temkintopforms} as a family of (perhaps unbounded) seminorms on the modules $\calF(U)$ that satisfy certain natural conditions. A K\"ahler seminorm $|\ |_\Omega$ on the sheaf $\Omega_{X_G/S_G}$ is introduced in \cite[6.1.1]{Temkintopforms}, and by \cite[Theorem~6.1.13]{Temkintopforms} $|\ |_\Omega$ is a so-called analytic seminorm, in particular, it is determined by its stalks as $|\omega|_{\Omega,U}=\max_{x\in U}|\omega|_{\Omega,x}$, see \cite[\S\S3.2.7, 3.3.1, 3.3.3]{Temkintopforms}. Finally, the stalks of $|\ |_\Omega$ are described by  \cite[Theorem~6.1.8]{Temkintopforms}. In particular, for $\Omega_{X_G}=\Omega_{X_G/k}$ this works as follows: take $\phi_x$ as in Section~\ref{fibdifsec} and define a seminorm on $\Omega_{X_G,x}$ by the rule $|\omega|_x=|\phi_x(\omega)|_{\Omega,\calH(x)/k}$.

\subsubsection{The sheaf $\Omega^\di_{X_G}$}
By $\Omega^\di_{X_G}$ we denote the unit ball of $|\ |_\Omega$. It is the $\calO^\circ_X$-submodule of $\Omega_X$ whose sections satisfy $|\omega|_{\Omega,x}\le 1$ at any point $x\in X_G$.

\begin{theorem}\label{stalkth}
Let $X$ be a nice compact curve. The stalk of $\Omega^\di_{X_G}$ at a point $x\in X_G$ is described as follows:

(i) If $x$ is of type 1 then $\Omega^\di_{X_G,x}=\Omega_{X_G,x}$.

(ii) If $x$ is of type 2, 3, or 5 then $\Omega^\di_{X_G,x}$ is a free $\calO^\circ_{X_G,x}$-module with basis $\frac{dt_x}{t_x}$ where $t_x$ is a tame monomial parameter at $x$.

(iii) If $x$ is of type 4 then $\Omega^\di_{X_G,x}=\kappa^\circcirc_sdt_x$, where $t_x$ is a tame parameter at $x$, $\kappa=\kappa_G(x)=\calO_{X_G,x}$ and $s=r_{t_x}(x)^{-1}$.
\end{theorem}
\begin{proof}
For shortness, we will denote the K\"ahler seminorm simply by $|\ |$. If $x$ is of type 1 then $\hatOmega_{\calH(x)/k}=\hatOmega_{k/k}=0$, so $|\omega|_x=0$ for any $\omega\in\Omega_{X_G,x}$. It follows from the analyticity of $|\ |_\Omega$ (see \cite[\S3.3.3 and Theorem~6.1.13]{Temkintopforms}) that $|\omega|\le 1$ in a sufficiently small neighborhood of $x$, and hence $\omega\in\Omega^\di_{X_G,x}$. This proves (i).

Next, let us prove (iii). If $\omega\in\kappa^\circcirc_sdt_x$ then $|\omega|_x<1$ by Theorem~\ref{compldifth}(i) and, by the semicontinuity, $\omega\in\Omega^\di_{X_G,x}$. Conversely, assume that $\omega=fdt_x$ with $f\in\kappa$ and $|f|\ge s$. Note that $|f|$ is fixed in a neighborhood of $x$. On the other hand, $\calH(x)=\wh{k(t_x)}$ hence $t_x$ is a coordinate of a sufficiently small disc $E$ with $x\in E\subseteq X$. At any point $y$ of the interval connecting $x$ with the maximal point of the disc we have that $s^{-1}<r_{t_x}(y)$ and hence $|\omega|_y=|f|_yr_{t_x}(y)=sr_{t_x}(y)>1$ when $y$ is close enough to $x$. Thus, $\omega\notin\Omega^\di_{X_G,x}$.

It remains to prove (ii). Shrinking $X$ at $x$ we can assume that $t_x$ is defined on all of $X$ and, by Lemma \ref{tameparamlem}, is a tame parameter at every point of $X$.  By Theorem~\ref{compldifth}(i), $\left|\frac{dt_x}{t_x}\right|_y\le 1$ for any $y\in X_G$. In particular, $\frac{dt_x}{t_x}\in \Omega^\di_{X_G,x}$. Recall that $\left|\frac{dt_x}{t_x}\right|_x=1$ by Corollary~\ref{compldifcor}(i). So, if $\omega\in\Omega^\di_{X_G,x}$ then $\omega=f\frac{dt_x}{t_x}$ with $f\in\calO_{X_G,x}$ and $|f|_x\le 1$. If $x$ is type 2 or 3 then this implies that $f\in\calO^\circ_{X_G,x}$ and so $\Omega^\di_{X_G,x}=\calO^\circ_{X_G,x}\frac{dt_x}{t_x}$, as claimed.

Assume, finally, that $x$ is of type 5. It suffices to show that for any $f\in\kappa_G(x)\setminus\kappa_G(x)^\circ$ the element $f\frac{dt_x}{t_x}$ is not contained in $\Omega^\di_{X_G,x}$. Working locally we can assume that $t_x$ induces a map $g\:X\to\bfA^1_k$. Let $I$ be an open interval in the direction of $x$. Shrinking $I$ we can achieve that $g$ maps $I$ into $(0,\infty)\subset\bfA^1_k$, i.e. $t_x$ is a monomial parameter at any point of $I$. By Lemma~\ref{pmmaplem}, we can also achieve that the map $I\to(0,\infty)$ is monomial, and then the multiplicity of $g$ along $I$ is constant and equals to its multiplicity at $x$. So, $t_x$ is a tame monomial parameter along $I$. Finally, we can shrink $I$ so that $|f|_y>1$ for any point $y\in I$. Then $\left|f\frac{dt_x}{t_x}\right|_y=|f|_y>1$, and hence $f\frac{dt_x}{t_x}\notin\Omega^\di_{X_G,x}$.
\end{proof}

\subsubsection{Relation to the different}
As a corollary of the above theorem, we can relate the different function to the annihilator of an appropriate sheaf. This fact will not be used in the sequel, but it clarifies the role of the sheaf $\Omega^\di_{Y_G}$ in the study of differents. Given a torsion $\calO^\circ_{X_G}$-sheaf $\calF$ define the {\em annihilator function} $a_\calF\:X\to(0,1]$ by $$a_\calF(x)=|\Ann(\calF_x\otimes_{\calO^\circ_{X_G,x}}\calH(x)^\circ)|.$$

\begin{corollary}\label{deltacor}
Let $f$ be as in Section \ref{fsec}. Then, the sheaf $\calF=\Omega^\di_{Y_G}/f^*\Omega^\di_{X_G}$ is torsion and $\delta_f=a_\calF|_{Y^\hyp}$, where the pullback is defined by $$f^*\Omega^\di_{X_G}=f^{-1}\Omega^\di_{X_G}\otimes_{f^{-1}\calO^\circ_{X_G}}\calOcirc_{Y_G}.$$
\end{corollary}
\begin{proof}
The stalks of $\Omega^\di_{Y_G}$ are quasi-invertible by Theorem ~\ref{stalkth} and the stalks of $f^*\Omega^\di_{X_G}$ are non-zero, hence $\calF$ is torsion. Choose a point $y\in Y^\hyp$ and set $x=f(y)$. Fix tame parameters $t_x$ and $t_y$ at these points. If the points are monomial, then we can also require that the parameters are monomial and then Theorem~\ref{stalkth}(ii) implies that $\calF_y\toisom\kcirc/a\kcirc$ where $|a|=\left|\frac{dt_x}{dt_y}t_yt_x^{-1}\right|_y$. Clearly, $a_\calF(y)=|a|$, and $\delta_f(y)=\left|\frac{dt_x}{dt_y}t_yt_x^{-1}\right|_y$ by Corollary~\ref{differentcor}(ii).

If the points are of type 4, then Theorem~\ref{stalkth}(iii) implies that $\calF_y\toisom\kcirccirc/a\kcirccirc$, where $|a|=\left|\frac{dt_x}{dt_y}\right|r_{t_y}(y)r_{t_x}(y)^{-1}$. Again, $a_\calF(y)=|a|$ and it remains to recall that $\delta_f(y)=\left|\frac{dt_x}{dt_y}\right|r_{t_y}(y)r_{t_x}(y)^{-1}$ by Corollary~\ref{differentcor}(ii).
\end{proof}

\subsection{$\calOcirc_{X_G,C}$-modules}\label{calOsec}

\subsubsection{Notation}
Throughout Section \ref{calOsec} we fix a type 2 point $x\in X$, set $C=C_x$, and denote the embedding of the generic point by $i\:x\into C$. By a {\em distinguished} parameter at a point $v\in C$ we mean a tame monomial parameter at $v$ such that $|t_v|_x=1$ and $\slope_v(|t_v|)=1$.

For any sheaf $\calF$ on $C$ we will use the notation $\calM\calF=i_*\calF_x$. In particular, $\calM_C=\calM\calO_C$ is the sheaf of meromorphic functions and $\calM\Omega_{C/\tilk}$ is the sheaf of meromorphic differentials.

\subsubsection{Restriction onto $C$}
For any sheaf $\calF$ on $X_G$ we denote by $\calF_C$ its restriction onto $C$ via the topological embedding $C\into X_G$. For example, $\calO^\circ_{X_G,C}$ is the restriction of $\calO_{X_G}^\circ$. Although we do not introduce a sheaf $\calMcirc_{X_G}$, we will use the notation $\calMcirc_{X_G,C}=\calM\calOcirc_{X_G,C}$ to denote the sheaf of ``meromorphic functions of $\calOcirc_{X_G,C}$''. It is the constant sheaf associated with $\calOcirc_{X_G,x}$.

\subsubsection{Reduction}
For any $\calOcirc_{X_G,C}$-module $\calG$ we define its {\em reduction} as $\tilcalG=\calG\otimes_{\kcirc}\tilk$. For example, the reductions of $\calOcirc_{X_G,C}$ and $\calMcirc_{X_G,C}$ are canonically isomorphic to $\calO_C$ and $\calM_C$, respectively. In general, $\tilcalG$ is an $\calO_C$-module.

\subsubsection{Twists}
Assume that $D=\sum_{v\in C}n_v v$ is a formal linear combination of closed points of $C$ such that almost all coefficients are non-negative. Then the twist $\calO_C(D)$ is the quasi-coherent submodule of the sheaf of meromorphic functions $\calM_C$ whose sections on an open $U$ satisfy $\ord_P(f)\geq -n_v$ for any $P\in U$. In particular, the stalk at $P$ is $\tilt_v^{-n_v}\calO_{C,P}$.    For any $\calO_C$-module $\calF$ we define $\calF(D)=\calF\otimes_\calO\calO(D)$. The opposite twist $\calF(-D)$ may be not defined, but if $\calF\toisom\calG(D)$ for an $\calO_C$-module $\calG$ then $\calG$ is unique up to a canonical isomorphism and we will use the notation $\calG=\calF(-D)$. In fact, we will need all this in the single case when $D=\sum_{v\in C}v$, and then we will simply write $\calG=\calF(-C)$.

A similar theory of twists exists for $\calOcirc_{X_G,C}$-modules, where $\calOcirc_{X_G,C}(D)$ is defined as the subsheaf of $\calMcirc_{X_G,C}$ whose stalk at $v$ equals to $t_v^{-n_v}\calOcirc_{X_G,v}$. Plainly, twists are compatible with the reduction, i.e. $\tilcalF(D)=\wt{\calF(D)}$.

\subsubsection{Pullbacks}
Assume that $f\:Y\to X$ is as in Section~\ref{fsec} and $y\in f^{-1}(x)$, and let $h\:C_y\to C_x$ be the induced map between the germ reductions. For any $\calOcirc_{X_G,C}$-module $\calF$ we define its {\em pullback} by $$h^*\calF=h^{-1}\calF\otimes_{h^{-1}\calOcirc_{X_G,C_x}}\calOcirc_{Y_G,C_y}.$$ Plainly, this operation is compatible with the reduction (which is also defined by a tensor product), namely, the following result holds.

\begin{lemma}\label{pullbacklem}
Keep the above notation, then $\wt{h^*\calF}=h^*\tilcalF$.
\end{lemma}

\subsection{Local Riemann-Hurwitz formula}

\subsubsection{Reduction of $\Omega^\di_{X_G,C_x}$}
The reduction of $\Omega^\di_{X_G,C_x}$ is a huge quasi-coherent sheaf, so it is more convenient to work with an appropriate twist.

\begin{lemma}\label{red2lem}
Assume that $X$ is a nice $k$-analytic curve and $x\in X$ is a type 2 point. Then $\Omega^\di_{X_G,C_x}(-C_x)$ exists and is a locally free $\calOcirc_{X_G,C_x}$-module whose reduction is isomorphic to $\Omega_{C_x/\tilk}$.
\end{lemma}
\begin{proof}
Set $C=C_x$ for shortness. Let $\calF$ be the subsheaf of $\Omega^\di_{X_G,C}$ such that for any open $U\subseteq C$ the module $\calF(U)$ consists of all elements $\phi\in\Omega^\di_{X_G,C}(U)$ such that $\phi_v\in t_v\Omega^\di_{X_G,v}$ for any $v\in U$. We claim that for each $v\in C$ the inclusion $\calF_v\subseteq t_v\Omega^\di_{X_G,v}$ is an equality and hence $\calF=\Omega^\di_{X_G,C}(-C)$. Indeed, by Theorem~\ref{stalkth}(ii) each $t_v\Omega^\di_{X_G,v}$ is a free module with basis $dt_v$, where $t_v$ is a distinguished parameter at $v$. For each $v$ we have that $\tilt_v$ is a local parameter on $C$ at $v$, hence $d\tilt_v$ is a generator of $\Omega_{C/\tilk,v}$, and there exists a neighborhood $U_v$ of $v$ in $C$ such that $\Omega_{C/\tilk}(U_v)$ is a free module generated by $d\tilt_v$. In particular, for each point $u\in U_v$ the element $\tilt_v-\tilt_v(u)$ is a local parameter at $u_v$. It follows that $t_v-a_u$ is a distinguished parameter at $u$, where $a_u\in\kcirc$ is a lifting of $\tilt_v(u)$. In particular, $dt_v=d(t_v-a_u)$ is a generator of $\calF_u$ for any $u\in U_v$, and hence $dt_v\in\calF(U_v)$ and $\calF_v$ is as required.

Moreover, we have proved above that each $\calF(U_v)$ is a free module with basis $dt_v$. So, sending $dt_v$ to $d\tilt_v$ we obtain an isomorphism $h_v\:\calF(U_v)\otimes_{\kcirc}\tilk\toisom\Omega_{C/\tilk}(U_v)$ and we claim that this globalizes to an isomorphism $h$ between the reduction of $\calF$ and $\Omega_{C/\tilk}$. We should only check that isomorphisms $h_v$ and $h_u$ are compatible on $U_u\cap U_v$. In other words, we should check that the reduction of an element $\frac{dt_v}{dt_u}\in\calOcirc_{X_G,x}$ equals to the element $\frac{d\tilt_v}{d\tilt_u}\in\wHx$.

Recall that $\hatOmega_{\calH(x)^\circ/\kcirc}$ is the unit ball of $\hatOmega_{\calH(x)/k}$ by Corollary~\ref{compldifcor}(ii). In particular, the map $\phi_x\:\Omega_{X_G,x}\to\hatOmega_{\calH(x)/k}$ from Section~\ref{fibdifsec} restricts to a map $\Omega^\di_{X_G,x}\to\hatOmega_{\calH(x)^\circ/\kcirc}$, and the commutativity of the diagram in Section~\ref{fibdifsec} implies that the left square in the following diagram is commutative
$$
\xymatrix{
\calOcirc_{X_G,x}\ar[r]\ar[d]^{d_1} & \calH(x)^\circ\ar[r]\ar[d]^d & \wHx\ar[d]^{d_2}\\
\Omega^\di_{X_G,x}\ar[r] & \hatOmega_{\calH(x)^\circ/\kcirc}\ar[r] & \Omega_{\wHx/\tilk}.
}
$$
The right square is obviously commutative and so the differentials $d_1$ and $d_2$ are compatible, as claimed.
\end{proof}

\subsubsection{Local Riemann-Hurwitz}
Now we are in a position to prove the following result.

\begin{theorem}\label{prop:local_RH}
Assume that $f\colon Y\to X$ is as in Section \ref{fsec}, $y\in\Int(Y)$ and $x=f(y)\in\Int(X)$ are inner type 2 points, and $h\colon C_y\to C_x$ is the corresponding map on the germ reductions at the points $x$ and $y$. Then,
$$
2g(y)-2 = n(2g(x) - 2) + \sum_{v\in C_y} (-\slope_v\delta_f + n_v - 1)
$$
where $g(x)$ and $g(y)$ are the genera of the curves $C_x$ and $C_y$, respectively, $n=\deg h$, and $n_v$ is the ramification index of $h$ at $v\in C_y$.
\end{theorem}
\begin{proof}
Since the points are inner, the residue curves $C_y$ and $C_x$ are proper. Set $\calG=\Omega^\di_{Y_G,C_y}(-C_y)$ and $\calF=\Omega^\di_{X_G,C_x}(-C_x)$, so that $\tilcalG\toisom\Omega_{C_y/\tilk}$ and $\tilcalF\toisom\Omega_{C_x/\tilk}$ by Lemma~\ref{red2lem}. Fix an element $a\in k$ such that $|a|=\delta_f(y)^{-1}$ and consider the $\calO_{X_G,C_y}^\circ$-submodule $\calE=ah^*\calF$ of $\Omega_{Y_G}$. Since $\calE\toisom h^*\calF$, Lemma~\ref{pullbacklem} tells us that $\tilcalE\toisom h^*\Omega_{C_x/\tilk}$.

Choose tame monomial parameters $t_x$ and $t_y$ at $x$ and $y$, respectively, such that $|t_x|=|t_y|=1$. Then $r_{t_y}(y)=r_{t_x}(y)=1$, $adt_x$ generates $\calE_y$ and $dt_y$ generates $\calG_y$. Note that $\left|a\frac{dt_x}{dt_y}\right|_y=1$ by Theorem~\ref{compdeltalem}, hence $\calE_y=\calG_y$ and we can view both $\tilcalG$ and $\tilcalE$ as subsheaves of $\calM\tilcalG_y=\calM\tilcalE_y$. Then the index $(\tilcalG:\tilcalE)_v\in\bfZ$ makes sense for any $v\in C_y$ and we have the global degree formula
$$\sum_{v\in C_y}(\tilcalG:\tilcalE)_v=\deg(\tilcalG)-\deg(\tilcalE)=2g(y)-2 - n(2g(x) - 2).$$

To complete the proof it suffices to show that $(\tilcalG:\tilcalE)_v=-\slope_v\delta_f + n_v - 1$. This is a local question at $v$, so fix distinguished parameters $t_v$ at $v$ and $t_u$ at $u=f(v)$. Let $I\subset Y_\mon$ be an interval starting at $y$ in the direction of $v$. Shrinking $I$ we can achieve that $t_v$ is a tame monomial parameter for any point in $I$ and $t_u$ is a tame monomial parameter at any point of $f(I)$. In particular, if $g=\frac{dt_u}{dt_v}$ then by Theorem~\ref{compdeltalem} we obtain that $\delta_f(z)=|gt_vt_u^{-1}|_z$ for any point $z\in I$, and hence $$-\slope_v\delta_f+n_v-1=\nu_v(g^{-1}t^{-1}_vt_u)+n_v-1=\nu_v(g^{-1})=\nu_v(ag^{-1}).$$

It remains to note that $dt_v$ is a basis of $\calG_v$ and $adt_u$ is a basis of $\calE_v$, and so $\nu_v$ of their ratio $\frac{dt_v}{adt_u}=ag^{-1}$ equals to $(\tilcalG:\tilcalE)_v$.
\end{proof}

\subsubsection{The differential indices $R_y$}\label{Ry}
The entries of the local Riemann-Hurwitz formula will show up again and again throughout the paper, so it makes sense to introduce special notation. For any branch $v$ we define the {\em differential slope index} $$S_{v,f}=-\slope_v\delta_f+n_v-1.$$ Since $f$ is usually fixed, we will simple denote it by $S_v=S_{v,f}$. Next, we define a characteristic function $\chi_f\:Y\to\bfN$ by $$\chi_f(y)=2g(y)-2-n_y(2g(x)-2),$$ where $x=f(y)$. Note that excluding a finite set of type 2 points, we have that $g(y)=0$ and hence $\chi_f(y)=2n_y-2$. Finally, for any point $y\in Y$ we define the {\em differential index} $$R_y=\chi_f(y)-\sum_{v\in\Br(y)}S_v.$$

\begin{remark}
We will later see that $\sum_{y\in Y}R_y$ relates the genera of $Y$ and $X$, so let us discuss when these indices do not vanish.

(0) At any unibranch point $y$ we have that $R_y=\slope_y\delta_f+n_y-1$.

(1) We will prove in Theorem~\ref{limth} that $R_y$ is the classical differential index for a type 1 point $y$. In particular, $R_y\ge 0$ and the equality takes place if and only if $y$ is not a ramification point.

(2) Assume that $y$ is of type 2. The local Riemann-Hurwitz formula states that if $y$ is inner then $R_y=0$, so if $R_y\neq 0$ then $y\in\partial(Y)$. In this case, $R_y$ can be negative (depending on the indices at the ``missing branches'').

(3) If $y$ is of type 3 then $y$ is inner because $Y$ is strict, hence $\Br(y)=\{u,v\}$, $\chi_f(y)=2n_y-2$, $n_v=n_u=n_y$ and the numbers $s_u=\slope_u\delta_f$ and $s_v=\slope_v\delta_f$ are opposite. Thus, $$R_y=2n_y-2-(-s_v+n_v-1)-(-s_u+n_u-1)=0.$$

(4) It follows from Theorem~\ref{deltatrivth} below that $R_y=0$ for any type 4 point.
\end{remark}

\subsection{Behaviour at type 1 points}\label{limsec}
We conclude Section~\ref{difsec} with studying the local behaviour of $\delta_f$ at type 1 points.

\subsubsection{Algebraic different}
Assume that $y\in Y$ is of type 1 and $x=f(y)$. Since $f$ is generically \'etale, $\Omega_{Y/X,y}$ is a torsion $\calO_{Y,y}$-module of a finite length $l$, and we set $\delta_{y/x}=l$. It follows from GAGA that if $Y\to X$ is the analytification of a morphism of algebraic $k$-curves then $\delta_{y/x}$ equals to the value of the classical (additive) different of $\calO_y/\calO_x$. Furthermore, the usual argument (e.g., from \cite[IV.2.2(b)]{Hartshorne}) shows that the different $\delta_{y/x}$ can be computed analogously to the formula in Corollary~\ref{differentcor}, but using the discrete valuation $\nu_y$ of $\calO_{Y,y}$.

\begin{lemma}\label{type1dif}
Keep the above notation and choose parameters $t_y$ and $t_x$ at $y$ and $x$, respectively. Then $\delta_{y/x}=\nu_y(\frac{dt_x}{dt_y})$.
\end{lemma}
\begin{proof}
This follows from the fact that $\Omega_{Y/X,y}=\calO_{Y,y}dt_x/\calO_{Y,y}dt_y$.
\end{proof}

\subsubsection{The limit formula}
Now, we can establish the limit formula for $\delta_f$. In particular, it shows that our definition of $R_y$ at type 1 points agrees with the differential ramification index used in the algebraic Riemann-Hurwitz formula.

\begin{theorem}\label{limth}
Assume that $f$ is as in Section \ref{fsec}, $y\in Y$ is a type 1 point and $x=f(y)$. Then
$\slope_y\delta_f=\delta_{y/x}-n_y+1$, or, equivalently, $R_y=\delta_{y/x}$. Moreover, let $I\subset Y$ be an interval starting at $y$, then

(a) If $\cha(k)>0$ then there exists a radius parametrization $r\:I\to[0,a]$ such that $\delta_f(z)=r(z)^{\delta_{y/x}-n_y+1}$ for any point $z\in I\setminus\{y\}$ close enough to $y$. In particular, $\lim_{z\to y}\delta_f(z)=0$ if and only if the ramification at $y$ is wild, and otherwise $\delta_f=1$ near $y$.

(b) If $\cha(k)=0$ then $\delta_f(z)=|n_y|$ on a small enough neighborhood of $y$ in $I$. In particular, $\delta_f(z)<1$ near $y$ if and only if the ramification is topologically wild, and the value of $\delta_f(z)$ near $y$ is the minimal possible for multiplicity $n_y$ (see Theorem~\ref{restrictth}(i)).
\end{theorem}
\begin{proof}
Choose parameters $t_y\in m_y\setminus m_y^2$ and $t_x\in m_x\setminus m_x^2$ and parameterize $I$ by $r(z)=|t_y|_z$. Then $\delta_{y/x}=\nu_y(h)$ for $h=\frac{dt_x}{dt_y}$, and for any point $z\in I$ close enough to $y$ we have that $|t_x|_z=ar(z)^{n_y}$ and $|h|_z=br(z)^{\delta_{y/x}}$ for some $a,b\in|k^\times|$. So, by Theorem~\ref{compdeltalem} we obtain that
\begin{equation}\label{deltaeq}
\delta_f=|ht_yt_x^{-1}|_z=a^{-1}br(z)^{\delta_{y/x}-n_y+1}.
\end{equation}
In particular, $\slope_y\delta_f=\delta_{y/x}-n_y+1$ and $R_y=2n_y-2-(-\slope_y\delta_f+n_y-1)=\delta_{y/x}$.

By the classical theory, $\delta_{y/x}-n_y+1\ge 0$ and the equality holds only in the tame case. So, $\delta_f$ vanishes at $y$ if and only if $y$ is a wild ramification point. In this case $\cha(k)>0$ and the order of zero is as asserted in (a). To complete the wild case it remains to get rid of the constant term, so we re-scale the radius function as $r'=(a^{-1}b)^{1/(\delta_{y/x}-n_y+1)}r$.

Assume now that $y$ is a tame ramification point, and so $n_y$ is invertible in $k$. Let $t_y^{n_y}+\sum_{i=0}^{n_y-1}a_it_y^i$ be the minimal polynomial of $t_y$ over $\calO_{X,x}$. Note that $\nu_x(a_i)\ge 1$ and $\nu_x(a_0)=1$, and so $a_0$ is a parameter at $x$ and we can replace $t_x$ with $a_0$. Then $t_x\in -t_y^{n_y}+t_y^{n_y+1}\calO_{Y,y}$ and hence $|t_x|=|t_y|^{n_y}$ on a small enough neighborhood of $y$ in $I$. In addition, $$n_yt_y^{n_y-1}dt_y+hdt_y+\sum_{i=1}^{n_y-1}ia_it_y^{i-1}dt_y=0$$ and all terms, except the first two, are of order at least $\nu_y(t_x)=n_y$ at $y$. It follows that $h\in -n_yt_y^{n_y-1}+t_y^{n_y}\calO_{Y,y}$ and hence $|h|=|n_yt_y^{n_y-1}|$ near $y$ on $I$. So, $a=1$ and $b=|n_y|$ in (\ref{deltaeq}) and we are done.
\end{proof}

\subsubsection{The log different function}\label{logdiffun}
Using Theorem~\ref{limth} we can extend $\delta_f$ to a piecewise monomial function $\delta_f^\rmlog\:Y_G\to[0,1]$ which has zeros at wild ramification points. We call the latter function the {\em log different function} because its zero at a type 1 point $y$ is of order $\delta_{y/x}-n_y+1$ rather than $\delta_{y/x}$.

\subsection{Aside on log differentials}
The reader might have noticed that log differentials showed up even before we introduced $\delta_f^\rmlog$. Indeed, we saw in Lemma~\ref{red2lem} that the reduction of $\Omega^\di_{X_G,C_c}$ is not $\Omega_{C_x/\tilk}$, as one might expect, but its huge twist $\Omega_{C_x/\tilk}(C_x)$, which is nothing else but the sheaf of log differentials of $C_x$. We conclude Section~\ref{difsec} with a brief explanation of the role of log differentials that was somewhat implicit throughout the section. This will not be used, so the uninterested reader can skip to Section~\ref{combsec}.

\subsubsection{Log different}
Given a finite separable extension of real-valued fields $L/K$, the log different $\delta^\rmlog_{L/K}$ is defined analogously to the usual different but using the module of logarithmic differentials $\Omega^\rmlog_{\Lcirc/\Kcirc}=\Omega_{(\Lcirc,\Lcirc\setminus\{0\})/(\Kcirc,\Kcirc\setminus\{0\})}$ instead of $\Omega_{\Lcirc/\Kcirc}$.

\subsubsection{Relation to the different}
For a real-valued field $K$ set $\lam_K=\sup_{\pi\in\Kcirccirc}|\pi|$. So, $\lam_K$ is the absolute value of a uniformizer $\pi_K$ if the valuation is discrete, and $\lam_K=1$ otherwise. Then $\delta^\rmlog_{L/K}=\delta_{L/K}\lam_L\lam_K^{-1}$ by \cite[Theorem~5.4.9(i)]{Temkintopforms}. Equivalently, $\delta^\rmlog_{L/K}=\delta_{L/K}$ if the valuation on $K$ is not discrete, and $\delta^\rmlog_{L/K}=\delta_{L/K}|\pi_L|^{1-e}$ if $K$ is discretely valued and $e=e_{L/K}$. In particular, this explains the formula for the additive log different we gave in \ref{logdiffun}.

\subsubsection{The log different function}
Since we work over $k$, which is not discretely valued, $\delta^\rmlog_{L/K}=\delta_{L/K}$ for any extension $L/K$ of analytic $k$-fields. In particular, the different function $\delta_f$ on $Y^\hyp$ can be also interpreted as the log different function $\delta^\rmlog_f$. We do feel the difference between the two notions when discrete valuations are used, and this happens at type 1 and type 5 points. In the first case, the continuation to type 1 points is related to the log different of the extension of their local rings, and in the second case, the slope of $\delta_f$ at a type 5 point is related to the log different of the corresponding discrete valuation on the reduction curve.

\begin{remark}
(i) The above discussion shows that it is much more natural to interpret $\delta_f$ as $\delta^\rmlog_f$. However, we preferred to work with the more classical object, the different, to avoid any use of log geometry.

(ii) Another indication of the relevance of log differentials is obtained when the ground field $k$ is discretely or trivially valued. In this case, the discrepancy between $\Omega_{\Lcirc/\Kcirc}$ and $\Omega^\rmlog_{\Lcirc/\Kcirc}$ is not negligible, and it turns out that this is $\Omega^\rmlog_{\Lcirc/\Kcirc}$ that induces the K\"ahler seminorm on $\Omega_{L/K}$ (see \cite[Theorem~5.1.8]{Temkintopforms}).

(iii) For analytic $k$-fields, the modules $\Omega_{\Lcirc/\Kcirc}$ and $\Omega^\rmlog_{\Lcirc/\Kcirc}$ are almost isomorphic, but $\Omega^\rmlog_{\Lcirc/\Kcirc}$ is still more convenient to work with. For example, if $L=\calH(y)$ for a type 3 point and $t$ is a tame monomial parameter then $\Omega^\rmlog_{\Lcirc/\kcirc}$ is a free module with basis $\frac{dt}{t}$, while $\Omega_{\Lcirc/\Kcirc}$ is isomorphic to $\Lcirccirc$.
\end{remark}

\section{Combinatorial Riemann-Hurwitz formula}\label{combsec}

\subsection{Genus graphs}\label{genusgraphsec}

\subsubsection{Combinatorial graphs}
Throughout section \ref{combsec}, a {\em graph} $\Gamma$ means a combinatorial graph $(V,E)$ that may contain loops (i.e. edges whose both endpoints coincide), where $V$ is the set of vertices and $E$ is the set of edges. We will only consider finite graphs. A morphism of graphs $\varphi\:\Gamma'\to\Gamma$ is a pair of maps $\varphi_E\:E'\to E$ and $\varphi_V\:V'\to V$ compatible with the incidence relation.

\subsubsection{Oriented edges and functions}
The set of oriented edges of a graph $\Gamma$ will be denoted by $E^{or}$. If $e$ is an oriented edge from $x$ to $y$, we will write $x\prec e$ and $e\prec y$ for short, and denote the opposite edge by $-e$. By an \emph{oriented function} on $\Gamma$ we mean a function $f\colon E^{or} \to \bfZ$ such that $f(-e) = -f(e)$ for any $e\in E^{or}$.

\subsubsection{Branches}
If $v\prec e$ then we say that $e$ is a {\em branch} at $v$ and the set of all branches at $v$ is denoted by $\Br(v)$. Any morphism of graphs $\varphi\:\Gamma\to\Gamma'$ induces the maps $\Br(v)\to\Br(\varphi(v))$ for $v\in V$.

\subsubsection{Proper morphisms}
An {\em edge-weighted morphism} or an {\em $n$-morphism} of graphs consists of a morphism $\varphi\:\Gamma\to\Gamma'$ and a {\em multiplicity function} $n\:E\to\bfZ_{>0}$. We also view $n$ as a multiplicity function on the set of branches of $\Gamma$. An $n$-morphism $(\varphi,n)$ has a locally constant multiplicity at a vertex $v\in V$ if for any choice of $e'\in\Br(f(v))$, the sum of the multiplicities $n_e$ with $e\in\Br(v)$ and $f(e)=e'$ is independent of $e'$. In such case, this sum is called the {\em multiplicity} of $\varphi$ at $v$ and denoted $n_v$.

We say that an $n$-morphism $(\varphi,n)$ of connected graphs is {\em proper} if it has a locally constant multiplicity at all vertices and, in addition, it has a constant global rank, i.e. for any $v'\in V'$, the number $\sum_{v\in\varphi^{-1}(v')} n_v$ does not depend on $v$. The latter number will be called the \emph{degree} of $\varphi$. We will not need this, but an $n$-morphism of non-connected graphs is proper if it restricts to proper morphisms on the connected components.

\subsubsection{Formal divisors}
A {\em divisor} on a graph $\Gamma$ is a formal sum $\sum_{v\in V} c_v v$, where $c_v\in \bfZ$. The \emph{degree} of a divisor $D = \sum_{v\in V} c_v v$ is $\deg D=\sum_{v\in V}c_v$. For a proper $n$-morphism of graphs $\varphi\colon \Gamma'\to \Gamma$ we define the {\em pullback} $\varphi^*D=\sum_{v'\in V} c_{\varphi(v')}n_{v'} v'$. Then $\deg(\varphi^* D) = \deg \varphi \deg D$ in the obvious way.

\subsubsection{Genus graphs}
A \emph{genus graph} is a finite connected graph $\Gamma$ together with a genus function $g\colon V\to\bfN$ that associates to any vertex $v$ its genus $g(v)$. We then define the genus of $\Gamma$ to be $g(\Gamma)=h^1(\Gamma)+\sum_{v\in V} g(v)$, where $h^1(\Gamma)=|E| - |V| + 1$ is the number of loops of $\Gamma$. The following example is our main motivation for introducing genus graphs.

\begin{rem}\label{genusrem}
(i) We have defined in \ref{typessec} a genus function on any nice compact curve $C$, hence any topological finite subgraph of a nice compact curve gives rise to a genus graph.

(ii) Assume that $Z$ is a connected nodal curve over $\tilk$. Then it is customary to consider the graph $\Gamma_Z$ whose vertices correspond to the irreducible components of $Z$ and whose edges correspond to the nodes. Assigning to a vertex the genus of the corresponding component, we obtain a genus graph. Any finite morphism of constant rank between connected nodal curves gives rise to a proper morphism of the corresponding graphs.

(iii) The examples of (i) and (ii) are related as follows. If $\gtX$ is a connected semistable $\kcirc$-curve with generic fiber $X$ and closed fiber $Z$ then the topological realization of $\Gamma_Z$ can be identified with a skeleton $\Gamma\subset X$ and $\Gamma_Z$ is the genus graph corresponding to $\Gamma$ via (i).
\end{rem}

\subsubsection{Canonical divisors}
Following \cite[Section~2]{ABBR}, we define the \emph{canonical divisor} on a genus graph $\Gamma$ as $K_\Gamma = \sum_{v\in V} (\val v + 2g(v) - 2) v$, where $\val v = |\Br(v)|$ is the {\em valency} of $v$. It is designed to mimic the usual canonical divisor. In particular, $\deg K_\Gamma = 2 g(\Gamma) - 2$ because $\sum_{v\in V}\val (v) = 2 |E|$.

\subsubsection{$\delta$-morphisms}
So far, our definitions were more or less analogous to those of \cite{ABBR}, though we use different terminology. Now, we are going to add a combinatorial datum corresponding to slopes of the different. It is not related to maps of nodal curves (unless an additional structure is specified), but, as we will later see, such a structure naturally arises on a simultaneous skeleton of a map between nice compact curves.

By a {\em $\delta$-morphism} between (genus) graphs we mean a triple $(\varphi,n,s\delta)$, where $(\varphi,n)\:\Gamma\to\Gamma'$ is a proper morphism of graphs and $s\delta$ is an oriented function on $\Gamma$. Intuitively, the latter can be thought off as the slope of the different along the edges, though the different function itself is not defined in this context. In particular, we will abuse notation by writing $s_e\delta$ instead of $s\delta(e)$.

\subsubsection{The ramification divisor}
The following definitions are analogous to those of Section \ref{Ry}. Assume that $(\varphi,n,s\delta)\colon \Gamma\to \Gamma'$ is a $\delta$-morphism between genus graphs. For any edge $e\in E$ $$S_{e} = -s_{e}\delta + n_{e} - 1$$ is called the \emph{differential slope index}. For any vertex $v\in V$ with $v'=\varphi(v)$ we set $$\chi(v)=2g(v)-2-n_v(2g(v')-2)$$ and define the {\em differential index} to be $$R_v=\chi(v)- \sum_{e\in \Br(v)} S_{e}.$$ The {\em ramification divisor} of the $\delta$-morphism is $R_\varphi = \sum_{v\in V} R_{v} v$.

\subsubsection{Riemann-Hurwitz for $\delta$-morphisms}
In the following result, $\Delta_\varphi = \sum_{v\in V} \Delta_{v} v$, where $\Delta_{v} = \sum_{e\in \Br(v)} -s_e\delta$.

\begin{theorem}\label{combRH}
Let $(\varphi,n,s\delta)\colon \Gamma \to \Gamma'$ be a $\delta$-morphism of genus graphs. Then,

(i) $K_\Gamma=\varphi^*(K_{\Gamma'}) + R_\varphi + \Delta_\varphi$,

(ii) $2 g(\Gamma) - 2 = \deg \varphi(2 g(\Gamma') - 2) + \sum_{v\in V}R_{v}$.
\end{theorem}
\begin{proof}
The proof of (i) reduces to comparing the coefficients of $v\in V$ on the both sides of the equality. Set $u=f(v)$ and let $l_v$ and $r_v$ be the coefficients of $v$ on the left side and on the right side, respectively. Then
%\begin{flalign*}
$$r_v=n_v(\val(u)+2g(u)-2)+R_v+\Delta_v=$$
$$n_v(\val(u)+2g(u)-2)+2g(v) - 2 - n_v(2g(u) -2)-\sum_{e\in \Br(v)}(s_e\delta+S_e)=$$
$$2g(v)-2+n_v\val(u)-\sum_{e\in \Br(v)}(n_e-1)=$$
$$2g(v)-2+\val(v)+n_v\val(u)-\sum_{e\in \Br(v)}n_e=2g(v)-2+\val(v)=l_v.$$

Note that $\deg \Delta_\varphi = 0$ because $s\delta$ is an oriented function. So, (ii) is obtained from (i) by comparing the degrees.
\end{proof}

\begin{rem}
The above Riemann-Hurwitz formula is essentially the formula \cite[2.14.2]{ABBR}. Although our ramification divisor is defined differently, so that $s\delta$ is taken into account, this difference is cancelled out in the Riemann-Hurwitz formula because $\deg\Delta_\varphi=0$.
\end{rem}

\subsubsection{Balanced vertices}
Given a $\delta$-morphism $\varphi$ as above, we say that a vertex $v\in V$ is {\em balanced} if $R_v=0$. Only non-balanced vertices contribute to the Riemann-Hurwitz formula. In our applications, the only non-balanced vertices will come from the ramification points and from the boundary components (i.e. non-proper $\tilk$-curves).

\subsection{Stability}

\subsubsection{Contractions of genus graphs}
By a {\em contraction} of a genus graph $\Gamma$ we mean an operation of one of the following two types:

(1) If $v$ is a leaf of genus 0 and $e$ is the edge incident to $v$ then one can remove $v$ and $e$ from $\Gamma$.

(2) If $v$ is a vertex of genus 0 and valence 2 and $v$ is not the only vertex of $\Gamma$, then one can remove $v$ and replace the two edges with endpoint $v$ by a single edge.

\begin{remark}
(i) Combinatorial contractions correspond to blowing down unstable rational components in the closed fiber of a semistable $\kcirc$-curve. Equivalently, such a contraction corresponds to an operation of decreasing a (non-minimal) skeleton of a nice compact curve.

(ii) Any contraction preserves both the topological type of $\Gamma$ and the set of positive genus vertices; in particular, it preserves the genus of $\Gamma$.
\end{remark}

\subsubsection{Stable genus graph}
A genus graph $\Gamma$ is called {\em stable} if it does not admit contractions.

\begin{remark}
It is a simple classical fact that if $g(\Gamma)>1$ then the stable graph $\Gamma'$ obtained from $\Gamma$ by a series of contractions is essentially unique (e.g., it has the same set of vertices $V'\subseteq V$).
\end{remark}

\subsubsection{Contractions of $\delta$-morphisms}
Assume that $\varphi\:\Gamma\to\Gamma'$ is a $\delta$-morphism of genus graphs. By a {\em contraction} of $\varphi$ we mean an operation of one of the following two types:

(1) Assume that $v'$ is a leaf of genus zero with edge $e'$, such that any $v\in\varphi^{-1}(v')$ is a leaf satisfying $g(v)=R_v=0$. Then one can remove $v'$ and $e'$ from $\Gamma'$, and $\varphi^{-1}(v')$ and $\varphi^{-1}(e')$ from $\Gamma$.

(2) Assume that $v'$ is a vertex of genus 0 and valence 2, such that $v'$ is not the only vertex of $\Gamma'$ and any vertex $v\in\varphi^{-1}(v')$ satisfies $g(v)=R_v=0$ and $\val(v)=2$. Then we can remove $v'$ from $\Gamma'$ replacing its two edges with a single edge, and do the same operation with all vertices of $\varphi^{-1}(v')$.

\begin{remark}
Contractions preserve the topological types and the positive genus sets of both $\Gamma$ and $\Gamma'$. In addition, they preserve all unbalanced vertices of $\Gamma$, hence the Riemann-Hurwitz formulas for $\varphi$ and its contraction are essentially the same.
\end{remark}

\subsubsection{Stable $\delta$-morphisms}
Similarly to the absolute case, we say that a $\delta$-morphism is {\em stable} if it cannot be contracted.

\subsection{Classification of stable $\delta$-morphisms of degree 2 and genus $1\mapsto 0$}\label{classsec}
We will later describe analytic morphisms $E\to\bfP^1_k$ of degree two with $E$ an elliptic curve. In this section, we study the combinatorial part of the problem.

\subsubsection{Tame and wild vertices}
Assume that $\varphi\:\Gamma\to\Gamma'$ is a $\delta$-morphism. We say that a vertex $v\in V$ is {\em tame} if $s_e\delta=0$ for any edge $e\in\Br(v)$. Any other vertex is called {\em wild}.

\subsubsection{Special $\delta$-morphisms}\label{specdelta}
A $\delta$-morphism $\varphi\:\Gamma\to\Gamma'$ will be called {\em special} if the following conditions are satisfied:
\begin{itemize}
\item[(1)] $\varphi$ is stable, $\deg(\varphi)=2$, $g(\Gamma)=1$ and $g(\Gamma')=0$.

\item[(2)] If $R_v\neq 0$ for a vertex $v$ of $\Gamma$ then $v$ is a leaf, $g(v)=0$, $R_v>0$ and $n_v=2$.

\item[(3)] One of the following three possibilities holds:

\begin{itemize}
\item[(T)] Tame case: all vertices are tame (i.e. $s\delta$ vanishes identically).

\item[(M)] Mixed case: any vertex with $R_v\neq 0$ is tame, but there also exist wild vertices.

\item[(W)] Wild case: any vertex with $R_v\neq 0$ is wild.
\end{itemize}

\item[(4)] If an edge $e$ of $\Gamma$ splits (i.e. $n_e=1$) then $s_e\delta=0$.

\item[(5)] If $s_e\delta\neq 0$ for an edge of $\Gamma$ then $s_e\delta$ is odd.
\end{itemize}
A special $\delta$-morphism models the minimal skeleton of a morphism of proper curves; in particular, there are no boundary points, and this explains why $R_v\ge 0$ in (2). The meaning of condition (5) is explained by Remark~\ref{restrictrem}. The trichotomy of (3) corresponds to the trichotomy of the characteristics of $k$ and $\tilk$. For the ramification points, its meaning is clear: they are wild when $\cha(k)=2$ and tame otherwise. Also, it is clear that everything is tame when $\cha(\tilk)\neq 2$.

In fact, in our case we will see that all vertices are wild in the wild case and all vertices with $R_v=0$ are wild in the mixed case, but this is an artefact of a relatively small classification that we are going to establish; in particular, this does not generalize to larger genera. Our goal in Section~\ref{classsec} is to classify all special $\delta$-morphisms.

\subsubsection{Ramification points}
Fix a special $\delta$-morphism $\varphi\:\Gamma\to\Gamma'$. By a {\em ramification point} we mean any vertex $v\in\Gamma$ with $R_v\neq 0$. Since $v$ is a leaf, there is a single oriented edge $e$ starting at $v$ and we set $S_v=S_e$ for simplicity.

The set of all ramification points will be denoted $\Ram(\varphi)$. Since $n_v=2$ and $g(v)=0$ for $v\in\Ram(\varphi)$, we have that $S_v=1-s_e\delta$ and $R_v=2-S_v=1+s_e\delta$.

\begin{lemma}\label{ramlem}
If $\varphi$ is a special $\delta$-morphism then one of the following possibilities holds:

(i) There is one ramification point $v$ and $R_v=4$.

(ii) There are two ramification points having $R_v=2$.

(iii) There are four ramification points having $R_v=1$.

Cases (i) and (ii) happen in the wild case, and case (iii) occurs in the mixed and tame cases.
\end{lemma}
\begin{proof}
By Theorem~\ref{combRH}, the sum of all $R_v$'s equals to $2g(\Gamma)-2-2(2g(\Gamma')-2)=4$. This makes the claim obvious.
\end{proof}

\subsubsection{Root subtrees}
By a {\em root subtree} of $\Gamma$ we mean a subtree $T\subseteq\Gamma$ with a special leaf $r$ such that: (a) $r$ is not a ramification point, (b) if $v\neq r$ is in $T$ then $g(v)=0$ and all edges of $v$ are in $T$. Note that by saying that $r$ is a leaf we assume that its valence is 1, and so $T\neq r$. We call $r$ the {\em root} of $T$ while ``leaves" will refer to other leaves only. If $e$ is the oriented edge of $T$ starting at $r$ then $S_e$ is called the {\em slope index} of $T$. We say that an oriented edge $e$ of $T$ is {\em upward} if it goes towards the leaves.% (although we will draw the root on top in the figures).

\begin{lemma}\label{subtreelem}
If $T\subseteq\Gamma$ is a root subtree of slope index $s$, then

(i) $n_e=2$ for any edge $e$ in $T$,

(ii) if $e$ is an upward edge then $s_e\delta\le 0$,

(iii) $\sum_{v\in\Ram(f)\cap T}R_v=s$.
\end{lemma}
\begin{proof}
The proof runs by induction on the depth of $T$, i.e. the maximal length of a chain from $r$ to a leaf. Any leaf $v\in T$ is also a leaf of $\Gamma$. Since $\varphi$ cannot be contracted by removing $v$ and its image in $\Gamma'$, we necessarily have that $n_v=2$ and $R_v>0$. If $T$ is of depth 1 then it has a single edge $e$ connecting $r$ with a leaf $v$. Orienting $e$ upward we obtain that $s=1-s_e\delta=1+s_{-e}\delta=R_v$. In particular, $s_e\delta=1-R_v\le 0$.

Assume that the depth is larger than one. Let $e$ be the upward edge starting at $r$, let $x$ be the other end of $e$, let $e_1,\dots,e_m$ be all upward edges starting at $x$, and let $T_i$ be the rational tree growing from $x$ in the direction of $e_i$. Since $e$ is the only edge not contained in any $T_i$, we obtain by the induction assumption that claims (i) and (ii) hold for all edges different from $e$. In particular, it remains to check all claims for $e$.

Since $n_{e_i}=2$ for any $1\le i\le m$ and the edges $f\in\Br(x)$ with $n_f=1$ come in pairs, we obtain that $n_e=2$. In addition, $S_{e_i}=\sum_{v\in\Ram(f)\cap T_i}R_v$ by the induction, and so $\sum_{v\in\Ram(f)\cap T}R_v=\sum_{i=1}^mS_{e_i}$. Since $x$ is balanced we have that $S_{-e}+\sum_{i=1}^mS_{e_i}=\chi(x)=2$, and so $\sum_{i=1}^mS_{e_i}=2-S_{-e}=S_e=s$. Finally, $s_e\delta=1-s\le 0$ because $s=\sum_{v\in\Ram(f)\cap T}R_v$ and, as we mentioned above, $R_v>0$ for any leaf $v$.
\end{proof}

Since the set of leaves of $T$ is a subset of $\Ram(f)$ and the latter was described in Lemma~\ref{ramlem}, the same inductive argument as in the lemma produces a complete list of root subtrees that may occur in $\Gamma$. So, we skip the justification and just describe the eight trees.  The first three have a single edge connecting the root with the leaf and the slope can be 0, 1 or 3. The remaining five are as follows, with the arrow always indicating the direction with negative $s\delta$:

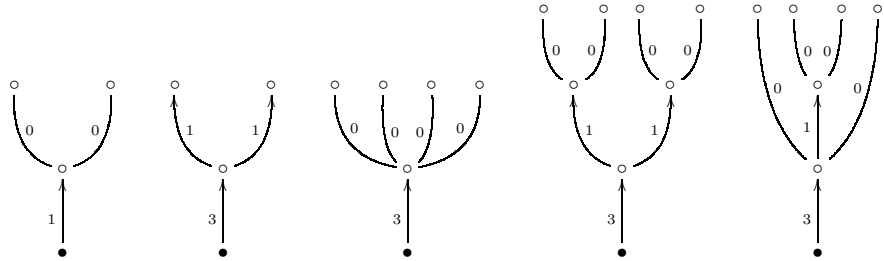
\begin{center}
\scalebox{0.8}{ \setlength{\tabcolsep}{8pt}
\begin{tabular}{ccccc}
{ $\xygraph{
    !{<0cm,0cm>;<1cm,0cm>:<0cm,-.7cm>::}
%    !~-{@{-}@[|1.5pt]}
    !{(0,4)}*+{\bullet}="root"
    !{(0,2)}*+{\circ}="s"
    !{(-0.8,0)}*+{\circ}="l"
    !{(0.8,0)}*+{\circ}="r"
    "root" :@{->}^(.4){1} "s"
    "s" -@`{(-0.9,1.8)}_(.6){0} "l"
    "s" -@`{(0.9,1.8)}^(.6){0} "r"
} $ }&{ $\xygraph{
    !{<0cm,0cm>;<1cm,0cm>:<0cm,-.7cm>::}
%    !~-{@{-}@[|1.5pt]}
    !{(0,4)}*+{\bullet}="root"
    !{(0,2)}*+{\circ}="s"
    !{(-0.8,0)}*+{\circ}="l"
    !{(0.8,0)}*+{\circ}="r"
    "root" :@{->}^(.4){3} "s"
    "s" :@`{(-0.9,1.8)}@{->}_(.6){1} "l"
    "s" :@`{(0.9,1.8)} @{->}^(.6){1} "r"
} $ }&{ $\xygraph{
    !{<0cm,0cm>;<1cm,0cm>:<0cm,-.7cm>::}
%    !~-{@{-}@[|1.5pt]}
    !{(0,4)}*+{\bullet}="root"
    !{(0,2)}*+{\circ}="s"
    !{(-1.2,0)}*+{\circ}="l1"
    !{(-0.4,0)}*+{\circ}="l2"
    !{(0.4,0)}*+{\circ}="r1"
    !{(1.2,0)}*+{\circ}="r2"
    "s" -@`{(-1.3,1.8)}_(.6){0} "l1"
    "s" -@`{(-0.5,1.8)}_(.6){0} "l2"
    "s" -@`{(0.5,1.8)}^(.6){0} "r1"
    "s" -@`{(1.3,1.8)}^(.6){0} "r2"
    "root" :@{->}^(.4){3} "s"
} $ }&{ $\xygraph{
    !{<0cm,0cm>;<1cm,0cm>:<0cm,-.7cm>::}
%    !~-{@{-}@[|1.5pt]}
    !{(0,4)}*+{\bullet}="root"
    !{(0,2)}*+{\circ}="s"
    !{(-0.8,0)}*+{\circ}="l"
    !{(-1.3,-1.8)}*+{\circ}="ll"
    !{(-0.3,-1.8)}*+{\circ}="lr"
    !{(0.8,0)}*+{\circ}="r"
    !{(1.3,-1.8)}*+{\circ}="rr"
    !{(0.3,-1.8)}*+{\circ}="rl"
    "root" :@{->}^(.4){3} "s"
    "s" :@`{(-0.9,1.8)}@{->}_(.6){1} "l"
    "s" :@`{(0.9,1.8)} @{->}^(.6){1} "r"
    "l" -@`{(-1.4,-0.2)}_(.6){0} "ll"
    "l" -@`{(-0.2,-0.2)}^(.6){0} "lr"
    "r" -@`{(1.4,-0.2)}^(.6){0} "rr"
    "r" -@`{(0.2,-0.2)}_(.6){0} "rl"
} $}
&
$ \xygraph{
    !{<0cm,0cm> ;<1cm,0cm>:<0cm,-.7cm>::}
%    !~-{@{-}@[|1.5pt]}
	!{(0,4)}*+{\bullet}="root"
    !{(0,2)}*+{\circ}="ro"
    !{(0,0)}*+{\circ}="s"
    !{(-0.4,-1.8)}*+{\circ}="l"
    !{(0.4,-1.8)}*+{\circ}="r"
    !{(-1,-1.8)}*+{\circ}="l1"
    !{(1,-1.8)}*+{\circ}="r1"
	"root" :@{->}^(.4){3} "ro"
    "ro" :@{->}^(.5){1} "s"
    "s" -@`{(-0.4,-0.2)}_(.6){0} "l"
    "s" -@`{(0.4,-0.2)}^(.6){0} "r"
    "ro" -@`{(-1,1)}_(.6){0} "l1"
    "ro" -@`{(1,1)}^(.6){0} "r1"
} $

\\
\end{tabular}
 }
\end{center}

\subsubsection{A classification: the terminology}
We will classify special $\delta$-morphisms by the characteristic type: tame, mixed or wild, and by the structure of $\Gamma$. Since $g(\Gamma)=1$, one of the following possibilities holds: (B) $\Gamma$ contains a loop (the bad reduction case), (G) $\Gamma$ contains a vertex $r$ of genus $1$ (the good reduction case). As we will see, in the mixed and wild cases, case (G) splits to the two cases: (O) $r$ is of valency 2 (ordinary reduction), (S) $r$ is of valency 1 (supersingular reduction), and in the mixed case there is also a possibility (ME) that $r$ is of valency 3. Also, it will be convenient to split the supersingular case to (S) and (SS). The latter can be thought of as ``strongly supersingular". In \S\ref{doubelcover} below, mixed and wild cases correspond to the case when $p=2$, and then case (SS) corresponds to the case $|j|\le|256|$ while supersingular reduction is obtained already when $|j|<1$.

\subsubsection{Good reduction}\label{goodsec}
Assume that $\Gamma$ contains a vertex $r$ of genus 1. Note that $\Gamma$ is a union of root subtrees $T_i$ with vertex $r$, and let $s_i$ be the slope index of $T_i$. Then the balancing condition at $r$ reads as $\sum_i s_i=\chi(r)=4$. Combining this with the list of root trees we obtain the following list of seven possibilities for $\Gamma$ with symmetric leaves:

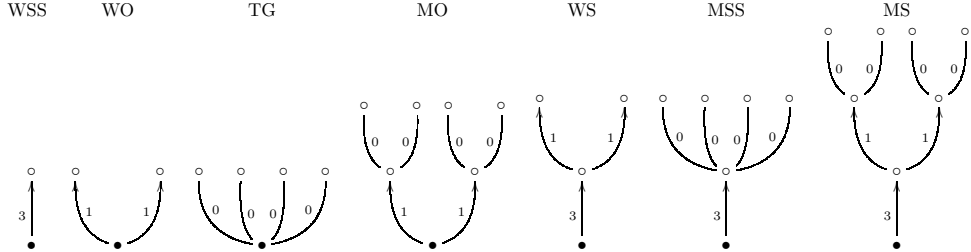
\begin{center}
\scalebox{0.7}{
\begin{tabular}{ccccccc}
WSS& WO& TG& MO& WS& MSS&MS
\\
$ \xygraph{
    !{<0cm,0cm>;<1cm,0cm>:<0cm,-.7cm>::}
%    !~-{@{-}@[|1.5pt]}
    !{(0,4)}*+{\bullet}="root"
    !{(0,2)}*+{\circ}="s"
    "root" :@{->}^(.4){3} "s"
}$ &

$ \xygraph{
    !{<0cm,0cm>;<1cm,0cm>:<0cm,-.7cm>::}
%    !~-{@{-}@[|1.5pt]}
    !{(0,4)}*+{\bullet}="s"
    !{(-0.8,2)}*+{\circ}="l"
    !{(0.8,2)}*+{\circ}="r"
    "s" :@`{(-0.9,3.8)}@{->}_(.6){1} "l"
    "s" :@`{(0.9,3.8)} @{->}^(.6){1} "r"
} $

& $ \xygraph{
    !{<0cm,0cm>;<1cm,0cm>:<0cm,-.7cm>::}
%    !~-{@{-}@[|1.5pt]}
    !{(0,4)}*+{\bullet}="s"
    !{(-1.2,2)}*+{\circ}="l1"
    !{(-0.4,2)}*+{\circ}="l2"
    !{(0.4,2)}*+{\circ}="r1"
    !{(1.2,2)}*+{\circ}="r2"
    "s" -@`{(-1.3,3.8)}_(.6){0} "l1"
    "s" -@`{(-0.5,3.8)}_(.6){0} "l2"
    "s" -@`{(0.5,3.8)}^(.6){0} "r1"
    "s" -@`{(1.3,3.8)}^(.6){0} "r2"
} $

& $ \xygraph{
    !{<0cm,0cm>;<1cm,0cm>:<0cm,-.7cm>::}
%    !~-{@{-}@[|1.5pt]}
    !{(0,4)}*+{\bullet}="s"
    !{(-0.8,2)}*+{\circ}="l"
    !{(-1.3,0.2)}*+{\circ}="ll"
    !{(-0.3,0.2)}*+{\circ}="lr"
    !{(0.8,2)}*+{\circ}="r"
    !{(1.3,0.2)}*+{\circ}="rr"
    !{(0.3,0.2)}*+{\circ}="rl"
    "s" :@`{(-0.9,3.8)}@{->}_(.6){1} "l"
    "s" :@`{(0.9,3.8)} @{->}^(.6){1} "r"
    "l" -@`{(-1.4,1.8)}_(.6){0} "ll"
    "l" -@`{(-0.2,1.8)}^(.6){0} "lr"
    "r" -@`{(1.4,1.8)}^(.6){0} "rr"
    "r" -@`{(0.2,1.8)}_(.6){0} "rl"
} $

& $ \xygraph{
    !{<0cm,0cm>;<1cm,0cm>:<0cm,-.7cm>::}
%    !~-{@{-}@[|1.5pt]}
    !{(0,4)}*+{\bullet}="root"
    !{(0,2)}*+{\circ}="s"
    !{(-0.8,0)}*+{\circ}="l"
    !{(0.8,0)}*+{\circ}="r"
    "root" :@{->}^(.4){3} "s"
    "s" :@`{(-0.9,1.8)}@{->}_(.6){1} "l"
    "s" :@`{(0.9,1.8)} @{->}^(.6){1} "r"
} $

& $ \xygraph{
    !{<0cm,0cm>;<1cm,0cm>:<0cm,-.7cm>::}
%    !~-{@{-}@[|1.5pt]}
    !{(0,4)}*+{\bullet}="root"
    !{(0,2)}*+{\circ}="s"
    !{(-1.2,0)}*+{\circ}="l1"
    !{(-0.4,0)}*+{\circ}="l2"
    !{(0.4,0)}*+{\circ}="r1"
    !{(1.2,0)}*+{\circ}="r2"
    "s" -@`{(-1.3,1.8)}_(.6){0} "l1"
    "s" -@`{(-0.5,1.8)}_(.6){0} "l2"
    "s" -@`{(0.5,1.8)}^(.6){0} "r1"
    "s" -@`{(1.3,1.8)}^(.6){0} "r2"
    "root" :@{->}^(.4){3} "s"
} $

& $ \xygraph{
    !{<0cm,0cm>;<1cm,0cm>:<0cm,-.7cm>::}
%    !~-{@{-}@[|1.5pt]}
    !{(0,4)}*+{\bullet}="root"
    !{(0,2)}*+{\circ}="s"
    !{(-0.8,0)}*+{\circ}="l"
    !{(-1.3,-1.8)}*+{\circ}="ll"
    !{(-0.3,-1.8)}*+{\circ}="lr"
    !{(0.8,0)}*+{\circ}="r"
    !{(1.3,-1.8)}*+{\circ}="rr"
    !{(0.3,-1.8)}*+{\circ}="rl"
    "root" :@{->}^(.4){3} "s"
    "s" :@`{(-0.9,1.8)}@{->}_(.6){1} "l"
    "s" :@`{(0.9,1.8)} @{->}^(.6){1} "r"
    "l" -@`{(-1.4,-0.2)}_(.6){0} "ll"
    "l" -@`{(-0.2,-0.2)}^(.6){0} "lr"
    "r" -@`{(1.4,-0.2)}^(.6){0} "rr"
    "r" -@`{(0.2,-0.2)}_(.6){0} "rl"
} $

\end{tabular}
}
\end{center}
and two exceptional configurations
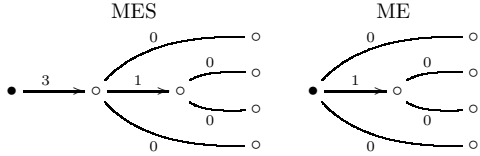
\begin{center}
\scalebox{0.8}{ \setlength{\tabcolsep}{8pt}
\begin{tabular}{cc}
MES& ME
\\
$ \xygraph{
    !{<0cm,0cm>;<0cm,1cm>:<-0.7cm,0cm>::}
%    !~-{@{-}@[|1.5pt]}
	!{(0,4)}*+{\bullet}="root"
    !{(0,2)}*+{\circ}="ro"
    !{(0,0)}*+{\circ}="s"
    !{(-0.3,-1.8)}*+{\circ}="l"
    !{(0.3,-1.8)}*+{\circ}="r"
    !{(-0.9,-1.8)}*+{\circ}="l1"
    !{(0.9,-1.8)}*+{\circ}="r1"
	"root" :@{->}^(.4){3} "ro"
    "ro" :@{->}^(.5){1} "s"
    "s" -@`{(-0.4,-0.2)}_(.6){0} "l"
    "s" -@`{(0.4,-0.2)}^(.6){0} "r"
    "ro" -@`{(-1,1)}_{0} "l1"
    "ro" -@`{(1,1)}^{0} "r1"
} $
&
$ \xygraph{
    !{<0cm,0cm>;<0cm,1cm>:<-0.7cm,0cm>::}
%    !~-{@{-}@[|1.5pt]}
    !{(0,2)}*+{\bullet}="ro"
    !{(0,0)}*+{\circ}="s"
    !{(-0.3,-1.8)}*+{\circ}="l"
    !{(0.3,-1.8)}*+{\circ}="r"
    !{(-0.9,-1.8)}*+{\circ}="l1"
    !{(0.9,-1.8)}*+{\circ}="r1"
    "ro" :@{->}^(.5){1} "s"
    "s" -@`{(-0.4,-0.2)}_(.6){0} "l"
    "s" -@`{(0.4,-0.2)}^(.6){0} "r"
    "ro" -@`{(-1,1)}_{0} "l1"
    "ro" -@`{(1,1)}^{0} "r1"
} $
\end{tabular}
 }
\end{center}

\subsubsection{Bad reduction}\label{badsec}
Now, assume that $\Gamma$ contains a loop $L$. Since $\Gamma'$ is a tree, $n_e=1$ for any edge in $L$ and $n_v=1$ for all but two vertices of $L$ that we denote $x$ and $y$. Note that $\Gamma$ is a union of $L$ and root trees with roots $r$ in $L$. Moreover, by Lemma~\ref{subtreelem}, a root tree can only start at a root $r$ with $n_r=2$, hence $\Gamma$ is a union of $L$ and root trees hanging on $x$ and $y$. In particular, if $v$ is a vertex of $L$ different from $x$ and $y$ then its valency is 2. Since $v$ is not a leaf, it is balanced and hence it can be contracted, contrary to the assumption that $\varphi$ is stable. This proves that $L$ consists of the vertices $x,y$ and two edges $e,f$ connecting them. Since $e$ and $f$ split, $s_e\delta=s_f\delta=0$ and hence $S_e=n_e-1=0$ and $S_f=0$. Thus, the balancing conditions for $x$ and $y$ imply that the sum of slope indices of root trees hanging on each of them equals to $\chi(x)=2$. This leaves us with the following three options:

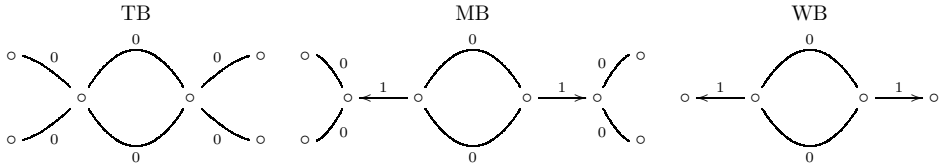
\begin{center}
\scalebox{0.8}{
\begin{tabular}{ccc}
TB& MB& WB

\\

$ \xygraph{
    !{<0cm,0cm>;<0cm,1cm>:<0.9cm,0cm>::}
    !{(0,4)}*+{\circ}="t"
    !{(0,2)}*+{\circ}="s"
    !{(0.7,5.3)}*+{\circ}="lu"
    !{(-0.7,5.3)}*+{\circ}="ld"
    !{(0.7,0.7)}*+{\circ}="ru"
    !{(-0.7,0.7)}*+{\circ}="rd"
    "t" -@`{(1.6,3)}_{0} "s"
    "t" -@`{(-1.6,3)}^{0} "s"
    "t" -@`{(0.8,4.9)}^(.4){0} "lu"
    "t" -@`{(-0.8,4.9)}_(.4){0} "ld"
    "s" -@`{(0.8,1.1)}_(.4){0} "ru"
    "s" -@`{(-0.8,1.1)}^(.4){0} "rd"
} $

& $ \xygraph{
    !{<0cm,0cm>;<0cm,1cm>:<0.9cm,0cm>::}
    !{(0,4)}*+{\circ}="t"
    !{(0,2)}*+{\circ}="s"
    !{(0,5.3)}*+{\circ}="l"
    !{(0.7,6.1)}*+{\circ}="lu"
    !{(-0.7,6.1)}*+{\circ}="ld"
    !{(0,0.7)}*+{\circ}="r"
    !{(0.7,-0.1)}*+{\circ}="ru"
    !{(-0.7,-0.1)}*+{\circ}="rd"
    "t" -@`{(1.6,3)}_{0} "s"
    "t" -@`{(-1.6,3)}^{0} "s"
    "t" :@{->}^{1} "l"
    "s" :@{->}_{1} "r"
    "l" -@`{(0.9,5.8)}^(.3){0} "lu"
    "l" -@`{(-0.9,5.8)}_(.3){0} "ld"
    "r" -@`{(0.9,0.2)}_(.3){0} "ru"
    "r" -@`{(-0.9,0.2)}^(.3){0} "rd"
} $

& $ \xygraph{
    !{<0cm,0cm>;<0cm,1cm>:<0.9cm,0cm>::}
    !{(0,4)}*+{\circ}="t"
    !{(0,2)}*+{\circ}="s"
    !{(0,5.3)}*+{\circ}="l"
    !{(0,0.7)}*+{\circ}="r"
    "t" -@`{(1.6,3)}_{0} "s"
    "t" -@`{(-1.6,3)}^{0} "s"
    "t" :@{->}^{1} "l"
    "s" :@{->}_{1} "r"
} $
\end{tabular}
}
\end{center}

\subsubsection{The final classification}
It remains to summarize the results of Section \ref{classsec}.

\begin{theorem}\label{specialClassification}
Up to an isomorphism, there exist twelve special $\delta$-morphisms $\varphi\:\Gamma\to\Gamma'$: (TB), (MB), (WB) are the bad reduction cases in each characteristic, (TG) is the good reduction in the tame case, (MO), (WO) are the ordinary reduction cases in the mixed and wild cases, (MS) and (WS) are supersingular reductions in the mixed and wild case, (MSS) and (WSS) are strongly supersingular configurations in the mixed and wild case, (ME) and (MES) are exceptional graphs in the mixed case. The possibilities for $\Gamma$ are shown on the figures in Sections~\ref{goodsec} and \ref{badsec}. In each case, the map $V\to V'$ is bijective and the map $E\to E'$ is bijective on all edges not contained in a loop and, if $h^1(\Gamma)=1$, sends the edges of the loop to the same edge in $E'$.
\end{theorem}
\begin{proof}
We have proved above that these twelve cases are the only possibilities for $\Gamma$. In addition, we proved that $n_v=2$ for any vertex of $\Gamma$ and $n_e=2$ if and only if $e$ does not lie in the loop. Thus, if $\Gamma$ extends to a special $\delta$-morphism $\varphi\:\Gamma\to\Gamma'$, then $\varphi$ has to be as stated in the theorem. It is a trivial check that, indeed, in all twelve cases this recipe produces a special $\delta$-morphism.
\end{proof}

\begin{remark}\label{specialrem}
Let $\Gamma_0$ be the convex hull of $f(\Ram(f))$ in $\Gamma'$. In the tame and mixed case, it is a tree with four leaves, so it has either the X-shape (a star graph of valency four) or the H-shape (two vertices of valency 3). The X-shape corresponds to the cases (TG) and (MSS), and the H-shape to (TB), (ME), (MES), (MB), (MO) and (MS). We will see in the next section that the latter three cases can be distinguished by the length $l$ of the {\em bar} (i.e. the path connecting the valency three vertices in H), and the exceptional configurations are excluded by the condition $l>0$.
\end{remark}

\subsection{Metric genus graphs}

\subsubsection{Metric graphs}
Usually, a {\em metric graph} means a topological graph all whose edges are provided with metrics making them homeomorphic to closed intervals $I\subset\bfR$. We extend this definition by allowing {\em infinite leaves}. Each such leaf is singular for the metric, i.e. the metric on its edge induces a homeomorphism $e\toisom[a,\infty]$. The edge $e$ will be called a {\em tail}. All other edges have finite length and they will be called {\em inner}. We will only consider metric graphs of finite type, in the sense that there are finitely many edges and vertices.

\begin{remark}
(i) One can also work within purely combinatorial framework by providing a combinatorial graph $\Gamma$ with a length function $l:E\to(0,\infty]$. The metric graph in our sense is a topological realization of such an object.

(ii) In our situation, tails will correspond to type 1 points. In tropical geometry tails correspond to divisors or marked points. In fact, these two contexts are tightly related.
\end{remark}

\subsubsection{Morphisms}
A morphism $\varphi\:\Gamma\to\Gamma'$ between metric graphs is a continuous map which sends vertices to vertices and edges to edges, and each induced map $e\to e'$ has a constant dilatation factor $n_e\in\bfZ_{>0}$.

\begin{remark}
On the combinatorial side, this corresponds to an $n$-morphism $(\varphi,n)\:\Gamma\to\Gamma'$ such that $l(\varphi(e))=n_el(e)$ for any edge $e\in E$.
\end{remark}

\subsubsection{Proper morphisms}
Similarly to the combinatorial case, a morphism between connected graphs is called {\em proper} if it has a locally constant multiplicity at all vertices (in particular, the multiplicities $n_v$ are defined) and the global rank is constant.

\subsubsection{Metric genus graphs}
By a {\em metric genus graph} we mean a metric graph provided with a genus function $g\:V\to\bfN$ such that $g(v)=0$ for any infinite leaf $v$.

\subsubsection{$\delta$-morphism of metric genus graphs}
Fix a non-archimedean real semivaluation $|\ |$ on $\bfZ$; it is either trivial, or $p$-adic, or induced from the trivial valuation on $\bfF_p$. A {\em $\delta$-morphism} between metric (genus) graphs (with respect to $|\ |$) is a pair $(\varphi,\delta)$, where $\varphi\:\Gamma\to\Gamma'$ is a proper morphism of metric graphs and $\delta\:\Gamma\to[0,1]$ is a continuous function such that $\log\delta|_e$ is a linear function with an integral slope for each edge $e\subset\Gamma$. In addition, we require that $\delta(v)=|n_v|$ for any infinite leaf $v$, and for any other vertex $v$ and edge $e\in\Br(v)$ the condition of Theorem~\ref{restrictth} is satisfied. In particular, if $n_e=2$ and $\cha(\tilk)=2$ then $s_e\delta$ is even only when $\delta=|2|$ along $e$.

\begin{remark}
(i) So far, our definitions run parallel to the combinatorial ones, but the situation with $\delta$ is different. The slope function $s\delta\:E^{or}\to\bfZ$ we considered in Section~\ref{genusgraphsec} does not have to be the differential of any function $\log\delta\:V\to\bfR$.

(ii) In our applications, $\delta$ will be the restriction of the different onto a skeleton. Its slope $s_e\delta$ along an edge $e$ is not determined only by the values of $\delta$ at the vertices of $e$. In order to compute $s_e\delta$ one should also use the length of $e$, so restricting $\log\delta$ onto the set of vertices $V$ and ignoring the lengths one gets a meaningless function not related to $s\delta$.
\end{remark}

\subsubsection{Special $\delta$-morphisms of metric genus graphs}\label{specialmetric}
A $\delta$-morphism of metric genus graphs is {\em special} if it induces a special $\delta$-morphism of the corresponding combinatorial graphs and, in addition, if $r$ is a vertex of genus one then $\delta(r)=1$. It is easy to see that the type of the combinatorial morphism is as follows: tame if $|2|=1$, mixed if $0<|2|<1$, wild if $|2|=0$.

\subsubsection{Classification}
We classify special $\delta$-morphisms of metric genus graphs into twelve types according to the type of the underlying special $\delta$-morphism. In fact, we will see that the exceptional cases cannot occur, so we are left with ten cases. In addition, we describe all possible metrics in these cases.

\begin{theorem}\label{speciallem}
(i) The ten non-exceptional cases are precisely the cases that can be lifted to special $\delta$-morphisms of metric graphs.

(ii) All possible metrics on the liftings are described by the following three rules, where we only describe the lengths in $\Gamma$ since the lengths in $\Gamma'$ are then defined as $l(\varphi(e))=n_el(e)$.

(a) The length of any tail is infinite.

(b) All inner edges of the same slope are of the same length, that we denote $l_0$, $l_1$ and $l_3$ according to the slope.

(c) Set $l_i=0$ if there are no inner edges of slope $i$. Then in each of the ten cases, the only restriction on the numbers $l_i$ is that in the mixed case $\sum_i il_i=-\log|2|$.
\end{theorem}

\begin{remark}
(i) Conditions (b) and (c) above can be explicated as follows. In the bad reduction case, the two edges in the loop are of the same length that can be equal to any positive number $l_0$. Inner edges of positive slope are as follows:

(MB) and (MO) The edges of slope 1 are of length $l_1=-\log|2|$.

(MS) The edges of slope 1 are of the same length $l_1\in(0,-\log|2|)$ and the edge of slope 3 is of length $l_3=\frac{-\log|2|-l_1}{3}$.

(MSS) The edge of slope 3 is of length $l_3=\frac{-\log|2|}3$.

(WS) The length of the edge of slope 3 is an arbitrary number $l_3\in(0,\infty)$.

(ii) The formula $\sum il_i=-\log|2|$ poses a restriction only in the mixed case, but it makes sense more generally. In the tame case, it means that $l_1=l_3=0$. To make sense of it in the wild case, one should redefine $l_i$ with $i>0$ by setting $l_i=\infty$ if there is a tail of slope $i$. Then the formula means that in the wild case there is a tail of slope 1 or 3.
\end{remark}

\begin{proof}[Proof of Theorem~\ref{speciallem}]
One checks straightforwardly that all the suggested metrics give rise to special $\delta$-morphisms of metric graphs. So, it remains to establish the asserted restrictions. In the bad reduction case, the two edges $e$ and $f$ of the loop are mapped to the same edge $h$ of $\Gamma'$ and $n_e=n_f=1$. Hence $l(e)=l(h)=l(f)$. Other restrictions, including the equality of lengths of the edges of the same slope, are only essential in the mixed case, and they all follow in an obvious way from the observation that $\delta=1$ on the loop and at the good reduction point, and $\delta=|2|$ on the tails. For an illustration, let us check this for (MS), (ME) and (MES) cases.

In the exceptional cases, there is an edge of a positive slope that connects two tails. This is impossible since the different on both its ends equals to $|2|$. In the case (MS), there are inner edges $a,b$ of slope 1 and an inner edge $c$ of slope 3. The paths $(c,a)$ and $(c,b)$ connect the good reduction point with the tail, hence $l(a)+3l(c)=l(b)+3l(c)=-\log|2|$.
\end{proof}

Finally, we can use the metric to complete Remark \ref{specialrem} by separating mixed cases.

\begin{remark}\label{specialmetricrem}
Assume that $\varphi\:\Gamma\to\Gamma'$ is of type (MB), (MO) or (MS). The convex hull $\Gamma_0$ of $\varphi(\Ram(\varphi))$ has an H-shape and let $l$ be the length of the bar. In all cases, the bar consists of the images of all inner edges of slopes 0 and 1, and $n_e=2$ on edges of slope 1. It follows that $l=4l_1+l_0$, and using Theorem~\ref{speciallem} we obtain that $l>-\log|16|$ in the case (MB), $l=-\log|16|$ in the case (MO), and $l<-\log|16|$ in the case (MS).
\end{remark}

\section{Main results}\label{mainsec}

\subsection{Ordinary behaviour of $\delta_f$}

\subsubsection{Orientation on a curve}
By an orientation on a curve $X$ we mean a map $\tau$ from the set of branches of $X$ to the set $\{-1,0,1\}$ such that for any point $x\in X$ and a branch $v$ at $x$ there exists an interval $[x,y]$ in the direction of $v$ such that if $x'\in[x,y)$ and $v'$ is the branch at $x'$ corresponding to $[x',y]$ then $\tau(v)=\tau(v')$.

We say that a branch $v$ is {\em downward, neutral} or {\em upward} according to the value of $\tau(v)$. Similarly, if $I=[x,y]$ is an interval and for any $x'\in[x,y)$ with branch $v'$ corresponding to $[x',y]$ the value of $\tau(v')$ is constant on $I$, then we say that $I$ is {\em downward, neutral} or {\em upward}, according to the value of $\tau$. It follows from the definition that any interval $I\subset X$ possesses a finite subdivision into a union of downward, neutral and upward intervals.

\begin{example}
(i) If $f\:Y\to X$ is a finite morphism of curves and $\tau$ is an orientation on $X$ then its pullback $f^*\tau=\tau\circ f$ is an orientation on $Y$.

(ii) Any piecewise monomial function $\phi\:X\to\bfR_+$ induces an orientation $\tau$ on $X$ such that an interval $I\subset X$ is downward, neutral, or upward if and only if $\phi|_I$ strictly decreases, is constant, or strictly increases, respectively. Actually, $\tau(v)={\rm sign}(\slope_v(\phi))$.
\end{example}

\subsubsection{Orientation induced by a skeleton}
Any skeleton $\Gamma\subset X$ naturally induces an orientation $\tau_\Gamma$ on $X$ that points towards $\Gamma$. Namely, the edges of $\Gamma$ are neutral for $\tau_\Gamma$ and any interval $[x,y]$ with $[x,y]\cap\Gamma=\{y\}$ is increasing.

\begin{remark}
In fact, any connected component $D$ of $X\setminus\Gamma$ is an open disc and the restriction of $\tau_\Gamma$ onto $D$ is induced by the radius function on $D$. More generally, the formula $r_\Gamma=\exp^{-d(x,\Gamma)}$ defines a piecewise monomial radius function on $X$ that measures the inverse exponential distance from $\Gamma$, and then $\tau_\Gamma$ is the orientation induced by $r_\Gamma$.
\end{remark}

\subsubsection{$\delta$-ordinary points}
Assume now that $f\:Y\to X$ is a finite generically \'etale morphism of nice compact curves and an open subdomain $V\subset Y$ is provided with an orientation. We say that $y\in V$ is a {\em $\delta$-ordinary} point of the covering $f$ if there is a unique upward direction $v$ at $y$ and $\slope_v(\delta_f)=1-n_y$. We say that $\delta_f$ {\em behaves ordinary} on $V$ if any point of $V$ is $\delta$-ordinary.

\begin{remark}
The condition on existence and uniqueness of $v$ is essential only for type 2 points; it is automatic for other types.
\end{remark}

\subsubsection{Skeletons and trivialization of $\delta_f$}
We say that $\delta_f$ is {\em trivialized} by a skeleton $\Gamma\subset Y$ if it behaves ordinary on $Y\setminus\Gamma$ with respect to the orientation induced by $\Gamma$.

\begin{lemma}\label{trivlem}
Assume that a skeleton $\Gamma$ of $Y$ trivializes $\delta_f$. Then $S_v=0$ for any downward branch $v$ and $R_y=0$ for any unibranch point $y\in Y\setminus\Gamma$.
\end{lemma}
\begin{proof}
Take a downward interval $I=[x,y]$ in the direction of $v$. By Lemma~\ref{pmmaplem}, choosing $I$ small enough we can achieve that $n_z=n_v$ for any $z\in (x,y]$. The opposite interval $[y,x]$ is upward and since $\delta_f$ behaves ordinary outside of $\Gamma$, it is of constant slope $1-n_v$ on $[y,x]$. Hence $\delta_f$ is of constant slope $n_v-1$ on $[x,y]$, in particular, $S_v=-\slope_v\delta_f+n_v-1=0$.

If $y\in Y\setminus\Gamma$ is unibranch then its branch is upward and hence $\slope_y\delta_f=1-n_y$, $S_y=2n_y-2$ and $R_y=0$.
\end{proof}

The following theorem is our first main result on the connection between $\delta_f$ and skeletons.

\begin{theorem}\label{deltatrivth}
Assume that $f\:Y\to X$ is a finite generically \'etale morphism of nice compact curves and $(\Gamma_Y,\Gamma_X)$ is a skeleton of $f$, then

(i) $\Gamma_Y$ trivializes $\delta_f$.

(ii) $\delta_f$ has constant slope on any oriented edge of $\Gamma_Y$ and hence induces a $\delta$-morphism $(f,n,s\delta_f)\:\Gamma_Y\to\Gamma_X$ of genus graphs.

(iii) Any non-balanced vertex of $(f,n,s\delta_f)$ is contained in $\partial(Y)\cup\Ram(f)$.
\end{theorem}
\begin{proof}
We start with (i). Fix a point $y\in Y\setminus\Gamma_Y$ and let us prove that it is $\delta$-ordinary. The connected component of $y$ in $Y\setminus\Gamma_Y$ is an open disc $Y_0$ and $f$ restricts to an \'etale covering $f_0\:Y_0\to X_0$, where $X_0$ is the connected component of $f(y)$ in $X\setminus\Gamma_X$. We can identify $Y_0$ and $X_0$ with open unit discs with coordinates $t$ and $z$, and then $f_0$ is given by sending $z$ to a series $h(t)=\sum_{i=0}^\infty h_it^i\in\kcirc[[t]]$. Furthermore, $f_0$ is \'etale, hence $h'(t)$ is invertible and so $|h'(a)|=|h_1|$ for any point $a\in Y_0$.

Let $y_r\in Y_0$ denote the maximal point of the disc $E_r$ of radius $r$ with center at 0. Assume, first, that $y$ is of type 2 or 3. Then we can choose $t$ to be monomial at $y$, i.e. we can assume that $y=y_s$ for some $0<s<1$. Note that $h$ induces a finite map $h_s\:E_s\to h(E_s)$ between discs and $y_s$ is the only preimage of the maximal point of the target, hence $n_y=\deg(h_s)$. On the other hand, it follows from the Weierstrass division theorem that $\deg(h_s)$ is the maximal number $d$ such that $|h|_y=\max_n |h_n|s^n$ equals to $|h_d|s^d$. Choose $s_1\in(s,1)$ such that $|h_n|s_1^n<h_{n_y}s_1^{n_y}$ for any $n>n_y$, then $n_{y_q}=n_y$ for any $q\in[s,s_1]$. In particular, if $x_q=f(y_q)$ then $r_z(x_q)=|h_{n_y}|q^{n_y}$. Since $r_t(y_q)=q$, Theorem~\ref{compdeltalem} implies that
$$\delta_f(z)=|h'|_z |h_{n_y}|^{-1}q q^{-n_y}=|h_1h^{-1}_{n_y}|q^{1-n_y}.$$ Therefore, the upward slope of $\delta_f$ equals to $1-n_y$ everywhere on the interval $[y,y_{s_1}]$.

It remains to consider the case when $y$ is of type 1 or 4. Consider an increasing interval $I$ starting at $y$. The function $n_y$ is constant in a neighborhood of $y$ in $I$ because it can only jump at type 2 points (see Lemma~\ref{multfuncor}) and the slope of $\delta_f$ is constant in a neighborhood of $y$ in $I$ because $\delta_f$ is piecewise monomial by Corollary~\ref{pmlem}. By the case of type 2 and 3 points, the upward slope of $\delta_f$ equals to $1-n_y$ for any point of $I\setminus y$, hence the same is true for $y$.

Let us prove (ii). Recall that the value of $n_y$ is fixed along any edge $e$ by Lemma~\ref{edgemultlem}, so we denote it by $n_e$. Then $(f,n)\:\Gamma_Y\to\Gamma_X$ is a proper $n$-morphism of graphs by Remark~\ref{mult5rem}. So, it suffices to show that for any type 2 point $y\in e$ with branches $u$ and $v$ pointing at different directions along $e$, the numbers $s_u=\slope_u(\delta_f)$ and $s_v=\slope_v(\delta_f)$ are opposite. Note that $g(y)=0$ since $y$ is not a vertex of $\Gamma_Y$, and hence also $g(f(y))=0$. In addition, $n=n_e$ coincides with $n_v$, $n_u$ and $n_y$. Any direction $w\in C_y\setminus\{u,v\}$ is downward, hence $S_w=0$ by Lemma~\ref{trivlem}, and the local Riemann-Hurwitz formula at $y$, see \ref{prop:local_RH}, reads as $$-2=-2n+(-s_u+n-1)+(-s_v+n-1).$$ Thus, $s_u+s_v=0$, as required.

Finally, if a non-boundary type 2 point $y$ is a vertex of $\Gamma_Y$, then a similar application of Lemma~\ref{trivlem} and the local Riemann-Hurwitz formula at $y$ proves that $y$ is balanced, whence (iii) follows.
\end{proof}

\subsection{The genus formulas}

\subsubsection{Genus of a nice compact curve}\label{gennicesec}
For any nice compact curve $X$ we define its genus as the sum of its first Betti number and all genera of its type 2 points: $g(X)=h^1(X)+\sum_{x\in X}g(x)$. It is a classical result that this definition agrees with the usual notion of genus when $X$ is a connected smooth proper curve. Since any skeleton $\Gamma$ of $X$ is a deformation retract of $C$ and contains all points of non-zero genus, we have that $g(X)=g(\Gamma)$, where $\Gamma$ is viewed as a genus graph.

\subsubsection{The genus formula for nice compact curves}
The following result extends the classical algebraic Riemann-Hurwitz formula to nice compact curves with boundary.

\begin{theorem}\label{genusth}
Assume that $f\:Y\to X$ is a finite generically \'etale morphism of degree $n$ between connected nice compact curves. Then $$2g(Y)-2-n(2g(X)-2)=\sum_{y\in Y}R_y=\sum_{y\in\Ram(f)}R_y+\sum_{b\in\partial(Y)}R_b.$$
\end{theorem}
\begin{proof}
Choose a skeleton $(\Gamma_Y,\Gamma_X)$ of $f$; by Theorem~\ref{deltatrivth}(ii) it induces a $\delta$-morphism $\varphi\:\Gamma_Y\to\Gamma_X$. By Section~\ref{gennicesec}, $g(\Gamma_X)=g(X)$ and $g(\Gamma_Y)=g(Y)$. Furthermore, any ramification or boundary point is a vertex of $\Gamma_Y$. For any vertex $y\in\Gamma_Y^0$, we have that $R_{y,f}=R_{y,\varphi}$ because $S_{v,f}=S_{v,\varphi}$ for any $v\in\Br(y)$ pointing along an edge of $\Gamma_Y$ and $S_{v,f}=0$ for any other branch at $y$. Thus, the genus formula for $f$ follows from the combinatorial genus formula for $\varphi$, see Theorem~\ref{combRH}(ii).
\end{proof}

\subsubsection{Wide open domains}\label{specsec}
A connected open domain $V\subset Y$ will be called {\em wide} if $S=\oV\setminus V$ is a finite non-empty set of type 2 points. (Then $V$ is a wide open curve as defined by Coleman.) Note that $V$ is a connected component of $Y\setminus S$. The genus of a wide open domain $V$ is defined similarly to the genus of a nice compact curve, namely $g(V)=h^1(V)+\sum_{y\in V}g(y)$. We will not need the following remark, so its justification is omitted.

\begin{remark}
(i) Wide open domains typically appear as formal fibers, i.e. preimages of closed points under the reduction map $\pi\:Y\to\gtY_s$, where $\gtY$ is a formal model of $Y$. In fact, one can show that any wide open $V$ is a formal fiber of some formal model.

(ii) If $\gty$ is a closed point of $\gtY_s$ and $V=\pi^{-1}(\gty)$ then $g(V)=\delta_\gty-n_\gty+1$, where $n_\gty$ is the number of branches at $\gty$ and $\delta_\gty$ is the classical $\delta$-invariant of $\gty$, that measures the contribution of $\gty$ to the arithmetic genus. In other words, if $Z$ is the normalization of $\gtY_s$ and $z$ is the preimage of $\gty$ in $Z$ with semilocal ring $\calO_z=\calO_{Z,z}$, then $n_\gty$ is the number of points in $z$ and $\delta_\gty=\dim_\tilk(\calO_z/\calO_\gty)$.
\end{remark}

\subsubsection{The genus formula for wide open domains}
Given a wide open domain $V\subseteq Y$ we say that $v$ is a {\em branch at infinity} of $V$ if $v$ is a branch at a point $x\in\oV\setminus V$ and any interval $[x,y]$ along $v$ intersects with $V$. The set of all branches at infinity will be denoted $V_\infty$.

\begin{theorem}\label{genusth2}
Assume that $f\:Y\to X$ is a finite generically \'etale morphism between nice compact curves, $U\subseteq X$ is a wide open domain and $V$ is a connected component of $f^{-1}(U)$. Then $$2g(V)-2-n(2g(U)-2)=\sum_{y\in\Ram(f)\cap V}R_y+\sum_{v\in V_\infty}(2n_v-2-S_v),$$ where $n$ is the degree of the induced morphism $V\to U$.
\end{theorem}
\begin{proof}
Choose a skeleton $(\Gamma_Y,\Gamma_X)$ of $f$ such that $\oU\setminus U\subseteq\Gamma_X^0$. Define a graph $\Gamma_U$ to be equal to to the compactification of $\Gamma_X\cap U$ by the points of $U_\infty$, that is, $\Gamma_U^0=(\Gamma_X^0\cap U)\cup U_\infty$ and the edges of $\Gamma_U$ are the edges of $\Gamma_X$ lying in $U$. We assign genus zero to the vertices of $U_\infty$. In the same fashion, we define $\Gamma_V$ to be the compactification of $\Gamma_Y\cap V$ by the vertices of $V_\infty$.

Now, the claim reduces to the combinatorial genus formula for $\Gamma_V\to\Gamma_U$ similarly to the proof of Theorem~\ref{genusth}. We omit the details and only remark that the ramification points of $\Gamma_V$ are the usual ramification points of $Y$ lying in $V$ and the points of $V_\infty$. Each $v\in V_\infty$ is a leaf of genus zero, hence $R_v=2n_v-2-S_v$ and we see that the right hand side of the asserted equation is the sum of $R_y$ over all ramification points of $\Gamma_V$.
\end{proof}

\begin{rem}
In fact, the assumption that $f\:V\to U$ comes from a morphism of nice compact curves is only needed to obtain the numbers $n_v$ and $S_v$ for $v\in V_\infty$. The theorem can be easily extended to the case when $V$ and $U$ are wide open domains and $f\:V\to U$ is a finite generically \'etale morphism such that for any $v\in V_\infty$ there exist an interval $[a,v)\subset V$ in the direction of $v$, a branch at infinity $u\in U_\infty$ and an interval $[b,u)\subset U$ in the direction of $u$ such that $f$ maps $[a,v)$ to $[b,u)$ and $n_y$ and $S_y$ are constant along $[a,v)$.
\end{rem}

\subsection{The different and the minimal skeleton of $f$}

\subsubsection{Coverings of an open disc}
Our next result shows that a (compactifiable) \'etale covering of an open disc is a disc if and only if $\delta_f$ behaves ordinary at the branches at infinity.

\begin{lemma}\label{coverdisclem}
Assume that $f\:Y\to X$ is a finite \'etale morphism between connected nice compact curves, $U\subset X$ is a wide open domain isomorphic to a disc and $V$ is a connected component of $f^{-1}(U)$. If $S_v=0$ for any $v\in V_\infty$ then $V$ is an open disc.
\end{lemma}
\begin{proof}
Let $n$ be the degree of $f|_V$; it is well defined since $V$ is connected and non-empty. Clearly, $n=\sum_{v\in V_\infty}n_v$. By our assumption, there are no ramification points hence the genus formula of Theorem~\ref{genusth2} reads as $$2g(V)+2n-2=\sum_{v\in V_\infty}(2n_v-2).$$ Since $\sum_{v\in V_\infty}(2n_v-2)\le 2n-2$ with equality holding if and only if $|V_\infty|=1$, we obtain that $g(V)=0$ and $V$ has a single branch at infinity. Using the semistable reduction theorem it follows easily that $V$ is an open disc.
\end{proof}

\subsubsection{A characterization of skeletons of $f$}
Now we can characterize the skeletons of $f$ in terms of the different. We say that a graph $\Gamma\subset Y$ {\em locally trivializes} $\delta_f$ if for any point $y\in\Gamma$ and a branch $v\in\Br(y)$ pointing outside of $\Gamma$ the equality $S_v=0$ holds.

\begin{theorem}\label{localcharth}
Let $f\:Y\to X$ be a finite generically \'etale morphism of nice compact curves, let $\Gamma_X\subset X$ be a skeleton and let $\Gamma_Y\subset Y$ be the preimage of $\Gamma_X$. Then $(\Gamma_Y,\Gamma_X)$ is a skeleton of $f$ if and only if $\Ram(f)\subset\Gamma_Y^0$ and $\Gamma_Y$ locally trivializes $\delta_f$.
\end{theorem}
\begin{proof}
The direct implication is covered by Theorem~\ref{deltatrivth}, so let us prove the opposite one. Let $D$ be any connected component of $X\setminus\Gamma_X$ and let $V$ be a connected component $f^{-1}(D)$. The finite map $V\to D$ is \'etale by our assumption on the ramification locus. In addition, any branch at infinity $v\in V_\infty$ is a branch at a point of $\Gamma_Y$ that points outside of $\Gamma_Y$. Hence $S_v=0$, and by Lemma~\ref{coverdisclem} we obtain that $V$ is an open disc. This proves that $\Gamma_Y$ is a skeleton of $Y$ and we are done.
\end{proof}

\begin{remark}\label{skeletonrem}
The main advantage of the new characterization of the skeletons of $f$ is that it is of local nature on $Y$, in particular, one obtains a pretty explicit way to construct a skeleton of $Y$ in terms of $X$ and the covering. Namely, start with any skeleton $\Gamma_X$ of $X$. Enlarge $\Gamma_X$ to contain the image of the ramification locus of $f$. If there is a point $y\in\Gamma_Y=f^{-1}(\Gamma_X)$ and a branch $v$ at $y$ pointing outside of $\Gamma_Y$ and having $S_v\neq 0$ then there exists an interval $[y,z]$ in the direction of $v$ such that $S_u\neq 0$ for any branch $u$ on $I$ towards $z$. Add $f(I)$ to $\Gamma_X$ and $f^{-1}(f(I))$ to $\Gamma_Y$, and repeat this procedure again. In the end, one obtains the minimal skeleton of $f$ that contains the original $\Gamma_X$ (though this may require transfinite induction if the intervals $I$ are chosen too short).
\end{remark}

\section{Coverings of degree $p$}\label{degpsec}

\subsection{Topological ramification locus}

\subsubsection{Radial sets}
Let $Y$ be a nice compact curve, $\Gamma_Y\subseteq Y$ a skeleton of $Y$, $\Gamma\subseteq \Gamma_Y$ a finite subgraph, and $\phi\colon\Gamma\to(0,1]$ a piecewise monomial function. We provide $Y$ with the orientation with respect to $\Gamma_Y$. For a point $x\in \Gamma$ let $C(\Gamma,x,\phi(x))$ denote the union of all closed downward intervals $I$ starting at $x$ such that $l(I) = -\log \phi(x)$. The \emph{radial set} $C(\Gamma,\phi)$ with center at $\Gamma$ of radius $\phi$ is the union of $C(\Gamma,x,\phi(x))$ for all $x\in \Gamma$.

\begin{remark}\label{conerem}
Let $B(\Gamma,\phi)$ be the metric neighborhood of $\Gamma$ given by $\phi$, i.e. $B(\Gamma,\phi)$ is the union of intervals at $x\in\Gamma$ of length $\psi(x)=-\log\phi(x)$. Obviously, $C(\Gamma,\phi)\subseteq B(\Gamma,\phi)$, but the inclusion may be strict. Indeed, assume that $[x,y]$ is an interval in $\Gamma$ and $\psi(x)-\psi(y)>l([x,y])$; for example, $\phi$ is monomial of slope smaller than $-1$ on $[x,y]$. Choose a downward interval $[y,z]$ of length $l$ such that $\psi(y)<l<\psi(x)-l([x,y])$. Then $z\notin C(\Gamma,\phi)$ since $\psi(y)<l$, but $d(z,x)<\psi(x)$ and hence $z\in B(\Gamma,\phi)$. Intuitively, the radial set behaves as a non-convex set in this case.
\end{remark}

\subsubsection{Coverings of degree $p$}

\begin{theorem}\label{coneth}
Assume that $f\colon Y\to X$ is a finite generically \'{e}tale morphism between nice compact curves and $\deg(f)=p=\cha(\tilk)$. Let $(\Gamma_Y,\Gamma_X)$ be a skeleton of $f$ and let $\Gamma\subseteq \Gamma_Y$ be the subgraph consisting of topological ramification points. Then the topological ramification locus $T$ of $f$ coincides with the radial set $C=C(\Gamma,\delta_f^{1/(p-1)})$.
\end{theorem}
\begin{proof}
By Theorem~\ref{deltatrivth}, $\Gamma_Y$ trivializes $\delta_f$. Since $\deg(f)=p$, it follows that for any ramification point $x\in \Gamma$ with $\delta_f(x)<1$ and a closed downward interval $I$ starting at $x$, the restriction of $\delta_f$ on $I$ is monomial with the slope $p-1$. Also, if $\delta_f(x)=1$, then $C(\Gamma,x,\delta_f^{1/(p-1)(x)})=\{x\}$. This shows that $C\subset T$ and we claim that this is, in fact, an equality because $f$ splits outside of $C$.

To prove the claim, choose any connected component $D$ of $Y\setminus C$. It is an open disc with limit point $y$ that lies on the boundary of $C$ and hence satisfies $\delta_f(y)=1$. Note that $D$ is a wide open domain (see \ref{specsec}) and $D_\infty=\{v\}$, where $v$ is the branch at $y$ in the direction of $D$. Recall that $\Gamma_Y$ trivializes $\delta_f$, hence $S_v=0$ and $\slope_v\delta_f=n_v-1$. Since $\delta_f(y)=1$, we necessarily have that $\slope_v\delta_f\le 0$ and hence $n_v=1$. The morphism $D\to f(D)$ is finite of rank $n_v$, hence $D\toisom f(D)$ and the claim is proved.
\end{proof}

\subsection{Double coverings of $\bfP^1_k$ of genus 1}\label{doubelcover}
We would like to finish the paper with illustrating our results on the particular case of a double covering $f\:E\to\bfP^1_k$ with $E$ being an elliptic curve. In the tame case, this is classical, e.g., see \cite[Section~9.7.3]{bgr}, but the description of the wild case is new, to the best of our knowledge.

\subsubsection{The minimal skeleton}
In the sequel, $(\Gamma_E,\Gamma_P)$ denotes the minimal skeleton of $f$, and $\varphi\:\Gamma_E\to\Gamma_P$ is the induced morphism of graphs. By Theorem~\ref{deltatrivth}, $(\varphi,\slope(\delta))$ is a $\delta$-morphism, that will be denoted by $\varphi$ for shortness.

\begin{lemma}\label{specskel}
The $\delta$-morphism $\varphi$ is special (\ref{specdelta}) and the type of $\varphi$ is as follows: tame or mixed if $\cha(k)\neq 2$, wild if $\cha(k)=2$. Moreover, the restriction of $\delta$ onto $\Gamma_E$ induces a special $\delta$-morphism of metric genus graphs.
\end{lemma}
\begin{proof}
Let us check conditions (1)--(5) of \ref{specdelta}. The minimality of the skeleton is equivalent to the stability of $\varphi$, and clearly $\deg(\phi)=2$. In addition, $g(\Gamma_E)=1$ and $g(\Gamma_P)=0$ by \ref{gennicesec}, so $\varphi$ satisfies condition (1). Condition (2) is satisfied by Theorem~\ref{deltatrivth}(iii) because $E$ is proper and so $\partial(E)=\emptyset$. Any ramification point $y\in\Ram(f)$ has multiplicity 2, hence the ramification is tame if and only if $\cha(k)\neq 2$. By Theorem~\ref{limth}, the ramification is tame at $v$ if and only if $\slope_v\delta_f=0$, and so $R_v=1$. This establishes condition (3) and the asserted dichotomy between tame or mixed, and wild cases. Condition (4) from \ref{specdelta} is satisfied in the obvious way, and (5) follows from Remark~\ref{restrictrem} in the case of $m=p=2$.

To prove that the morphism of metric graphs is special we should check two more conditions. In the mixed case, Theorem~\ref{limth} implies that $\delta=|n_e|$ for any tail $e$. If $y\in Y$ has genus 1 and $x=f(y)$ then $\calH(y)/\calH(x)$ is an extension of degree 2, and the residue field extension $\wt{\calH(y)}/\wt{\calH(x)}$ separable, because otherwise it must be purely inseparable and we would have that $g(y)=g(x)=0$. Thus, $\calH(y)/\calH(x)$ is unramified and hence $\delta_f(y)=1$.
\end{proof}

\begin{lemma}\label{ordinarylem}
Keep the above notation and assume that $\cha(\tilk)=2$, $Y$ has good reduction, and $y\in Y$ is the point of genus 1. Then $Y$ has ordinary reduction if and only if the valence of $y$ in $\Gamma_E$ is 2.
\end{lemma}
\begin{proof}
Let $x=f(y)$ and let $\tilf\:\tilE\to\bfP^1_\tilk$ be the morphism of smooth proper $\tilk$-curves associated to the extension $\wt{\calH(y)}/\wt{\calH(x)}$. If $y$ has valence two then there are two ramified branches at $y$ hence the morphism $\tilf$ has two ramification points and so $\tilE$ is ordinary. If the valence of $y$ is one then there is $v\in\Br(y)$ with $\slope_v\delta_f=-3$. It follows easily that the different of $\tilf$ at $v$ is $4$, and hence $\tilE$ is supersingular.
\end{proof}

\subsubsection{The tame and mixed cases}
If $\cha(k)\neq 2$ then the ramification is tame, hence $|\Ram(f)|=4$. Moving three ramification points to $0,1,\infty$ we can achieve that the fourth one is $\lam$ such that $|\lam|\ge 1$ and $|1-\lam|\ge 1$. Since $f$ is Kummer, it is given by the equation $y^2=x(x-1)(x-\lambda)$. Note that the $j$-invariant of $E$ is $j=2^8\frac{(\lam^2-\lam+1)^3}{\lam^2(\lam-1)^2}$ in this case (e.g., \cite[p. 317]{Hartshorne}), and so $|j|=|256|\cdot|\lam|^2$ when $|\lam|>1$.

Let $\Gamma_0$ be the convex hull of $f(\Ram(f))$ in $X$. By Remarks~\ref{specialrem} and \ref{specialmetricrem}, $\Gamma_0$ is either of X-shape or H-shape, and the shape together with the length $l$ of the bar, which equals to $|\lam|$, determines the type completely. In addition, the metric is determined by the formulas $|\lam|=l=4l_1+l_0$ and $l_1+3l_3=-\log|2|$ from Remark~\ref{specialmetricrem} and Theorem~\ref{speciallem}(c).

\subsubsection{The wild case}
Assume, now, that $\cha(k)=2$. We should replace the Weierstrass form with a reasonable non-constant one-parametric family. Perhaps the most natural choice is to take Deuring's normal form: $y^2+\alp xy+y=x^3$. Let $E_\alp$ be the associated curve; its $j$-invariant can be computed by Tate's formulae, see \cite[Section~2]{Tate}. The following modular forms from Tate's list are non-zero for this equation: $a_1=\alp$, $a_3=1$, $b_2=\alp^2$, $c_4=\alp^4$, $c_6=\alp^6$, $\Delta=\alp^3+1$ and $j=\frac{\alp^{12}}{\alp^3+1}$. In particular, the $\alp$-line provides a 12-fold covering of the moduli space of elliptic curves, $E_\alp$ is supersingular if and only if $\alp=0$ and $E_\alp$ is nodal if and only if $\alp\in\{1,\infty\}$. Note also that if $|\alp|\le 1=|\alp+1|$ then $E_\alp$ is a good reduction curve whose genus 1 point sits over the Gauss point of the $x$-line. The reduction curve is given by $\tily^2+\wt{\alp}\tilx\tily+\tily=\tilx^3$, so it is supersingular if and only if $|\alp|<1$.

The metric skeleton is as follows: $\varphi$ is of type (WB) if and only if $E_\alp$ has bad reduction. It is classical that this happens if and only if $|j|>1$, and then $\log|j|$ is the length of the loop (the interested reader can also deduce this directly by analysing the case $|\alp+1|<1$). It follows from Lemma~\ref{ordinarylem} that $E$ has ordinary reduction if and only if $\varphi$ is of type (WO). To distinguish the cases (WS) and (WSS) corresponding to the supersingular reduction we note that $|\Ram(f)|=1$ and so $E$ is supersingular and $j=0$ in the case (WSS), while $|\Ram(f)|=2$ and so $E$ is ordinary and $j\neq 0$ in the case (WS). In the cases (WO) and (WSS), $\Gamma_P$ consists of tails. The metric structure of $\Gamma_P$ in (WS) is determined by the length $l_3$ of the edge $e$ connecting the supersingular point with the path between the ramification points. The double covering $f\:E_\alp\to\bfP^1_k$ of the $x$-plane is ramified over the points $x=\frac{1}{\alp},\infty$, hence the image of $e\subset\Gamma_E\subset E_\alp$ in $\bfP^1_k$ is the interval connecting the Gauss point with the line $[\frac{1}{\alp},\infty]$. Its length equals to $-\log|\alp|$ and hence $l(e)=-\frac{1}{2}\log|\alp|=-\frac{1}{24}\log|j|$.

\subsubsection{The summary}
Using the fact the reduction is good if and only if $|j|\le 1$ and the reduction is supersingular if and only if $|j|<1$ and $\cha(\tilk)=2$, we can summarize our classification of double coverings as follows. The relations between $|j|$ and $|\lam|$ or $|\alp|$ we have observed earlier, are used to express the metric in terms of $|j|$ only.

\begin{theorem}\label{skelelliptic}
The ten non-exceptional special $\delta$-morphisms from Theorem~\ref{specialClassification} are precisely the $\delta$-morphisms that occur as the minimal skeleton $\varphi\:\Gamma_E\to\Gamma_P$ of a double covering $f\colon E\to \bfP^1_k$ with $E$ an elliptic curve. Moreover, a special $\delta$-morphism $\Gamma\to\Gamma'$ of metric genus graphs (with respect to the semivaluation of $\bfZ$ induced from $k$), see Theorem~\ref{speciallem}, lifts to such a double covering if and only if the lengths of the inner edges of $\Gamma$ belong to $|k^\times|$. These cases are characterized as follows:

(i) $\varphi$ is (TB) if and only if $\cha(\tilk)\neq 2$ and $|j|>1$ if and only if $\cha(\tilk)\neq 2$ and $E$ has bad reduction. In this case, $l_0=\frac{1}{2}\log|j|$.

(ii) $\varphi$ is (TG) if and only if $\cha(\tilk)\neq 2$ and $|j|\le 1$ if and only if $\cha(\tilk)\neq 2$ and $E$ has good reduction.

(iii) $\varphi$ is (MB) if and only if $\cha(k)=0$, $\cha(\tilk)=2$ and $|j|>1$ if and only if $\cha(k)=0$, $\cha(\tilk)=2$ and $E$ has bad reduction. In this case, $l_0=\frac{1}{2}\log|j|$ and $l_1=-\log|2|$.

(iv) $\varphi$ is (MO) if and only if $\cha(k)=0$, $\cha(\tilk)=2$ and $|j|=1$ if and only if $\cha(k)=0$, $\cha(\tilk)=2$ and $E$ has ordinary reduction. In this case, $l_1=-\log|2|$.

(v) $\varphi$ is (MS) if and only if $\cha(k)=0$, $\cha(\tilk)=2$ and $|256|<|j|<1$. In this case, the reduction is supersingular, $l_1=\frac{1}{8}\log|j|-\log|2|$ and $l_3=-\frac{1}{24}\log|j|$.

(vi) $\varphi$ is (MSS) if and only if $\cha(k)=0$, $\cha(\tilk)=2$ and $|j|\le|256|$. In this case, the reduction is supersingular and $l_3=-\frac{1}{3}\log|2|$.

(vii) $\varphi$ is (WB) if and only if $\cha(k)=2$ and $|j|>1$ if and only if $\cha(k)=2$ and $E$ has bad reduction. In this case, $l_0=\frac{1}{2}\log|j|$.

(viii) $\varphi$ is (WO) if and only if $\cha(k)=2$ and $|j|=1$ if and only if $\cha(k)=2$ and $E$ has ordinary reduction.

(ix) $\varphi$ is (WS) if and only if $\cha(k)=2$ and $0<|j|<1$ if and only if $\cha(k)=2$, $E$ is ordinary and the reduction is supersingular. In this case, $l_3=-\frac{1}{24}\log|j|$.

(x) $\varphi$ is (WSS) if and only if $\cha(k)=2$ and $j=0$ if and only if $E$ is supersingular.

\end{theorem}

\begin{corollary}
(i) The type of the graph is determined by $|j|$ and the characteristics of $k$ and $\tilk$.

(ii) If one only considers the type of reduction (bad, ordinary, supersingular) instead of $|j|$ then all cases are distinguished except the following two pairs: (MS) versus (MSS), and (WS) versus (WSS). The latter pair is distinguished by the type of $E$ itself.
\end{corollary}

\begin{remark}
Using the notion of canonical subgroups one can also distinguish cases (MS) and (MSS). Recall that if $E$ has ordinary reduction then there is a canonical subgroup $C$ of the 2-torsion group $E[2]$, which lifts the connected component of $\tilE[2]$. Moreover, it is well known that this subgroup extends to some elliptic curves with supersingular reduction. In fact, these are precisely the (MS) curves and one should simply take $C=\{\infty,\lam\}$. For (MSS) curves, any disc in $E$ containing two points of $E[2]$ contains all of $E[2]$.
\end{remark}

\begin{remark}
By Theorem \ref{coneth} the topological ramification locus of $f\:E\to\bfP^1_k$ is the radial set $C(\Gamma_0,\delta_f)$ with center at a subgraph $\Gamma_0$ of $\Gamma_E$ obtained by removing the loop edges. The configuration is supersingular if and only if there is an edge with slope of the different equal to 3. It follows easily from Remark~\ref{conerem} that this happens if and only if $C(\Gamma_0,\delta_f)$ is strictly smaller than the metric neighborhood $B(\Gamma_0,\delta_f)$ of $\Gamma_0$.
\end{remark}

\bibliographystyle{amsalpha}
\bibliography{wild_ramification}

\end{document}